\documentclass[dvipdfmx]{article}

\usepackage[margin=30truemm]{geometry}
\usepackage[all]{xy} 
\usepackage{amsmath,amssymb,amsthm,bm,setspace,mathrsfs}
\numberwithin{equation}{section}
\DeclareMathOperator{\Par}{Par}
\DeclareMathOperator{\Hom}{Hom}
\DeclareMathOperator{\cl}{cl}
\DeclareMathOperator{\wt}{wt}
\DeclareMathOperator{\id}{id}
\DeclareMathOperator{\af}{af}
\DeclareMathOperator{\Ker}{Ker}
\DeclareMathOperator{\aff}{aff}
\DeclareMathOperator{\ad}{ad}

\DeclareMathOperator{\gch}{gch}
\DeclareMathOperator{\re}{re}
\DeclareMathOperator{\SiB}{SiB}
\theoremstyle{definition} 
\newtheorem{theorem}{Theorem}[section]
\newtheorem{proposition}[theorem]{Proposition}
\newtheorem{definition}[theorem]{Definition}
\newtheorem{remark}[theorem]{Remark}

\newtheorem{corollary}[theorem]{Corollary}
\newtheorem{lemma}[theorem]{Lemma}

\title{Branching rules for level-zero extremal weight modules from $U_q(\widehat{\mathfrak{sl}}_{n+1})$ to $U_q(\widehat{\mathfrak{sl}}_n)$}

\author{Shutaro Nakaoka}
\begin{document}
\maketitle

\begin{abstract}
In this paper, we study the structure of a $U_q(\widehat{\mathfrak{sl}}_n)$-module $\Psi_{\varepsilon}^* V(\lambda)$, where $V(\lambda)$ is the extremal weight module of level-zero dominant weight $\lambda$ over the quantum affine algebra $U_q(\widehat{\mathfrak{sl}}_{n+1})$ and $\Psi_{\varepsilon}: U_q(\widehat{\mathfrak{sl}}_n) \to U_q(\widehat{\mathfrak{sl}}_{n+1})$ is an injective  algebra homomorphism. We establish a direct sum decomposition $\Psi_{\varepsilon}^* V(\lambda) \cong M_{0,\varepsilon} \oplus \cdots \oplus M_{m,\varepsilon}$, where $M_{0,\varepsilon}$ and $M_{m,\varepsilon}$ are isomorphic to a tensor product of an extremal weight module over $U_q(\widehat{\mathfrak{sl}}_n)$ and a symmetric Laurent polynomial ring. Moreover, when $\lambda$ is a multiple of a level-zero fundamental weight, we show that $\Psi_{\varepsilon}^* V(\lambda)$ is isomorphic to a direct sum of extremal weight modules.
\end{abstract}

\tableofcontents

\section{Introduction}

Let $\mathfrak{g}$ be a Kac-Moody algebra, and let $U_q(\mathfrak{g})$ denote its associated quantized enveloping algebra, introduced by Drinfeld \cite{Dri} and Jimbo \cite{Jim}. In this paper, we focus on \textit{quantum affine algebras}, that is, we mainly consider the case where $\mathfrak{g}$ is of affine type \cite{Kac}. Quantum affine algebras and their representation theory arise naturally in the study of solvable lattice models in statistical mechanics and quantum field theory (see, for example, \cite[Section 9]{HK}), providing a mathematical framework for understanding symmetries in these systems. They are used, for example, to construct solutions to the Yang-Baxter equation, which is fundamental to solvable lattice models.

 Let $V(\lambda)$ denote the \textit{extremal weight module} with weight $\lambda$ over a quantum affine algebra $U_q(\mathfrak{g})$ introduced by Kashiwara \cite{Kas2}. It has nice properties, including the existence of a global crystal basis. If the level of $\lambda$ is positive (resp. negative), then $V(\lambda)$ is isomorphic to an irreducible highest (resp. lowest) weight module. Our interest lies in the case where $\lambda$ is level-zero. The level-zero extremal weight module $V(\lambda)$ has been studied, with particular focus on the structure of its crystal basis $\mathcal{B}(\lambda)$. In \cite[Section 13]{Kas3}, Kashiwara proposed a conjecture on $V(\lambda)$, which describes the structure of $\mathcal{B}(\lambda)$. This conjecture was proven by Beck-Nakajima \cite{BN}. Ishii-Naito-Sagaki \cite{INS} introduced a combinatorial model for $\mathcal{B}(\lambda)$, known as the \textit{semi-infinite Lakshmibai-Seshadri path model}.

In this paper, we study the structure of the $U_q(\widehat{\mathfrak{sl}}_n)$-module $\Psi_{\varepsilon}^* V(\lambda)$ when $\lambda$ is a level-zero dominant weight, where $\Psi_{\varepsilon}: U_q(\widehat{\mathfrak{sl}}_n) \to U_q(\widehat{\mathfrak{sl}}_{n+1})\;(\varepsilon\in \{\pm1\})$ is an injective  algebra homomorphism. Our main result shows that when $\lambda$ is a multiple of a level-zero fundamental weight, $\Psi_{\varepsilon}^* V(\lambda)$ is isomorphic to a direct sum of extremal weight modules.

Such a result is generally referred to as a \textit{branching rule}. A classical example of a branching rule is the branching rule for  irreducible rational representations $\mathop{GL}_{n+1}(\mathbb{C})$ restricted to $\mathop{GL}_n(\mathbb{C})$ (see, for example, \cite[Section 41]{Bum}). If an embedding of quantized enveloping algebras $U_q(\mathfrak{g}) \hookrightarrow U_q(\mathfrak{g}^{\prime})$ is arising from an embedding of Dynkin diagrams, the branching rule for integrable highest weight representations of $U_q(\mathfrak{g}^{\prime})$ restricted to $U_q(\mathfrak{g})$ is obtained by counting highest weight elements (as a crystal of $U_q(\mathfrak{g})$) in the crystal basis (see \cite[Section 4.6]{Kas5}).

Our result, however, differs from such situations. In fact, the Chevalley generators $E_0, F_0\in U_q(\widehat{\mathfrak{sl}}_n)$ do not correspond to any Chevalley generators in $U_q(\widehat{\mathfrak{sl}}_{n+1})$ via the embedding $\Psi_{\varepsilon}: U_q(\widehat{\mathfrak{sl}}_n) \to U_q(\widehat{\mathfrak{sl}}_{n+1})\;(\varepsilon\in \{\pm1\})$. Our work reveals an unexplored aspect of the structure of the level-zero extremal weight modules, which is not directly based on the properties of the crystal basis.

Let us explain the results in the paper more precisely. We denote the level-zero fundamental weights for $\widehat{\mathfrak{sl}}_n$ by $\varpi_i$ ($i=1,\ldots,n-1$) and those for $\widehat{\mathfrak{sl}}_{n+1}$ by $\breve{\varpi}_i$ ($i=1,\ldots,n$). The set of weights for $\widehat{\mathfrak{sl}}_{n+1}$ of the form $\sum_{i=1}^n m_i \breve{\varpi}_i$ with $m_i \geq 0$ ($i=1,\ldots,n$) is denoted by $\breve{P}_{0,+}$. The coweight lattice for $\widehat{\mathfrak{sl}}_{n+1}$ is denoted by $\breve{P}^*$. Let $\breve{\alpha}_i^\vee$ ($i=0,\ldots,n$) denote the simple coroots for $\widehat{\mathfrak{sl}}_{n+1}$. The generator of the imaginary roots for $\widehat{\mathfrak{sl}}_n$ is denoted by $\delta$.

We fix $\lambda=\sum_{i=1}^n m_i\breve{\varpi}_i\in \breve{P}_{0,+}$ and set $m=m_1+\cdots+m_n$. We define $\tilde{h} \in \breve{P}^*$ by 
\[
\tilde{h}=\sum_{i=1}^n i\breve{\alpha}_i^{\vee}
\]
and set
\[
M_{p,\varepsilon}=\{u\in \Psi_{\varepsilon}^*V(\lambda) \mid q^{\tilde{h}}u=q^{\left\langle \tilde{h},\lambda\right\rangle-(n+1)p}u\}
\]
for $p\in \mathbb{Z}$. Since $q^{\tilde{h}} \in U_q(\widehat{\mathfrak{sl}}_{n+1})$ commutes with all elements in $\Psi_{\varepsilon}(U_q(\widehat{\mathfrak{sl}}_n))$, each $M_{p,\varepsilon}$ is a $U_q(\widehat{\mathfrak{sl}}_n)$-submodule of $\Psi_{\varepsilon}^* V(\lambda)$ for $p \in \mathbb{Z}$. We obtain the following proposition.

\begin{proposition}[= Proposition \ref{dirsum2}]
\textit{There exists a direct sum decomposition
\[
  \Psi_{\varepsilon}^* V(\lambda)\cong M_{0,\varepsilon}\oplus\cdots\oplus M_{m,\varepsilon}
  \]
as a $U_q(\widehat{\mathfrak{sl}}_n)$-module.}
\end{proposition}

Let $\mathbb{Q}(q)[t_1^{\pm1},\ldots,t_m^{\pm1}]$ be the ring of Laurent polynomials. We equip $\mathbb{Q}(q)[t_1^{\pm1},\ldots,t_m^{\pm1}]$ with the $U_q(\widehat{\mathfrak{sl}}_n)$-module structure as follows:
\begin{itemize}
  \item $E_i$ and $F_i$ act trivially;
  \item the monomial $t_1^{k_1}\cdots t_l^{k_m}\;(k_1\ldots,k_m \in \mathbb{Z})$ has weight $(k_1+\cdots+k_l)\delta$.
\end{itemize}

The $m$-th symmetric group $\mathfrak{S}_m$ acts on $\mathbb{Q}(q)[t_1^{\pm1},\ldots,t_m^{\pm1}]$ by permuting variables $t_1,\ldots,t_m$. The set of fixed points by this action is denoted by $\mathbb{Q}(q)[t_1^{\pm1},\ldots,t_m^{\pm1}]^{\mathfrak{S}_m}$.

Let $\breve{W}$ denote the Weyl group for $\widehat{\mathfrak{sl}}_{n+1}$. For $x\in \breve{W}$, let $V_x^-(\lambda)\subset V(\lambda)$ denote the Demazure submodule (for the definition, see Section \ref{dema}). We describe the graded character of $(M_{p,\varepsilon}\cap V_e^-(\lambda))$ for $0\le p\le m$ in terms of Macdonald polynomials. Using this, we obtain the following result.

\begin{theorem}[=Theorem \ref{M_0}+Theorem \ref{M_m}]
  \textit
  {There are isomorphisms of $U_q(\widehat{\mathfrak{sl}}_n)$-modules
  \begin{gather*}
    M_{0,\varepsilon}\cong V\left(\sum_{i=1}^{n-1}m_i\varpi_i\right)\otimes\left(\mathbb{Q}(q)[t_1^{\pm1},\ldots,t_{m_n}^{\pm1}]^{\mathfrak{S}_{m_n}}\right),\\
    M_{m,\varepsilon}\cong \left(\mathbb{Q}(q)[t_1^{\pm1},\ldots,t_{m_1}^{\pm1}]^{\mathfrak{S}_{m_1}}\right)\otimes V\left(\sum_{i=1}^{n-1}m_{i+1}\varpi_i\right).
  \end{gather*}}
\end{theorem}

In the case where $\lambda=m\breve{\varpi}_i\;(i=1,\ldots,n)$, we prove that $M_{1,\varepsilon},\ldots, M_{m-1,\varepsilon}$ are also isomorphic to direct sums of extremal weight modules. This leads to the following result.

\begin{theorem}[=Corollary \ref{2in}+Corollary \ref{i1}+Corollary \ref{in}]
\textit{
We have
\[
  \Psi_{\varepsilon}^*V(m\breve{\varpi}_i)\cong\left\{
\begin{array}{ll}
\bigoplus_{p=0}^m \mathbb{Q}(q)[t_1^{\pm1},\ldots,t_p^{\pm1}]^{\mathfrak{S}_p}\otimes V((m-p)\varpi_1) & (i=1) \\
\bigoplus_{p=0}^mV(p\varpi_{i-1}+(m-p)\varpi_i) & (2\le p\le n-1)\\
\bigoplus_{p=0}^m \mathbb{Q}(q)[t_1^{\pm1},\ldots,t_{m-p}^{\pm1}]^{\mathfrak{S}_{m-p}}\otimes V(p\varpi_{n-1}) & (i=n)
\end{array}
\right..
\]
}
\end{theorem}

The organization of this paper is as follows: In Section 2, we provide a review of quantized enveloping algebras, crystal bases, and extremal weight modules. In Section 3, we recall key results on level-zero extremal weight modules over quantum affine algebras and their Demazure submodules. In Section 4, we present an explicit embedding $\Psi_{\varepsilon}:U_q(\widehat{\mathfrak{sl}}_n) \to U_q(\widehat{\mathfrak{sl}}_{n+1})$ for $\varepsilon \in \{\pm1\}$ and study the structure of $\Psi_{\varepsilon}^*V(\breve{\varpi}_i)$ for $i = 1, \ldots, n$. We also establish a direct sum decomposition $\Psi_{\varepsilon}^* V(\lambda) \cong M_{0,\varepsilon} \oplus \cdots \oplus M_{m,\varepsilon}$ and prove that $M_{0,\varepsilon}$ and $M_{m,\varepsilon}$ are isomorphic to a tensor product of an extremal weight module over $U_q(\widehat{\mathfrak{sl}}_n)$ and a symmetric Laurent polynomial ring. In Section 5, we focus on the special case where $\lambda$ is a multiple of a level-zero fundamental weight and prove our main result: $\Psi_{\varepsilon}^* V(\lambda)$ is isomorphic to a direct sum of extremal weight modules.

\section{Quantized enveloping algebras and extremal weight modules}

\subsection{Quantized enveloping algebras}

We fix a root datum $(P, \Sigma, P^*, \Sigma^{\vee})$, where $P$ is the weight lattice, $P^*$ is the coweight lattice, $\Sigma=\{\alpha_i \mid i\in I\}\subset P$ is the set of simple roots indexed by $I$, and $\Sigma^{\vee}=\{\alpha_i^{\vee} \mid i\in I\}$ is the set of simple coroots. We denote by $\langle \cdot, \cdot \rangle : P\times P^*\to \mathbb{Z}$ the natural paring. We also fix a $\mathbb{Q}$-valued symmetric bilinear form $(\cdot, \cdot) : P\times P\to \mathbb{Q}$.

We assume that they satisfy the following conditions:

\begin{enumerate}
  \item $(\alpha_i, \alpha_i) > 0$ for all $i\in I$;
  \item $\langle \alpha_i^{\vee}, \lambda\rangle =\frac{2(\alpha_i, \lambda)}{(\alpha_i, \alpha_i)}$ for all $i\in I$ and $\lambda\in P$;
  \item $(\langle \alpha_i^{\vee}, \alpha_j\rangle)_{i,j\in I}$ is a symmetrizable generalized Cartan matrix. 
\end{enumerate}

We set $a_{ij}=\langle \alpha_i^{\vee}, \alpha_j\rangle$ for $i,j\in I$ and let $\mathfrak{g}$ be the Kac-Moody algebra over $\mathbb{Q}$ associated with $A=(a_{ij})_{i,j\in I}$. We denote the sets of positive roots, and negative roots by $\Delta, \Delta_+$, and $\Delta_-$, respectively.

Take an integer $D>0$ such that $\frac{(\alpha_i,\alpha_i)}{2}\in D^{-1}\mathbb{Z}$. Let $q$ be an indeterminate and set $q_s=q^{\frac{1}{D}}$. We set $q_i=q^{\frac{(\alpha_i,\alpha_i)}{2}}$. For an integer $m\ge 0$, we define $[m]_{q_i}=\frac{{q_i}^m-{q_i}^{-m}}{q_i-{q_i}^{-1}}$. Also, we set $[m]_{q_i}!=[m]_{q_i}[m-1]_{q_i}\cdots[1]_{q_i}$ for $m>0$ and $[0]_{q_i}!=1$. We define the quantized universal enveloping algebra $U_q(\mathfrak{g})$ of $\mathfrak{g}$ to be the associative algebra over $\mathbb{Q}(q_s)$ with $1$ generated by $E_i, F_i, q^h\;(i\in I, h\in D^{-1}P^*)$, with defining relations
\begin{gather*}
  q^0=1,\quad q^{h+h^{\prime}}=q^hq^{h^{\prime}}\quad(h,h^{\prime}\in D^{-1}P^*), \\
  q^hE_iq^{-h}=q^{\langle h,\alpha_i\rangle}E_i,\quad q^hF_iq^{-h}=q^{-\langle h,\alpha_i\rangle}F_i\quad (i\in I,h\in D^{-1}P^*), \\
  E_iF_j-F_jE_i=\delta_{ij}\frac{t_i-t_i^{-1}}{q_i-q_i^{-1}}\quad (i,j\in I), \\
  \sum_{s=0}^{1-a_{ij}}(-1)^sE_i^{(s)}E_jE_i^{(1-a_{ij}-s)}=\sum_{s=0}^{1-a_{ij}}(-1)^sF_i^{(s)}F_jF_i^{(1-a_{ij}-s)}=0\quad(i,j\in I,i\neq j)
\end{gather*}
where $t_i=q^{\frac{(\alpha_i,\alpha_i)}{2}\alpha_i^{\vee}},\;E_i^{(m)}=\frac{E_i^m}{[m]_{q_i}!}$, and $F_i^{(m)}=\frac{F_i^m}{[m]_{q_i}!}$ for $i\in I$.

We define the coproduct $\Delta$ on $U_q(\mathfrak{g})$ by
\begin{equation}
  \Delta(q^h)=q^h\otimes q^h,\quad\Delta(E_i)=E_i\otimes t_i^{-1}+1\otimes E_i,\quad\Delta(F_i)=F_i\otimes1+t_i\otimes F_i, \label{cop}
\end{equation}
where $h\in D^{-1}P^*$ and $i\in I$.

We define a $\mathbb{Q}$-algebra automorphism $-$, a $\mathbb{Q}(q_s)$-algebra automorphism $\vee$, and a $\mathbb{Q}$-algebra antiautomorphism $\Omega$ of $U_q(\mathfrak{g})$ by
\begin{gather*}
   \overline{E_i}=E_i,\quad\overline{F_i}=F_i,\quad\overline{q^h}=q^{-h},\quad\overline{q_s}={q_s^{-1}},\\
  E_i^{\vee}=F_i,\quad F_i^{\vee}=E_i,\quad (q^h)^{\vee}=q^{-h},\\
 \Omega(E_i)=F_i,\quad\Omega(F_i)=E_i\quad(i\in I),\quad\Omega(q^h)=q^{-h}\quad\Omega(q_s)=q_s^{-1},
\end{gather*}
where $i\in I, h\in D^{-1}P^*$.

We also define a $\mathbb{Q}$-algebra automorphism $\Phi$ of $U_q(\mathfrak{g})$ by $\vee\circ -$.

Let $U_q^+(\mathfrak{g})$ (resp. $U_q^-(\mathfrak{g})$) be the subalgebra of $U_q(\mathfrak{g})$ generated by $\{E_i \mid i\in I\}$ (resp. $\{F_i \mid i\in I\}$). Let $U_q^0(\mathfrak{g})$ be the subalgebra of $U_q(\mathfrak{g})$ generated by $q^h\;(h\in D^{-1}P^*)$. We have the triangular decomposition $U_q(\mathfrak{g})\cong U_q^+(\mathfrak{g})\otimes U_q^0(\mathfrak{g})\otimes U_q^-(\mathfrak{g})$.

A $U_q({\mathfrak{g}})$-module $M$ is called \textit{integrable} if it satisfies the following conditions:
\begin{enumerate}
  \item $M$ admits a \textit{weight space decomposition}
\[
M=\bigoplus_{\lambda \in P} M_{\lambda},\text{ where }M_{\lambda}=\{u\in M \mid q^hu=q^{\langle h, \lambda\rangle}u\; \text{ for all }\;\textstyle h\in D^{-1}P^*\};
\]
\item For any $u\in M$ and $i\in I$, there exists $N\geq 1$ such that $E_i^{(m)}u=F_i^{(m)}u=0$ for all $m\geq N.$
\end{enumerate}

If a $U_q(\mathfrak{g})$-module $M$ has a weight space decomposition $M=\bigoplus_{\lambda \in P} M_{\lambda}$ and $u\in M_{\lambda}$, we denote the weight of $u$ by $\wt(u)=\lambda$. If $M_1$ and $M_2$ are integrable $U_q(\mathfrak{g})$-modules, then $M_1\otimes M_2$ is also integrable.

 Let $W=\langle s_i \mid i\in I\rangle$ be the Weyl group, where $s_i(\lambda)=\lambda-\langle \alpha_i^{\vee}, \lambda\rangle\alpha_i\;(\lambda\in P)$ is the simple reflection for $\alpha_i$. For each $i\in I$ and for any integrable $U_q(\mathfrak{g})$-module $M$, there exists a $\mathbb{Q}(q_s)$-linear automorphism $T_i:M\to M$ (see \cite[Chapter 5]{L}, where it is denoted by $T_{i,1}^{\prime\prime}$). Also, for each $i\in I$, there exists an automorphism $T_i:U_q(\mathfrak{g})\to U_q(\mathfrak{g})$ such that for any integrable $U_q(\mathfrak{g})$-module $M$, we have 
 \[
 T_x(xu)=T_x(x)T_x(u)
 \]
 for all $x\in U_q(\mathfrak{g})$ and $u\in M$ (see \cite[Chapter 37]{L}, where it is denoted by $T_{i,1}^{\prime\prime}$). For $x\in W$, we take a reduced expression $s_{i_1}\cdots s_{i_p}$ of $x$ and define an automorphism $T_x:U_q(\mathfrak{g})\to U_q(\mathfrak{g})$ by
 \[
 T_x=T_{i_1}T_{i_2}\cdots T_{i_p}.
 \]
 This definition is independent of the choice of the reduced expression of $x$.

Let $M$ be an integrable $U_q(\mathfrak{g})$-module, and let $u \in M$ be a weight vector of weight $\lambda$. By \cite[Proposition 5.2.2]{L}, if $E_i u = 0$, then we have
\begin{gather}
  T_i(u)=(-q_i)^{\langle \alpha_i^{\vee}, \lambda \rangle}F_i^{(\langle \alpha_i^{\vee}, \lambda \rangle)}u,\quad T_i^{-1}(u)=F_i^{(\langle \alpha_i^{\vee}, \lambda \rangle)}u. \label{T1}
\end{gather}
Similarly, if $F_i u = 0$, then we have
\begin{gather}
  T_i(u)=E_i^{(-\langle \alpha_i^{\vee}, \lambda \rangle)}u,\quad T_i^{-1}(u)=(-q_i)^{-\langle \alpha_i^{\vee}, \lambda \rangle}E_i^{(-\langle \alpha_i^{\vee}, \lambda \rangle)}u. \label{T2}
\end{gather}

By \cite[Chapter 37]{L}, we have
\begin{gather*}
  T_i(q^h)=q^{s_i(h)}\quad(h\in D^{-1}P^*),\quad T_i(E_i)=-F_it_i,\quad T_i(F_i)=-t_i^{-1}E_i,\\
  T_i(E_j)=\sum_{s=0}^{-a_{ij}}(-1)^sq_i^{-s}E_i^{(-a_{ij}-s)}E_jE_i^{(s)}\quad(j\neq i),\\
  T_i(F_j)=\sum_{s=0}^{-a_{ij}}(-1)^sq_i^sF_i^{(s)}F_jF_i^{(-a_{ij}-s)}\quad(j\neq i),\\
  T_i^{-1}(q^h)=q^{s_i(h)}\quad(h\in D^{-1}P^*),\quad T_i^{-1}{E_i}=-t_i^{-1}F_i,\quad T_i^{-1}(F_i)=-E_it_i,\\
  T_i^{-1}(E_j)=\sum_{s=0}^{-a_{ij}}(-1)^{-a_{ij}-s}q_i^{a_{ij}+s}E_i^{(-a_{ij}-s)}E_jE_i^{(s)}\quad (j\neq i),\\
  T_i^{-1}(F_j)=\sum_{s=0}^{-a_{ij}}(-1)^{-a_{ij}-s}q_i^{-a_{ij}-s}F_i^{(s)}F_jF_i^{(-a_{ij}-s)}\quad (j\neq i)
\end{gather*}
for $i\in I$. By \cite[Section 1.3]{B}, we have $\Omega\circ T_i=T_i\circ \Omega$ and $\Phi\circ T_i=T_i^{-1}\circ \Phi$ for $i\in I$.

\subsection{Crystal bases}

If $M$ is an integrable $U_q(\mathfrak{g})$-module, then we have
\[
M=\bigoplus_{0\leq n\leq \langle h_i,\lambda \rangle} F_i^{(n)} (\Ker E_i\cap M_{\lambda})
\]
by the theory of integrable representation of $U_q(\mathfrak{sl}_2)$.

Then, we define the map $\tilde{e}_i,\;\tilde{f}_i : M\to M$ by
\[
\tilde{f}_i(u)=\sum_{n\geq 0}F_i^{(n+1)}u_n,\quad\tilde{e}_i(u)=\sum_{n>0}F_i^{(n-1)}u_n
\]
where $u=\sum_{n\geq 0}F_i^{(n)}u_n\in M\;(u_n\in \Ker E_i)$.

We define the subrings $\mathbb{A}_0, \mathbb{A}_{\infty}, \mathcal{A}$ of $\mathbb{Q}(q_s)$ as follows:

\[
  \mathbb{A}_0=\{f/g \mid f,g\in \mathbb{Q}[q_s], g(0)\neq 0\},\quad
  \mathbb{A}_{\infty}=\{f/g \mid f,g\in \mathbb{Q}[q_s^{-1}], g(0)\neq 0\}, \quad
  \mathcal{A}=\mathbb{Q}[q_s, q_s^{-1}].
\]

\begin{definition}
  Let $M$ be an integrable $U_q(\mathfrak{g})$-module. A pair $(\mathcal{L}, \mathcal{B})$ is called a \textit{crystal basis} of $M$ if
  \begin{enumerate}
    \item $\mathcal{L}$ is a free $\mathbb{A}_0$-submodule of $M$ such that $M\cong \mathbb{Q}(q_s)\otimes_{\mathbb{A}_0}\mathcal{L}$;
    \item $\mathcal{L}=\bigoplus_{\lambda\in P}\mathcal{L}_{\lambda}$, where $\mathcal{L}_{\lambda}=\mathcal{L}\cap M_{\lambda}$;
    \item $\tilde{e}_i\mathcal{L}\subset \mathcal{L},\;\tilde{f}_i\mathcal{L}\subset \mathcal{L}$ for all $i\in I$;
    \item $\mathcal{B}=\sqcup_{\lambda\in P}\mathcal{B}_{\lambda}$ is a $\mathbb{Q}$-basis of $\mathcal{L}/q\mathcal{L}\cong \mathbb{Q}\otimes_{\mathbb{A}_0}M$, where $\mathcal{B}_{\lambda}=(\mathcal{L}_{\lambda}/q_s\mathcal{L}_{\lambda})\cap \mathcal{B}$;
    \item the operators $\tilde{e}_i,\;\tilde{f}_i:\mathcal{L}/q\mathcal{L}\to \mathcal{L}/q\mathcal{L}$ induced by $\tilde{e}_i, \tilde{f}_i$ satisfy $\tilde{e}_i\mathcal{B}\subset\mathcal{B}\sqcup\{0\},\; \tilde{f}_i\mathcal{B}\subset\mathcal{B}\sqcup\{0\}$;
    \item for any $b, b^{\prime}\in \mathcal{B}$, we have $\tilde{e}_ib=b^{\prime}$ if and only if $\tilde{f}_ib^{\prime}=b$.
  \end{enumerate}
\end{definition}

If $b\in \mathcal{B}_{\lambda}$, then we write $\wt(b)=\lambda$. For $i\in I$, we define the maps \[\varepsilon_i,\;\varphi_i:\mathcal{B}\to \mathbb{Z}_{\geq0}
\]by
\[
\varepsilon_i(b)=\max\{n \mid \tilde{e}_i^nb\neq 0\}\quad\text{and}\quad\varphi_i(b)=\max\{n \mid \tilde{f}_i^nb\neq 0\}.
\]

By abstracting the combinatorial properties of crystal basis, we define \textit{crystals}.

\begin{definition}
  A crystal is a set $\mathcal{B}$ together with a collection of maps
  \[
    \wt:\mathcal{B}\to P,\quad
    \varepsilon_i,\;\varphi_i:\mathcal{B}\to \mathbb{Z}\sqcup\{-\infty\}\quad(i\in I),\quad
    \tilde{e}_i,\;\tilde{f}_i:\mathcal{B}\to\mathcal{B}\sqcup\{0\}\quad(i\in I)
  \]
  satisfying the  following properties:
  \begin{enumerate}
    \item for all $b\in \mathcal{B}$, we have $\varphi_i(b)=\varepsilon_i(b)+\langle \alpha_i^{\vee}, \wt(b)\rangle$;
    \item for all $b\in \mathcal{B}$ with $\tilde{e}_ib\in \mathcal{B}$, we have $\wt(\tilde{e}_ib)=\wt(b)+\alpha_i$;
    \item for all $b\in \mathcal{B}$ with $\tilde{f}_ib\in \mathcal{B}$, we have $\wt(\tilde{f}_ib)=\wt(b)-\alpha_i$;
    \item for any $b, b^{\prime}\in \mathcal{B}$, we have $\tilde{e}_ib=b^{\prime}$ if and only if $\tilde{f}_ib^{\prime}=b$;
    \item if $\varepsilon_i(b)=-\infty$, then $\tilde{e}_ib=\tilde{f}_ib=0$.
  \end{enumerate}
\end{definition}

\begin{definition}
  A morphism of crystal $\psi:\mathcal{B}_1\to \mathcal{B}_2$ is a map $\psi:\mathcal{B}_1\sqcup\{0\}\to \mathcal{B}_2\sqcup\{0\}$ such that
  \begin{enumerate}
    \item $\psi(0)=0$;
    \item for all $b\in B_1$, we have $\psi(\wt(b))=\wt(\psi(b)),\;\psi(\varepsilon_i(b))=\varepsilon_i(\psi(b)),\;\psi(\varphi(b))=\varphi(\psi(b))$;
    \item if $b\in B_1$ satisfies $\tilde{e}_ib\neq0$ and $\psi(b)\neq0$, then we have $\psi(\tilde{e}_ib)=\tilde{e}_i\psi(b)$;
    \item if $b\in B_1$ satisfies $\tilde{f}_ib\neq0$ and $\psi(b)\neq0$, then we have $\psi(\tilde{f}_ib)=\tilde{f}_i\psi(b)$.
  \end{enumerate}
\end{definition}

Let $\mathcal{C}(I, P)$ denote the category of crystals.

\subsection{Global bases}

Let $V$ be a vector space over $\mathbb{Q}(q_s)$. For a subring of $A$ of $\mathbb{Q}(q_s)$, a $A$-lattice is a $A$-submodule $M$ of $V$ such that $\mathbb{Q}(q_s)\otimes_{A}M\cong V$. Let $V_{\mathcal{A}}$ be a $\mathcal{A}$-lattice of $V$, $\mathcal{L}_0$ a $\mathbb{A}_0$-lattice, $\mathcal{L}_{\infty}$ a $\mathbb{A}_{\infty}$-lattice. Then we have the following lemma.

\begin{lemma}[{\cite[Lemma 2.1.1]{Kas1}}]
  \textit{Let $E=V_{\mathcal{A}}\cap \mathcal{L}_0\cap \mathcal{L}_{\infty}$. Then the following conditions are equivalent.
  \begin{enumerate}
    \item The canonical map $E \to \mathcal{L}_0/q_s\mathcal{L}_0$ is an isomorphism;
    \item The canonical map $E\to \mathcal{L}_{\infty}/q_s^{-1}\mathcal{L}_{\infty}$ is an isomorphism;
    \item The canonical maps $\mathcal{A}\otimes_{\mathbb{Q}}E\to V_{\mathcal{A}}, \;\mathbb{A}_0\otimes_{\mathbb{Q}}E\to \mathcal{L}_0, \;\mathbb{A}_{\infty}\otimes_{\mathbb{Q}}E\to \mathcal{L}_{\infty},\; \mathbb{Q}(q_s)\otimes_{\mathbb{Q}}E\to V$ are isomorphisms.\qed
    \end{enumerate}
    }
\end{lemma}

The triple $(\mathcal{L}_0, \mathcal{L}_{\infty}, V_{\mathcal{A}})$ is called \textit{balanced} if it satisfies the equivalent conditions above.

Let $M$ be an integrable $U_q(\mathfrak{g})$-module with a crystal base $(\mathcal{L}, \mathcal{B})$. Take an involution $-$ of $M$ such that $\overline{xu}=\overline{x}\,\overline{u}$ for all $x\in U_q(\mathfrak{g}), u\in M$. Let $M_{\mathcal{A}}$ be a $\mathcal{A}$-submodule of $M$ such that $\overline{M_{\mathcal{A}}}=M_{\mathcal{A}}$ and $u-\overline{u}\in (q_s-1)M_{\mathcal{A}}$ for all $u\in M_{\mathcal{A}}$. We say that $M$ has a \textit{global base} $(\mathcal{L}, \mathcal{B}, M_{\mathcal{A}}, -)$ if $(\mathcal{L}, \overline{\mathcal{L}}, M_{\mathcal{A}})$ is balanced.

Let $G$ be the inverse map of the canonical map $M_{\mathcal{A}}\cap \mathcal{L}\cap \overline{\mathcal{L}}\to  \mathcal{L}/q_s\mathcal{L}$. Then, $\{G(b) \mid b\in \mathcal{B}\}$ is a basis of $M$. It is called a \textit{global crystal basis} of $M$.

\subsection{Extremal weight representations}

\begin{definition}
  Let $M$ be an integrable $U_q(\mathfrak{g})$-module and let $i\in I$. A weight vector $u$ of weight $\lambda$ is called $i$-extremal if $E_iu=0$ or $F_iu=0$. In this case, we set $S_iu=F_i^{(\langle \alpha_i^{\vee}, \lambda \rangle)}u$ or $S_iu=E_i^{(-\langle \alpha_i^{\vee}, \lambda \rangle)}u$, respectively.
\end{definition}

\begin{definition}
  Keep the setting of the definition above. A weight vector $u$ is called \textit{extremal} if $S_{i_1}\cdots S_{i_p}u$ is $i$-extremal for all $i, i_1,\ldots, i_p\in I$.
\end{definition}

For $x=s_{i_1}\cdots s_{i_p}\in W$ and an extremal vector $u\in M$, we define $S_xu$ by setting
\[
S_xu=S_{i_1}\cdots S_{i_p}u.
\]

This is well-defined, i.e., $S_xu$ depends only on $x$ (\cite[Section 2.5]{N}).

The following lemmas are used later.

\begin{lemma}[{\label{st}\cite[Lemma 2.11]{N}}]
  \textit{Suppose that $M$ is an integrable $U_q(\mathfrak{g})$-module and $u\in M$ is an extremal vector of weight $\lambda$. Then we have
  \[
  S_xu=(-1)^{N_+^{\vee}}q^{-N_+}T_xu\quad (x\in W),
  \]
  where $\displaystyle N_+=\sum_{\alpha\in \Delta^+\cap x^{-1}(\Delta^-)}\max((\alpha,\lambda),0)$ and $\displaystyle N_+^{\vee}=\sum_{\alpha\in \Delta^+\cap x^{-1}(\Delta^-)}\max(\langle \alpha^{\vee},\lambda\rangle ,0)$. \qed}
\end{lemma}

By (\ref{T1}) and (\ref{T2}), a weight vector $u\in M$ of weight $\lambda$ is extremal if and only if 
\[
\left\{
\begin{array}{ll}
E_iT_xu=0 & (\langle \alpha_i^{\vee},x\lambda\rangle\ge 0), \\
F_iT_xu=0 & (\langle \alpha_i^{\vee},x\lambda\rangle\le 0)
\end{array}
\right.
\]
for all $i\in I$ and $x\in W$.

\begin{lemma}\label{ext}
  \textit{Let $M_1$ and $M_2$ be integrable $U_q(\mathfrak{g})$-modules. Let $u_1\in M_1$ and $u_2\in M_2$ be extremal vectors with weights $\lambda_1$ and $\lambda_2$, respectively. Assume that $\lambda_1$ and $\lambda_2$ are in the same Weyl chamber; in other words,
  \[
  \langle \alpha_i^{\vee}, x\lambda_1\rangle \geq 0 \Leftrightarrow \langle \alpha_i^{\vee}, x\lambda_2 \rangle \geq 0\; \text{ for all } \; i\in I,\;x\in W.
  \]
  Then we have:
  \begin{enumerate}
    \item $T_x(u_1\otimes u_2)=T_x(u_1)\otimes T_x(u_2)$;
    \item $u_1\otimes u_2 \in M_1\otimes M_2$ is extremal;
    \item $S_x(u_1\otimes u_2)=S_x(u_1)\otimes S_x(u_2)$.
  \end{enumerate}
  }
\end{lemma}

\begin{proof}
  First, we prove that $T_x(u_1\otimes u_2)=T_x(u_1)\otimes T_x(u_2)$. For $i\in I$, we define a linear map $L_i:M_1\otimes M_2\to M_1\otimes M_2$ by
  \[
  L_i(v_1\otimes v_2)=\sum_{n\ge 0}(-1)^nq_i^{-\frac{n(n-1)}{2}}\left(\prod_{a=1}^n(q_i^a-q_i^{-a})\right)F_i^{(n)}v_1\otimes E_i^{(n)}v_2.
  \]

  Since $u_1\in M_1$ and $u_2\in M_2$ are extremal vectors with weights in the same Weyl chamber, we have $F_i^{(n)}u_1\otimes E_i^{(n)}u_2=0$ for all $n>0$. It follows that $L_i(u_1\otimes u_2)=u_1\otimes u_2$. By \cite[Proposition 5.3.4]{L}, we get $T_i(u_1\otimes u_2)=T_i(u_1)\otimes T_i(u_2)$. Using induction on $\ell(x)$, we can then prove that $T_x(u_1\otimes u_2)=T_x(u_1)\otimes T_x(u_2)$ for all $x\in W$.

  Next, we prove that $u_1\otimes u_2 \in M_1\otimes M_2$ is extremal. If $\langle \alpha_i^{\vee},x(\lambda_1+\lambda_2)\rangle \ge 0$, then we have $\langle \alpha_i^{\vee},x(\lambda_1)\rangle \ge 0$ and $\langle \alpha_i^{\vee},x(\lambda_2)\rangle \ge 0$ since $\lambda_1$ and $\lambda_2$ are in the same Weyl chamber. Hence, we obtain $E_i(T_x(u_1\otimes u_2))=E_i(T_x(u_1)\otimes T_x(u_2))=0$. In case $\langle \alpha_i^{\vee},x(\lambda_1+\lambda_2)\rangle \le 0$, we can similarly show that $F_i(T_x(u_1\otimes u_2))=0$.

  Finally, we prove that $S_x(u_1\otimes u_2)=S_x(u_1)\otimes S_x(u_2)$. We set
  \begin{gather*}
    N_{1,+}=\sum_{\alpha\in \Delta^+\cap x^{-1}(\Delta^-)}\max((\alpha,\lambda_1),0),\quad N_{1,+}^{\vee}=\sum_{\alpha\in \Delta^+\cap x^{-1}(\Delta^-)}\max(\langle \alpha^{\vee},\lambda_1\rangle ,0),\\
    N_{2,+}=\sum_{\alpha\in \Delta^+\cap x^{-1}(\Delta^-)}\max((\alpha,\lambda_2),0),\quad N_{2,+}^{\vee}=\sum_{\alpha\in \Delta^+\cap x^{-1}(\Delta^-)}\max(\langle \alpha^{\vee},\lambda_2\rangle ,0).
  \end{gather*}
  Since $\lambda_1$ and $\lambda_2$ are in the same Weyl chamber, we have 
  \[
  (\alpha,\lambda_1)\geq 0 \Leftrightarrow (\alpha,\lambda_2)\geq 0\quad\text{and}\quad\langle\alpha^{\vee},\lambda_1\rangle\geq 0 \Leftrightarrow \langle\alpha^{\vee},\lambda_2\rangle\geq0.
  \]
  Hence, we obtain
\begin{align*}
    S_x(u_1\otimes u_2)&=(-1)^{N_{1,+}+N_{2,+}}q^{-(N_{1,+}^{\vee}+N_{1,+}^{\vee})}T_x(u_1\otimes u_2)\\
    &=(-1)^{N_{1,+}+N_{2,+}}q^{-(N_{1,+}^{\vee}+N_{2,+}^{\vee})}T_x(u_1)\otimes T_x(u_2)=S_x(u_1)\otimes S_x(u_2)
  \end{align*}
  by Lemma \ref{st}.
\end{proof}

For $\lambda\in P$, let us denote by $V(\lambda)$ the $U_q(\mathfrak{g})$-module generated by $u_{\lambda}$ with the defining relation that $u_{\lambda}$ is an extremal vector of weight $\lambda$ (denoted by $V^{\max}(\lambda)$ in \cite{Kas2}). By \cite[Proposition 8.2.2]{Kas2}, $V(\lambda)$ has a crystal basis $(\mathcal{L}(\lambda), \mathcal{B}(\lambda))$ together with a global crystal base $\{G(b) \mid b\in \mathcal{B}(\lambda)\}$. We denote by the same letter $u_{\lambda}$ the element of $\mathcal{B}(\lambda)$ corresponding to $u_{\lambda}\in V(\lambda)$.

  For $x\in W$ and $\lambda\in P$, $u_{\lambda} \mapsto S_{x^{-1}}u_{x\lambda}$ gives an isomorphism of $U_q(\mathfrak{g})$-modules:
\[
V(\lambda)\stackrel{\sim}{\longrightarrow} V(x\lambda).
\]

If $\lambda$ is dominant (resp. anti-dominant), then $V(\lambda)$ is the irreducible highest (resp. lowest) weight module with highest (resp. lowest) weight $\lambda$.

\section{Level-zero extremal weight modules over quantum affine algebras}

Keep the setting of the previous section.

\subsection{Level-zero extremal weight modules over quantum affine algebras}

In the remainder of this paper, we assume that $A=(a_{ij})_{i,j\in I}\;(I=\{0,1,\ldots, n\})$ is the generalized Cartan matrix of untwisted affine type where the numbering of the simple roots is the one given in \cite[Section 4.8]{Kac}. Let $\mathfrak{g}$ be the affine Lie algebra over $\mathbb{Q}$ associated with $A$, with Cartan subalgebra $\mathfrak{h}$. Fix $d\in\mathfrak{h}$ such that $\langle d,\alpha_j\rangle =\delta_{0j}$ for all $j\in I$. Let $c=\sum_{i\in I}a_i^{\vee}\alpha_i^{\vee}$ be the canonical central element and $\delta=\sum_{i\in I}a_i\alpha_i$ be the generator of imaginary roots where $a_i^{\vee}$ and $a_i$ are the numerical labels defined in \cite[Section 4.8 and Section 6.1]{Kac}. For $j\in I$, define the fundamental weight the $\Lambda_j\in \mathfrak{h}^*$ such that $\langle \alpha_i^{\vee},\Lambda_j\rangle=\delta_{ij}\;(i\in I)$ and $ \langle d,\Lambda_j\rangle=0$. We take an integral weight lattice $P$ as $P=\left(\bigoplus_{i\in I}\mathbb{Z}\Lambda_i\right)\oplus\mathbb{Z}\delta$ and set $Q=\bigoplus_{i\in I}\mathbb{Z}\alpha_i\subset P$. We set $P^*=\Hom_{\mathbb{Z}}(P,\mathbb{Z})=\left(\bigoplus_{i\in I}\mathbb{Z}\alpha_i^{\vee}\right)\oplus \mathbb{Z}d$. The invariant symmetric bilinear form on $\mathfrak{g}$ is assumed to be normalized such that $\langle c,\lambda\rangle=(\delta,\lambda)$ for all $\lambda\in \mathfrak{h}^*$. Then we have $(\alpha_i, \alpha_j)=a_i^{\vee}a_i^{-1}a_{ij}\;(i,j\in I)$.

We set $\mathfrak{h}_{\cl}^*=\mathfrak{h}^*/\mathbb{Q}\delta$ and let $\cl:\mathfrak{h}^*\to \mathfrak{h}_{\cl}^*$ be the projection. We also set $P_{\cl}=\cl(P)\subset \mathfrak{h}_{\cl}^*$. Then we have $P_{\cl}=\bigoplus_{i\in I}\mathbb{Z}\cl(\Lambda_i)$. Let $U_q^{\prime}(\mathfrak{g})$ denote the quantized universal enveloping algebra with $P_{\cl}$ as a weight lattice. Hence, $U_q^{\prime}(\mathfrak{g})$ is the subalgebra of $U_q(\mathfrak{g})$ generated by $E_i, F_i\;(i\in I), q^{h}\;(h\in D^{-1}(P_{\cl})^*)$.

Let $M$ be a $U_q^{\prime}(\mathfrak{g})$-module with the weight space decomposition $M=\bigoplus_{\lambda\in P_{\cl}} M_{\lambda}$. Then, we define the $U_q(\mathfrak{g})$-module $M_{\aff}$ with a weight space decomposition $M_{\aff}=\bigoplus_{\lambda\in P}(M_{\aff})_{\lambda}$ by \[
(M_{\aff})_{\lambda}=M_{\cl(\lambda)}.
\]
The action of $E_i, F_i\;(i\in I)$ is defined so that the canonical map $M_{\aff}\to M$ is $U_q^{\prime}(\mathfrak{g})$-linear. Let $z$ denote the $U_q^{\prime}(\mathfrak{g})$-linear automorphism of $M_{\aff}$ defined as
\[
(M_{\aff})_{\lambda}\stackrel{\cong}{\longrightarrow} M_{\cl(\lambda)}=M_{\cl(\lambda+\delta)}\stackrel{\cong}{\longrightarrow} (M_{\aff})_{\lambda+\delta}.
\]
For $a\in \mathbb{Q}(q)$, we define the $U_q^{\prime}(\mathfrak{g})$-module $M_a$ by
\[
M_a=M_{\aff}/(z-a)M_{\aff}.
\]

We set $I_0=I\setminus\{0\}$. Let $\mathfrak{g}_0$ be the simple Lie algebra with Cartan matrix $(a_{ij})_{i,j\in I_0}$. Set $\mathfrak{h}^{*0}=\{\lambda\in \mathfrak{h}^* \mid \langle \lambda, c \rangle =0\}$ and $P^0=P\cap \mathfrak{h}^{*0}=\left(\bigoplus_{i\in I_0}\mathbb{Z}\varpi_i\right)\oplus\mathbb{Z}\delta$. Here $\varpi_i=\Lambda_i-a_i^{\vee}\Lambda_0\in P$ is the level-zero fundamental weight. We identify $\cl(\mathfrak{h}^{*0})$ with the dual of the Cartan subalgebra of $\mathfrak{g}_0$. We regard $P_{\cl}^0=\cl(P^0)$ as the weight lattice of $\mathfrak{g}_0$. We set $W_0=\langle r_i \mid i\in I_0 \rangle$ and $Q_0^{\vee}=\bigoplus_{i\in I_0}\mathbb{Z}\alpha_i^{\vee}$.

For $\xi\in Q_0^{\vee}$, define $t_{\xi}\in W$ as
\[
t_{\xi}(\lambda)=\lambda+\langle\lambda,c\rangle\nu(\xi)-\left(\langle\lambda,\xi\rangle+\frac{1}{2}(\xi,\xi)\langle \lambda,c\rangle\right)\delta
\]
where $\nu:\mathfrak{h}\to\mathfrak{h}^*$ is defined by 
\[
\langle\nu(h_1),h_2\rangle=(h_1,h_2)\quad(h_1, h_2\in \mathfrak{h}).
\]

Then we have $W= W_0\ltimes\{t_{\xi} \mid \xi\in Q_0^{\vee}\}\cong W_0\ltimes Q_0^{\vee}$.

\begin{lemma}[{\label{wt}\cite[Proposition 5.8]{Kas3}}]
  \textit{If $x_1,x_2\in W$ satisfies $x_1\varpi_i=x_2\varpi_i$, we have $S_{x_1}u_{\varpi_i}=S_{x_2}u_{\varpi_i}$ in $V(\varpi_i)$.\qed}
\end{lemma}

We denote $S_xu_{\varpi_i}\in V(\varpi_i)$ by $u_{x\varpi_i}$. There is no confusion with this notation by Lemma \ref{wt}.

By \cite[Section 5.3]{Kas2}, we there is a unique $U_q^{\prime}(\mathfrak{g})$-linear automorphism $z_i:V(\varpi_i)\to V(\varpi_i)$ which sends $u_{\varpi_i}$ to $u_{\varpi_i+\delta}$. We define the \textit{level-zero fundamental representation} $W(\varpi_i)$ by 
\[
W(\varpi_i)=V(\varpi_i)/(z_i-1)V(\varpi_i).
\]

\begin{theorem}[{\label{f}\cite[Theorem 5.17]{Kas3}, \cite[Lemma 4.3]{Kas4}}]
\textit{
  \begin{enumerate}
    \item $W(\varpi_i)$ is a finite dimensional $U_q^{\prime}(\mathfrak{g})$-module;
    \item $V(\varpi_i)$ is isomorphic to $W(\varpi_i)_{\aff}$;
    \item If $a_i^{\vee}=1$, then $W(\varpi_i)$ is an irreducible $U_q(\mathfrak{g}_0)$-module. Here $U_q(\mathfrak{g}_0)$ is the subalgebra of $U_q(\mathfrak{g})$ generated by $E_i, F_i, t_i\;(i\in I_0)$.\qed
  \end{enumerate}
  }
\end{theorem}

We set 
\[
 P_0=\bigoplus_{i\in I_0}\mathbb{Z}\varpi_i\subset P,\quad\textstyle P_{0,+}=\{\lambda=\sum_{i\in I_0}m_i\varpi_i\in P_0 \mid m_i\geq0\text{ for all }i\in I\}.
\]

Let $\lambda=m_i\varpi_i \in P_{0,+}$. We define a $U_q(\mathfrak{g})$-module $\widetilde{V}(\lambda)$ as
\[
\widetilde{V}(\lambda)=\bigotimes_{i\in I_0}V(\varpi_i)^{\otimes m_i}.
\]

Here we take any ordering of $I_0$ to define the tensor product.

Let $\tilde{u}_{\lambda}=\otimes_{i\in I_0}u_{\varpi_i}^{\otimes m_i}\in \widetilde{V}(\lambda)$. Then $\tilde{u}_{\lambda}$ is an extremal vector by Lemma \ref{ext}. So there is a unique $U_q(\mathfrak{g})$-linear morphism $\Phi_{\lambda}: V(\lambda)\to \widetilde{V}(\lambda)$ such that $\Phi_{\lambda}(u_{\lambda})=\tilde{u}_{\lambda}$.

For each $i\in I_0$ and $\nu=1,\ldots,\lambda_i$, let $z_{i,\nu}$ denote the $U_q^{\prime}(\mathfrak{g})$-linear automorphism of $\widetilde{V}(\lambda)$ obtained by the action of $z_i:V(\varpi_i)\to V(\varpi_i)$ on the $\nu$-th factor of $V(\varpi_i)^{\otimes m_i}$ in $\widetilde{V}(\lambda)$.

\begin{definition}\label{schur}
  For $\lambda=\sum_{i\in I_0}m_i\varpi_i\in P_{0,+}$, we define 
  \[
  \overline{\Par}(\lambda)=\{\mathbf{c}_0=(\rho^{(i)})_{i\in I_0} \mid \text{$\rho^{(i)}$ is a partition of length $\leq m_i$ for all $i\in I$}\}.
  \]

For $\mathbf{c}_0\in \overline{\Par}(\lambda)$, we define $S_{\mathbf{c}_0}\in U_q^+(\mathfrak{g})$ as in \cite[Section 3.1]{BN} and set $S_{\mathbf{c}_0}^-=\overline{S_{\mathbf{c}_0}^{\vee}}$.
\end{definition}

For $\mathbf{c}_0=(\rho^{(i)})_{i\in I_0}$, we set 
\[
s_{\mathbf{c}_0}(z_{i,\nu}^{-1})=\prod_{i\in I_0}s_{\rho^{(i)}}(z_{i,1}^{-1},\ldots,z_{i,m_i}^{-1})
\]
where $s_{\rho}(x_1,\ldots,x_m)$ denotes the Schur polynomial in the variables $x_1,\ldots,x_m$ corresponding to the partition $\rho$.

\begin{proposition}[{\cite[Proposition 4.10]{BN}}]
  \textit{Let $\mathbf{c}_0\in \overline{\Par(\lambda)}$. Then we have 
  \[
  \pushQED{\qed}
  \Phi_{\lambda}(S_{\mathbf{c}_0}^-u_{\lambda})=s_{\mathbf{c}_0}(z_{i,\nu}^{-1})\tilde{u}_{\lambda}.\qedhere
\popQED
  \]
  }
\end{proposition}

\begin{corollary}\label{BN}
  \textit{Let $\mathbf{c}_0\in \overline{\Par(\lambda)}$ and $\xi=\sum_{i\in I_0}k_i\alpha_i^{\vee}$. We have 
  \[
   \Phi_{\lambda}(S_{\mathbf{c}_0}^-S_{t_{\xi}}u_{\lambda})=s_{\mathbf{c}_0}(z_{i,\nu}^{-1})\prod_{i\in I_0}(z_{i,1}\cdots z_{i,m_i})^{-k_i}\tilde{u}_{\lambda}.
  \]
  }
\end{corollary}

\begin{proof}
For each $i\in I_0$, there is a $U_q^{\prime}(\mathfrak{g})$-linear automorphism of $V(\varpi_i)$ given by $u_{\varpi_i}\mapsto z_i^{-k_i}u_{\varpi_i}$. By considering the tensor product of these automorphisms, we obtain a $U_q^{\prime}(\mathfrak{g})$-linear automorphism of $\tilde{V}(\lambda)$ such that 
\[
\tilde{u}_{\lambda}\mapsto \prod_{i\in I_0}(z_{i,1}\cdots z_{i,m_i})^{-k_i}\tilde{u}_{\lambda}=S_{t_{\xi}}\tilde{u}_{\lambda}.
\] Since $S_{\mathbf{c}_0}\tilde{u}_{\lambda}=s_{\mathbf{c}_0}(z_{i,\nu}^{-1})\tilde{u}_{\lambda}$ is mapped to $s_{\mathbf{c}_0}(z_{i,\nu}^{-1})\prod_{i\in I_0}(z_{i,1}\cdots z_{i,m_i})^{-k_i}\tilde{u}_{\lambda}$ by this automorphism, we obtain the desired result.
\end{proof}

\begin{theorem}[{\label{inj}\cite[Corollary 4.15]{BN}}]
  \textit{The $U_q(\mathfrak{g})$-linear morphism $\Phi_{\lambda}:V(\lambda)\to \widetilde{V}(\lambda)$ is injective.\qed}
\end{theorem}

In the remainder of this paper, we fix the order of $I_0$ as $1<2<\cdots <n$ unless specifically stated otherwise.

\subsection{Semi-infinite LS path model for level-zero extremal weight modules}

In this subsection, we fix $\lambda=\sum_{i\in I_0}m_i\varpi_i\in P_{0,+}$ and set $J=\{i\in I_0 \mid m_i=0\}$.

Let $\Delta^{\re}$ denote the set of real roots. For $\alpha\in \Delta^{\re}$, let $s_{\alpha}$ denote the reflection with respect to $\alpha$.

We define
\begin{gather*}
  \Delta_J=\left(\bigoplus_{j\in J}\mathbb{Z}\alpha_j\right)\cap \Delta, \\
  (\Delta_J)_{\af}=\{\alpha+n\delta \mid \alpha\in \Delta_J, n\in \mathbb{Z}\}, \\
  (\Delta_J)_{\af}^+=(\Delta_J)_{\af}\cap \Delta_+, \\
  (W_J)_{\af}=\langle s_{\beta} \mid \beta\in (\Delta_J)_{\af}\rangle, \\
  (W^J)_{\af}=\{x\in W \mid x\beta>0 \text{ for all } \beta\in (\Delta_J)_{\af}^+\}.
\end{gather*}

For $x\in W$, there is a unique $x_1\in (W^J)_{\af}$ and a unique $x_2\in (W_J)_{\af}$ such that $x=x_1x_2$ by \cite[Lemma 10.6]{LS}. So we can define a map $\Pi^J:W\to (W^J)_{\af}$ by $x\mapsto x_1$ for $x=x_1x_2$ with $x_1\in (W^J)_{\af}$ and $x_2\in (W_J)_{\af}$.

 Let $\ell: W\to \mathbb{Z}$ denote the length function. We define a function $\mathrm{ht}:Q_0^{\vee}\to \mathbb{Z}$ by $\mathrm{ht}(\sum_{i\in I_0}k_i\alpha_i^{\vee})=\sum_{i\in I_0}k_i$.

We set $(W_0)_J=\langle s_j \mid j\in J\rangle$. Let $(W_0)^J$ denote the set of minimal coset representatives for $W_0/(W_0)_J$.

\begin{definition}
 An element $\xi\in Q_0^{\vee}$ is said to be $J$-adjusted if $\langle \xi,\gamma \rangle\in \{0,-1\}$ for all $\gamma\in  \Delta_J\cap \Delta_+$. Let $Q_0^{\vee, J\mathchar`-\ad}$ denote the set of $J$-adjusted elements.
\end{definition}

\begin{lemma}[{\cite[Lemma 2.3.5]{INS}\label{ad}}]
\textit{
\begin{enumerate}
  \item For each $\xi\in Q_0^{\vee}$, there exists a unique $\varphi_J(\xi)\in \bigoplus_{j\in J}\mathbb{Z}\alpha_j^{\vee}$ such that $\xi+\varphi_J(\xi)\in Q_0^{\vee, J\mathchar`-\ad}$. In particular, $\xi\in Q_0^{\vee, J\mathchar`-\ad}$ if and only if $\varphi_J(\xi)=0$;
  \item For each $\xi\in Q_0^{\vee}$, the element $\Pi^J(t_{\xi})\in (W^J)_{\af}$ is of the form $\Pi^J(t_{\xi})=z_{\xi}t_{\xi+\varphi_J(\xi)}$ for some $z_{\xi}\in (W_0)_J$;
  \item We have 
  \[
  \pushQED{\qed} 
  (W^J)_{\af}=\{wz_{\xi}t_{\xi} \mid w\in (W_0)^J,\;\xi\in Q_0^{\vee, J\mathchar`-\ad}\}.\qedhere
\popQED
  \]
\end{enumerate}
}
\end{lemma}

\begin{definition}
  Let $x=wt_{\xi}\in W\;(w\in W_0,\;\xi\in Q_0^{\vee})$. We define the semi-infinite length $\ell^{\frac{\infty}{2}}(x)$ of $x$ by
  \[
  \ell^{\frac{\infty}{2}}(x)=\ell(w)+2\mathrm{ht}(\xi).
  \]
\end{definition}

\begin{definition}
\begin{enumerate}
  \item Let $\SiB^J$ be the $\Delta_+$-labeled directed graph with vertex set $W^J$ and edges of the form $x\xrightarrow{\beta}r_{\beta}x$, where $x\in W^J$ and $\beta\in \Delta_+$ satisfying $r_{\beta}x\in W^J$ and $\ell^{\frac{\infty}{2}}(r_{\beta}x)=\ell^{\frac{\infty}{2}}(x)+1$;
  \item The semi-infinite Bruhat order $\prec$ on $W^J$ is defined as follows:
  \[
  x\prec y \Leftrightarrow \text{there is a directed path from $x$ to $y$ in $\SiB^J$}.
  \]
\end{enumerate}
\end{definition}

\begin{definition}\label{1}
For $a\in \mathbb{Q}$ with $0<a\le 1$, we define $\SiB(\lambda; a)$ as the subgraph of $\SiB^J$ with the same vertex set but having only the edges of the form $x\xrightarrow{\beta}r_{\beta}x$ with $a\langle x\lambda,\beta^{\vee}\rangle\in \mathbb{Z}$. 
\end{definition}

\begin{remark}
  We have $\SiB(\lambda;1)=\SiB^J$ in the setting of Definition \ref{1}. 
\end{remark}

\begin{definition}
  The pair $(\bm{x};\bm{a})$ of a sequence $\bm{x}=(x_1,\ldots,x_s)$ in $W^J$ satisfying $x_1\succeq\cdots\succeq x_s$ and a sequence $\bm{a}=(a_0,a_1,\ldots,a_s)$ of rational numbers satisfying $0=a_0<a_1<\cdots<a_s$ is called a semi-infinite Lakshmibai-Seshadri path (semi-infinite LS path for short) if there exists a directed path from $x_{u+1}$ to $x_u$ in $\SiB(\lambda;a_u)$ for $u=1,\ldots,s-1$. Let $\mathbb{B}^{\frac{\infty}{2}}(\lambda)$ denote the set of all semi-infinite LS paths. Let $\pi_e$ denote the semi-infinite LS path $(e;0,1)$.
\end{definition}

We equip $\mathbb{B}^{\frac{\infty}{2}}(\lambda)$ with a crystal structure as in \cite[Section 3.1]{INS}.

\begin{theorem}
  \textit{There is an isomorphism $\Psi_{\lambda}:\mathcal{B}(\lambda)\stackrel{\sim}{\longrightarrow} \mathbb{B}^{\frac{\infty}{2}}(\lambda)$ in the category $\mathcal{C}(I, P)$, which sends $u_{\lambda}$ to $\pi_e$.\qed}
\end{theorem}

\subsection{Characterization of Demazure submodules in terms of semi-infinite LS paths}\label{dema}

For $\lambda\in P$ and $x\in W$, we define the $U_q^-(\mathfrak{g})$-module $V_x^-(\lambda)$ by
\[
V_x^-(\lambda)=U_q^-(\mathfrak{g})S_xu_{\lambda}\subset V(\lambda).
\]

By \cite[Section 4.1]{NS}, $V_x^-(\lambda)$ is compatible with the global basis of $V(\lambda)$; namely, 
\[
V_x^-(\lambda)=\bigoplus_{b\in \mathcal{B}_x^-(\lambda)}\mathbb{Q}(q_s)G(b)\subset V(\lambda)=\bigoplus_{b\in \mathcal{B}(\lambda)}\mathbb{Q}(q_s)G(b)
\]
for some subset $\mathcal{B}_x^-(\lambda)$ of $\mathcal{B}(\lambda)$.

In the remainder of this subsection, we fix $\lambda=\sum_{i\in I_0}m_i\varpi_i\in P_{0,+}$ and set $J=\{i\in I_0 \mid m_i=0\}$.

\begin{lemma}[{\cite[Lemma 4.1.3]{NS}\label{eq}}]
  \textit{For $x,y\in W$, 
  \[
  \pushQED{\qed} 
  V_x^-(\lambda)=V_y^-(\lambda) \Leftrightarrow \mathcal{B}_x^-(\lambda)=\mathcal{B}_y^-(\lambda) \Leftrightarrow x^{-1}y\in (W_J)_{\af}.\qedhere
\popQED
  \]
  }
\end{lemma}

For $\eta=(x_1,\ldots,x_s;a_0,a_1,\ldots,a_s)\in \mathbb{B}^{\frac{\infty}{2}}(\lambda)$, we set $\kappa(\eta)=x_s$. For each $x\in (W^J)_{\af}$, we set
\[
\mathbb{B}_{\succeq x}^{\frac{\infty}{2}}(\lambda)=\{\eta\in \mathbb{B}^{\frac{\infty}{2}}(\lambda) \mid \kappa(\eta)\succeq x\}.
\]

\begin{theorem}[{\cite[Theorem 4.2.1]{NS}}]
  \textit{For $x\in W$, we have
  \[
  \pushQED{\qed} 
  \Psi_{\lambda}(\mathcal{B}_x^-(\lambda))=\mathbb{B}_{\succeq x}^{\frac{\infty}{2}}(\lambda).\qedhere
\popQED
  \]
  }
\end{theorem}

\begin{lemma}\label{txi}
  \textit{Let $\xi\in Q_0^{\vee}$. Then we have an isomorphism
  \[
  \Upsilon:V(\lambda)\stackrel{\cong}{\longrightarrow} V(\lambda)\otimes V(-\langle\xi,\lambda\rangle\delta)
  \]
  which maps $u_{\lambda}$ to $S_{t_{-\xi}}(u_{\lambda}\otimes u_{-\langle\xi,\lambda\rangle\delta})$.
  }
\end{lemma}

\begin{proof}
  We have $U_q(\mathfrak{g})$-linear morphisms 
  \[
    \Upsilon:V(\lambda)\to V(\lambda)\otimes V(-\langle\xi,\lambda\rangle\delta),\quad
    \Phi:V(\lambda)\to V(\lambda)\otimes V(\langle\xi,\lambda\rangle\delta)
  \]
  such that $\Upsilon(u_{\lambda})=S_{t_{-\xi}}(u_{\lambda}\otimes u_{-\langle\xi,\lambda\rangle\delta})=(S_{t_{-\xi}}u_{\lambda})\otimes u_{-\langle\xi,\lambda\rangle\delta},\; \Phi(u_{\lambda})=S_{t_{\xi}}(u_{\lambda}\otimes u_{\langle\xi,\lambda\rangle\delta})=(S_{t_{\xi}}u_{\lambda})\otimes u_{\langle\xi,\lambda\rangle\delta}$. Also, we have an $U_q(\mathfrak{g})$-linear isomorphism 
  \[
  \Omega:V(\lambda)\stackrel{\cong}{\longrightarrow} V(\lambda)\otimes V(\langle\xi,\lambda\rangle) \otimes V(-\langle\xi,\lambda\rangle\delta)
  \]
  such that $\Omega(u_{\lambda})= u_{\lambda}\otimes u_{\langle\xi,\lambda\rangle\delta}\otimes u_{-\langle\xi,\lambda\rangle\delta}$. Then we have
  \[
    (\Omega^{-1}\circ(\Phi\otimes \id)\circ\Upsilon)(u_{\lambda})=(\Omega^{-1}\circ(\Phi\otimes \id))((S_{t_{-\xi}}u_{\lambda})\otimes u_{-\langle\xi,\lambda\rangle\delta})
    =\Omega^{-1}(u_{\lambda}\otimes u_{\langle\xi,\lambda\rangle\delta}\otimes u_{-\langle\xi,\lambda\rangle\delta})=u_{\lambda}
  \]
  and 
  \[
    (\Upsilon\circ\Omega^{-1}\circ(\Phi\otimes \id))(u_{\lambda}\otimes u_{-\langle\xi,\lambda\rangle\delta})=(\Upsilon\circ\Omega^{-1})((S_{t_{\xi}}u_{\lambda})\otimes u_{\langle\xi,\lambda\rangle\delta}\otimes u_{-\langle\xi,\lambda\rangle\delta})
    =\Upsilon(S_{t_{\xi}}u_{\lambda})=u_{\lambda}\otimes u_{-\langle \xi,\lambda\rangle\delta},
  \]
  which implies that $\Omega^{-1}\circ(\Phi\otimes \id)$ is the inverse map of $\Upsilon$.
\end{proof}

\begin{proposition}\label{dem}
\textit{
  \begin{enumerate}
    \item Let $\xi=\sum_{i\in I_0}k_i\alpha_i^{\vee}\in Q_0^{\vee}$ and $\zeta=\sum_{i\in I_0}l_i\alpha_i^{\vee}\in Q_0^{\vee}$. If $k_i\ge l_i$ for all $i\in I_0$, then
    \[
    V_{t_{\xi}}^-(\lambda) \subset V_{t_{\zeta}}^-(\lambda);
    \]
    \item We have
    \[
    V(\lambda)=\bigcup_{\xi\in \bigoplus_{i\in I_0\setminus J}\mathbb{Z}\alpha_i^{\vee}}V_{t_{\xi}}^-(\lambda).
    \]
  \end{enumerate}
  }
\end{proposition}

\begin{proof}

We first prove the first assertion. It suffices to prove the statement for the case where $\xi=\zeta+\alpha_i^{\vee}$ for some $i\in I_0$. We first handle the case $\zeta=0$. Here, it is sufficient to show that $S_{t_{\alpha_i}^{\vee}}u_{\lambda}\in V_e^-(m\varpi_i)$. By \cite[Lemma 4.5]{BN}, we have $S_{t_{\alpha_i}^{\vee}}u_{\lambda}=\tilde{P}_{i, -m_i}u_{\lambda}$ where $\tilde{P}_{i, -m_i}\in U_q^-(\mathfrak{g})$ is the vector corresponding to the `elementary symmetric polynomial' defined in \cite[Remark 4.1]{BN}. Hence, we have $S_{t_{\alpha_i}^{\vee}}u_{\lambda}\in V_e^-(\lambda)$.

 Next, we consider the general case $\zeta=\sum_{i\in I_0}l_i\alpha_i^{\vee}\in Q_0^{\vee}$. By Lemma \ref{txi}, we have the isomorphism of $U_q(\mathfrak{g})$-modules
  \[
  \textstyle \Upsilon : V(\lambda)\stackrel{\sim}{\longrightarrow} V(\lambda)\otimes V(-\langle\zeta,\lambda\rangle\delta)
  \]
  such that $\Upsilon(u_{\lambda})=S_{t_{-\zeta}}(u_{\lambda}\otimes u_{-\langle\zeta,\lambda\rangle\delta})$. Since $\Upsilon(S_{t_{\zeta}}u_{\lambda})=u_{\lambda}\otimes u_{-\langle\zeta,\lambda\rangle\delta}$ and $\Upsilon(S_{t_{\xi}}u_{\lambda})=(S_{t_{\alpha_i^{\vee}}}u_{\lambda})\otimes u_{-\langle\zeta,\lambda\rangle\delta}$, we have $\Upsilon(V_{\zeta}^-(\lambda))=V_e^-(\lambda)\otimes u_{-\langle\zeta,\lambda\rangle\delta}, \Upsilon(V_{\xi}^-(\lambda))=V_{t_{\alpha_i^{\vee}}}^-(\lambda)\otimes u_{-\langle\zeta,\lambda\rangle\delta}$. Hence, we have $V_{\xi}^-(\lambda)\subset V_{\zeta}^-(\lambda)$.

  Next, we prove the second assertion. Take $\eta\in \mathbb{B}^{\frac{\infty}{2}}(\lambda)$. Then it suffices to show that $\eta\in \mathbb{B}_{\succeq z_{\xi}t_{\xi}}^{\frac{\infty}{2}}(\lambda)$ for some $\xi\in Q_0^{\vee, J\mathchar`-\ad}$. Let $x=\kappa(\eta)$. Since $x\in (W^J)_{\af}$, we can write $x=wz_{\zeta}t_{\zeta}$ for some $w\in (W_0)^J$ and $\zeta\in Q_0^{\vee, J\mathchar`-\ad}$ by Lemma \ref{ad}. Take a reduced expression $s_{i_1}\cdots s_{i_p}$ of $w$. Then $s_{i_a}\cdots s_{i_p}\in (W_0)^J$ for all $1\leq a\leq p$ and we have the following path in $\SiB^J$:
    \[
    x=s_{i_1}\cdots s_{i_p}z_{\zeta}t_{\zeta}\succeq s_{i_2}\cdots s_{i_p}z_{\zeta}t_{\zeta} \succeq \cdots \succeq s_{i_p}z_{\zeta}t_{\zeta} \succeq z_{\zeta}t_{\zeta}.
    \]
    Hence, we have $\eta\in \mathbb{B}_{\succeq z_{\zeta}t_{\zeta}}^{\frac{\infty}{2}}(\lambda)$, as desired.
\end{proof}

\subsection{Graded character of the Demazure submodule $V_e^-(\lambda)$}

Let $M=\bigoplus_{\lambda\in P}M_{\lambda}$ be an integrable $U_q(\mathfrak{g})$-module and $N$ be an $U_q^-(\mathfrak{g})$-submodule of $M$ such that $N=\bigoplus_{\lambda\in P}(N\cap M_{\lambda})$. Then we define the graded character $\gch N$ of $N$ by the formal sum
\[
\gch N=\sum_{\mu\in P_0, k\in \mathbb{Z}}(\dim (N\cap M_{\nu+k\delta}) )x^{\nu}q^k.
\]

Let $\lambda=\sum_{i\in I_0}m_i\varpi_i\in P_{0,+}$. Let $P_{\lambda}(x;q,t)\in \mathbb{Q}(q,t)[P_0]$ denotes the symmetric Macdonald polynomial where we regard $P_0$ as the weight lattice of $\mathfrak{g}_0$ (see \cite[Chapter 5]{M2} for the definition of the symmetric Macdonald polynomials).

\begin{theorem}[{\cite[Theorem 6.1.1]{NS}\label{bq}}]
 \textit{Keep the setting above. The graded character $\gch V_e^-(\lambda)$ of $V_e^-(\lambda)$ is expressed as
 \[
 \pushQED{\qed}
 \gch V_e^-(\lambda)=\left(\prod_{i\in I_0}\prod_{r=1}^{m_i}(1-q^{-r})\right)^{-1}P_{\lambda}(x;q^{-1},0)\in \mathbb{Z}[P_0](\!(q^{-1})\!).\qedhere
\popQED
 \]
 }
\end{theorem}

\begin{corollary}\label{br}
  \textit{Keep the setting above. Let $\xi\in Q_0^{\vee}$. The graded character $\gch V_{t_{\xi}}^-(\lambda)$ of $V_{t_{\xi}}^-(\lambda)$ is expressed as
  \[
  \gch V_{t_{\xi}}^-(\lambda)=q^{-\langle\xi,\lambda\rangle}\left(\prod_{i\in I_0}\prod_{r=1}^{m_i}(1-q^{-r})\right)^{-1}P_{\lambda}(x;q^{-1},0).
  \]
  }
\end{corollary}

\begin{proof}
  We have the isomorphism of $U_q(\mathfrak{g})$-modules
  \[
  V(\lambda)\stackrel{\cong}{\longrightarrow} V(\lambda)\otimes V(-\langle\xi,\lambda\rangle\delta)
  \]
that maps $u_{\lambda}$ to $(S_{-t_{\xi}}u_{\lambda})\otimes u_{-\langle \xi,\lambda\rangle\delta}$ by Lemma \ref{txi}. The assertion then follows directly from Theorem \ref{bq}.
\end{proof}

In the remainder of this chapter, we assume that $\mathfrak{g}_0$ is of type $A$. We refer basics on partitions and Young diagram to \cite{M1}.

Let $\mathcal{P}$ denote the set of all partitions. For $\Lambda \in \mathcal{P}$, the length of $\Lambda$ is denoted by $\ell(\Lambda)$ and the size of $\Lambda$ is denoted by $|\Lambda|$. For $k \in \mathbb{Z}_{\geq 0}$, we define $\mathcal{P}_k = \{ \Lambda \in \mathcal{P} \mid \ell(\Lambda) \leq k \}$. Let $X_{n+1} = \bigoplus_{i=1}^{n+1} \mathbb{Z} \varepsilon_i$ denote the weight lattice of $\mathop{GL}_{n+1}$. Let $\pi_{n+1}: X_{n+1} \to X_{n+1} / \mathbb{Z}(\varepsilon_1 + \cdots + \varepsilon_{n+1})$ denote the canonical projection. We identify $X_{n+1} / \mathbb{Z}(\varepsilon_1 + \cdots + \varepsilon_{n+1})$ with $P_0$ via $\pi_{n+1}(\varepsilon_1 + \cdots + \varepsilon_i) \mapsto \varpi_i$. A $\mathop{GL}_{n+1}$-weight $\Lambda=\sum_{i=1}^{n+1}\lambda_i\varepsilon_i$ is said to be \textit{polynomial} if $\lambda_i\ge 0$ for all $i=1,\ldots,n+1$. A $\mathop{GL}_{n+1}$-weight $\Lambda=\sum_{i=1}^{n+1}\lambda_i\varepsilon_i$ is said to be \textit{dominant} if $\lambda_1\ge\cdots\ge\lambda_{n+1}$. We identify $\mathcal{P}_{n+1}$ with the set of all dominant polynomial weights in $X_{n+1}$ via $(\lambda_1,\lambda_2,\ldots)\mapsto \sum_{i=1}^{n+1}\lambda_i\varepsilon_i$.

  There exists a `$\mathop{GL}_{n+1}$-version' of the symmetric Macdonald polynomial $P_{\Lambda}^{\mathop{GL}}(x;q,t)\in \mathbb{Q}(q,t)[X_{n+1}]$ for a dominant polynomial weight $\Lambda\in X_{n+1}$ (see \cite[Chapter V\hspace{-1pt}I]{M1} for the definition). We define a $\mathbb{Q}(q,t)$-linear map $\Pi_{n+1}:\mathbb{Q}(q,t)[X_{n+1}]\to \mathbb{Q}(q,t)[P_0]$ by $\Pi_{n+1}(x^{\Lambda})=x^{\pi_{n+1}(\Lambda)}$. Then, we have
  \[
  \Pi_{n+1}(P_{\Lambda}^{\mathop{GL}}(x;q,t))=P_{\pi_{n+1}(\Lambda)}(x;q,t).
  \]

  By setting $x^{\varepsilon_i}=x_i$, $P_{\Lambda}^{\mathop{GL}}(x;q,t)\in \mathbb{Q}(q,t)[X_{n+1}]$ can also be written as $P_{\Lambda}^{\mathop{GL}}(x_1,\ldots,x_{n+1};q,t)$. This is a homogeneous polynomial of degree $|\Lambda|$.

\subsection{Branching rules for Macdonald polynomials}

For partitions $\Lambda=(\lambda_1,\lambda_2,\ldots)$ and $M=(\mu_1,\mu_2,\ldots)$, we write $\Lambda\supset M$ to mean that the Young diagram of $\Lambda$ contains the Young diagram of $M$, i.e. that $\lambda_i\ge\mu_i$ for all $i\ge1$. The skew diagram of a skew shape $\Lambda/M$ is the set-theoretic difference of the Young diagrams of $\Lambda$ and $M$. We say that $\Lambda/M$ is a \emph{horizontal strip} if $\lambda_i \geq \mu_i \geq \lambda_{i+1}$ for all $i \geq 1$.

\begin{theorem}[{\cite[Section V\hspace{-1pt}I-7]{M1}\label{brMac}}]
\textit{Let $\Lambda\in \mathcal{P}_n$. For a partition $M=(\mu_1,\ldots,\mu_n)$ such that $\Lambda/M$ is horizontal strip, we define $\psi_{\Lambda/M}$ by
\[
\psi_{\Lambda/M}=\prod_{1\le i\le j\le \ell(\mu)}\frac{f(q^{\mu_i-\mu_j}t^{j-i})f(q^{\lambda_i-\lambda_{j+1}}t^{j-i})}{f(q^{\lambda_i-\mu_j}t^{j-i})f(q^{\mu_i-\lambda_{j+1}}t^{j-i})},
\]
where $f(a)=\frac{(at)_{\infty}}{(aq)_{\infty}}$ with $(a)_{\infty}=\prod_{i\ge0}(1-aq^i)$. Then, we have
  \[
  \pushQED{\qed} 
  P_{\Lambda}^{\mathop{GL}}(x_1,\ldots,x_n,x_{n+1};q,t)=\sum_{\substack{M\in \mathcal{P}_n\\ \Lambda/M:\text{horizontal strip}}} x_{n+1}^{|\Lambda-M|}\psi_{\Lambda/M}P_M^{\mathop{GL}}(x_1,\ldots,x_n;q,t).\qedhere
\popQED
  \]
  }
\end{theorem}

When $q=t$, we have $P_{\Lambda}^{\mathop{GL}}(x_1,\ldots,x_n,x_n;q,t)=s_{\Lambda}(x_1,\ldots,x_n)$ by \cite[Section V\hspace{-1pt}I-4]{M1} for $n\in \mathbb{Z}_{\ge0}$. Hence, we obtain the following.

\begin{corollary}\label{schbr}
   \textit{\[
   \pushQED{\qed} 
  s_{\Lambda}(x_1,\ldots,x_n,x_{n+1})=\sum_{\substack{M\in \mathcal{P}_n\\ \Lambda/M:\text{horizontal strip}}} x_{n+1}^{|\Lambda-M|}s_M(x_1,\ldots,x_n).\qedhere
\popQED
  \]
  }
\end{corollary}

\section{The $U_q(\widehat{\mathfrak{sl}}_n)$-module $\Psi_{\varepsilon}^*V(\lambda)$}

Keep the setting of the previous section.

\subsection{Embedding among $U_q(\widehat{\mathfrak{sl}}_n)$}

We assume that $A=(a_{ij})_{i,j\in I}$ is of type $A_k^{(1)}$ (i.e., $\mathfrak{g}=\widehat{\mathfrak{sl}}_{k+1}$) for $k\in \mathbb{Z}_{>0}$ hereafter. The integer $D>0$ in the definition of the quantum enveloping algebra is assumed to be $D=1$.

In this subsection, we present an explicit embedding $\Phi_{\varepsilon}: U_q(\widehat{\mathfrak{sl}}_n) \hookrightarrow U_q(\widehat{\mathfrak{sl}}_{n+1})$ for $\varepsilon \in \{\pm1\}$. To avoid confusion, we will use the notation ` $\breve{}$ ' for objects related to $\widehat{\mathfrak{sl}}_{n+1}$ as follows.

Let $\breve{I} = \{0, 1, \ldots, n\}$, and let $\breve{E}_i$ and $\breve{F}_i$ (for $i \in \breve{I}$) denote the Chevalley generators of $U_q(\widehat{\mathfrak{sl}}_{n+1})$. We denote the Cartan subalgebra of $\widehat{\mathfrak{sl}}_{n+1}$ by $\breve{\mathfrak{h}}$ and its dual by $\breve{\mathfrak{h}}^*$. The simple coroots are given by $\breve{\alpha}_i^{\vee}$ (for $i \in \breve{I}$), and the simple roots by $\breve{\alpha}_i$ (for $i \in \breve{I}$). The degree operator is denoted by $\breve{d}$, and let $\breve{\delta} = \sum_{i=0}^n \breve{\alpha}_i$ denote the generator of the imaginary roots.

Let $\breve{P}$ denote the weight lattice of $\widehat{\mathfrak{sl}}_{n+1}$, $\breve{Q}\subset \breve{P}$ the root lattice. The coweight lattice is denoted by $\breve{P}^*$. The fundamental weights are $\breve{\Lambda}_i$ (for $i \in \breve{I}$), and the level-zero fundamental weights are $\breve{\varpi}_i$ (for $i \in \breve{I}_0 = \breve{I} \setminus \{0\}$). Define $\breve{P}_0 = \bigoplus_{i \in \breve{I}_0} \mathbb{Z} \breve{\varpi}_i \subset \breve{P}$ and $\breve{P}_{0,+}=\{\lambda=\sum_{i\in \breve{I}_0}m_i\breve{\varpi}_i\in \breve{P}_0 \mid m_i\ge 0 \;\text{for all} \;i\in \breve{I}_0\}$. We denote the Weyl group by $\breve{W}$. The operator acting on the extremal vectors of an integrable $U_q(\widehat{\mathfrak{sl}}_{n+1})$-module, defined in Section 2.4, is denoted by $\breve{S}_x$ (for $x \in \breve{W}$). For $\lambda\in \breve{P}_{0,+}$ and $\mathbf{c}_0\in\overline{\Par}(\lambda)$, let $\breve{S}_{\mathbf{c}_0}^-$ denote the element of $U_q^-(\widehat{\mathfrak{sl}}_{n+1})$ defined in Definition \ref{schur}. We set $\breve{Q}_0^{\vee}=\bigoplus_{i=1}^n\mathbb{Z}\breve{\alpha}_i^{\vee}$.

For objects related to $\widehat{\mathfrak{sl}}_n$, we do not use the notation ` $\breve{}$ '.

Define the $\mathbb{Q}$-linear map $j:\mathfrak{h}\to \breve{\mathfrak{h}}$ by
\[
  j(\alpha_0^{\vee})=s_n(\breve{\alpha}_0^{\vee})=\breve{\alpha}_0^{\vee}+\breve{\alpha}_n^{\vee},\quad j(\alpha_i^{\vee})=\breve{\alpha}_i^{\vee}\quad (i\in I_0=I\setminus\{0\}),\quad j(d)=\breve{d}.
\]

The map $j$ induces the $\mathbb{Q}$-linear map $j^*:\breve{\mathfrak{h}}^*\to \mathfrak{h}^*$ such that
\[
  j^*(\breve{\Lambda}_0)=\Lambda_0,\quad j^*(\breve{\Lambda}_i)=\Lambda_i\quad(i\in I_0),\quad
  j^*(\breve{\Lambda}_n)=\Lambda_0,\quad j^*(\breve{\delta})=\delta.
\]

We define a $\mathbb{Q}(q,t)$-linear map $\Theta:\mathbb{Q}(q,t)[\breve{P}_0]\to \mathbb{Q}(q,t)[P_0]$ by $\Theta(x^{\nu})=x^{j^*(\nu)}$.

Define the $\mathbb{Z}$-linear map $\gamma:Q\to \breve{Q}$ by
\[
\gamma(\alpha_0)=s_n(\breve{\alpha}_0)=\breve{\alpha}_0+\breve{\alpha}_n,\quad\gamma(\alpha_i)=\breve{\alpha}_i\quad(i\in I_0).
\]

\begin{lemma}\label{2}
  \textit{We have $\langle h, \zeta\rangle =\langle j(h), \gamma(\zeta) \rangle$ for all $h\in \mathfrak{h}, \zeta\in Q$.}
\end{lemma}

\begin{proof}
 Let $i\in I$ and let $h=\sum_{j=0}^{n-1}k_j\alpha_j^{\vee}+l d \in \mathfrak{h}$ be an arbitrary element. We show that $\langle h, \alpha_i\rangle =\langle j(h), \gamma(\alpha_i) \rangle$.

 First, assume $i=0$. We have
 
 \[
  \langle h, \alpha_0 \rangle
  =\left\langle \sum_{j=0}^{n-1}k_j\alpha_j^{\vee}+l d, \alpha_0 \right\rangle
  =2k_0-k_1-k_{n-1}+l
\]
and
\begin{align*}
  \langle j(h), \gamma(\alpha_0) \rangle
  &=\left\langle k_0s_n(\breve{\alpha}_0^{\vee})+\sum_{j=1}^{n-1}k_j\breve{\alpha}_j^{\vee}+l\breve{d}, s_n(\breve{\alpha}_0) \right\rangle=\left\langle \sum_{j=0}^{n-2}k_j\breve{\alpha}_j^{\vee}+k_{n-1}s_n(\breve{\alpha}_{n-1}^{\vee})+l \breve{d}, \breve{\alpha}_0 \right\rangle\\
  &=\left\langle \sum_{j=0}^{n-1}k_j\breve{\alpha}_j^{\vee}+k_{n-1}\breve{\alpha}_n^{\vee}+l \breve{d}, \breve{\alpha}_0 \right\rangle=2k_0-k_1-k_{n-1}+l.
\end{align*}

Thus, we obtain $\langle h, \alpha_0\rangle=\langle j(h), \gamma(\alpha_0)\rangle$.

Next, assume $1\leq i\leq n-2$. We have

\[
  \langle h, \alpha_i \rangle
  =\left\langle \sum_{j=0}^{n-1}k_j\alpha_j^{\vee}+l d, \alpha_i \right\rangle
  =2k_i-k_{i-1}-k_{i+1}
\]
and
\begin{align*}
  \langle j(h), \gamma(\alpha_i) \rangle
  &=\left\langle k_0s_n(\breve{\alpha}_0^{\vee})+\sum_{j=1}^{n-1}k_j\breve{\alpha}_j^{\vee}+l \breve{d}, \breve{\alpha}_i \right\rangle\\
  &=\left\langle \sum_{j=0}^{n-1}k_j\breve{\alpha}_j^{\vee}+k_0\breve{\alpha}_n^{\vee}+l\breve{d}, \breve{\alpha}_i \right\rangle=2k_i-k_{i-1}-k_{i+1}.
\end{align*}
Thus, we obtain $\langle h, \alpha_i\rangle=\langle j(h), \gamma(\alpha_i)\rangle$.

Finally, assume $i=n-1$. We have
\[
  \langle h, \alpha_{n-1} \rangle
  =\left\langle \sum_{j=0}^{n-1}k_j\alpha_j^{\vee}+l d, \alpha_{n-1} \right\rangle
  =2k_{n-1}-k_{n-2}-k_0
\]
and
\begin{align*}
  \langle j(h), \gamma(\alpha_{n-1}) \rangle
  &=\left\langle k_0s_n(\breve{\alpha}_0^{\vee})+\sum_{j=1}^{n-1}k_j\breve{\alpha}_j^{\vee}+l \breve{d}, \breve{\alpha}_{n-1} \right\rangle\\
  &=\left\langle \sum_{j=0}^{n-1}k_j\breve{\alpha}_j^{\vee}+k_0\breve{\alpha}_n^{\vee}+l \breve{d}, \breve{\alpha}_{n-1} \right\rangle=2k_{n-1}-k_{n-2}-k_0.
\end{align*}
Thus, we obtain $\langle h, \alpha_{n-1}\rangle=\langle j(h), \gamma(\alpha_{n-1})\rangle$. These calculations complete the proof.
\end{proof}

\begin{proposition}
  \textit{There is a homomorphism of algebras $\Psi_{\varepsilon}: U_q(\widehat{\mathfrak{sl}}_n) \to U_q(\widehat{\mathfrak{sl}}_{n+1})$ such that
\begin{gather*}
    \Psi_{\varepsilon}(E_0)=T_n^{\varepsilon}(\breve{E}_0),\quad\Psi_{\varepsilon}(F_0)=T_n^{\varepsilon}(\breve{F}_0), \\
  \Psi_{\varepsilon}(E_i)=\breve{E}_i,\quad\Psi_{\varepsilon}(F_i)=\breve{F}_i\quad(i=1,\ldots,n-1), \quad
  \Psi_{\varepsilon}(q^h)=q^{j(h)}\quad(h\in P^*).
\end{gather*}}
\end{proposition}

\begin{proof}
  It suffices to prove the following equalities in $U_q(\widehat{\mathfrak{sl}}_{n+1})$:
  \begin{gather}
  q^{j(0)}=1,\quad q^{j(h+h^{\prime})}=q^{j(h)}q^{j(h^{\prime})}\quad(h,h^{\prime}\in P^*), \label{e1}\\
  q^{j(h)}\Psi_{\varepsilon}(E_i)q^{-j(h)}=q^{\langle h,\alpha_i\rangle}\Psi_{\varepsilon}(E_i),\quad q^{j(h)}\Psi_{\varepsilon}(F_i)q^{-j(h)}=q^{-\langle h,\alpha_i\rangle}\Psi_{\varepsilon}(F_i)\quad (i\in I,h\in P^*), \label{e2}\\
  \Psi_{\varepsilon}(E_i)\Psi_{\varepsilon}(F_j)-\Psi_{\varepsilon}(F_j)\Psi_{\varepsilon}(E_i)=\delta_{ij}\frac{q^{j(\alpha_i^{\vee})}-q^{-j(\alpha_i^{\vee})}}{q-q^{-1}}\quad(i,j\in I), \label{e3}\\
  \sum_{s=0}^{1-a_{ij}}(-1)^s\Psi_{\varepsilon}(E_i^{(s)})\Psi_{\varepsilon}(E_j)\Psi_{\varepsilon}(E_i^{(1-a_{ij}-s)})=0 \quad(i,j\in I,\quad i\neq j), \label{e4}\\
  \sum_{s=0}^{1-a_{ij}}(-1)^s\Psi_{\varepsilon}(F_i^{(s)})\Psi_{\varepsilon}(F_j)\Psi_{\varepsilon}(F_i^{(1-a_{ij}-s)})=0 \quad(i,j\in I,\quad i\neq j). \label{e5}
\end{gather}

It is easy to show the equation (\ref{e1}).

Next, we deduce the equation (\ref{e2}). Observe that the weight of $\Psi_{\varepsilon}(E_i)$ (resp. $\Psi_{\varepsilon}(F_i)$) is $\gamma(\breve{\alpha}_i)$ (resp. $-\gamma(\breve{\alpha}_i)$) for each $i\in I$. Thus, we have
\begin{align*}
  q^{j(h)}\Psi_{\varepsilon}(E_i)q^{-j(h)}=q^{\langle j(h),\gamma(\alpha_i)\rangle}\Psi_{\varepsilon}(E_i)=q^{\langle h,\alpha_i\rangle}\Psi_{\varepsilon}(E_i),\\
  q^{j(h)}\Psi_{\varepsilon}(F_i)q^{-j(h)}=q^{-\langle j(h),\gamma(\alpha_i)\rangle}\Psi_{\varepsilon}(F_i)=q^{-\langle h,\alpha_i\rangle}\Psi_{\varepsilon}(F_i)
\end{align*}
by Lemma \ref{2}, which proves the required result.

Next, we prove the equation (\ref{e3}). This relation holds trivially for $i,j\neq0$. Since $T_n(\breve{E}_0)=\breve{E}_n\breve{E}_0-q^{-1}\breve{E}_0\breve{E}_n$ and $T_n(\breve{F_0})=\breve{F}_0\breve{F}_n-q\breve{F}_n\breve{F}_0$, we have $T_n(\breve{E}_0)F_k=F_kT_n(\breve{E}_0)$ and $T_n(F_0)E_k=E_kT_n(F_0)$ for $k\in I, k\neq 0$. Therefore, it suffices to prove the case where $i=j=0$. Since $T_n^{\varepsilon}$ is an $\mathbb{Q}(q)$-algebra automorphism, we get
\[
  T_n^{\varepsilon}(\breve{E}_0)T_n^{\varepsilon}(\breve{F}_0)-T_n^{\varepsilon}(\breve{F}_0)T_n^{\varepsilon}(\breve{E}_0)
  =T_n^{\varepsilon}(\breve{E}_0\breve{F}_0-\breve{F}_0\breve{E}_0)
  =\frac{q^{j(\alpha_0^{\vee})}-q^{-j(\alpha_0^{\vee})}}{q-q^{-1}}.
\]
This gives the desired result for the case $i=j=0$.

Finally, we prove the equations $(\ref{e4})$ and $(\ref{e5})$. Since $T_n(\breve{E}_0)=\breve{E}_n\breve{E}_0-q^{-1}\breve{E}_0\breve{E}_n$ and $T_n(\breve{F_0})=\breve{F}_0\breve{F}_n-q\breve{F}_n\breve{F}_0$, we have
\begin{gather}
  \sum_{s=0}^{1-a_{ij}}(-1)^s\Psi_1(E_i^{(s)})\Psi_1(E_j)\Psi_1(E_i^{(1-a_{ij}-s)})=0, \label{s1}\\
  \sum_{s=0}^{1-a_{ij}}(-1)^s\Psi_1(F_i^{(s)})\Psi_1(F_j)\Psi_1(F_i^{(1-a_{ij}-s)})=0 \label{s2}
\end{gather}
by \cite[Theorem 2.1.1]{Li}. Observe that $\Phi(T_n(E_0))=T_n^{-1}(F_0)$ and $\Phi(T_n(F_0))=T_n^{-1}(E_0)$. By applying $\Phi$ to both sides of (\ref{s1}) and (\ref{s2}), we obtain
\begin{gather*}
  \sum_{s=0}^{1-a_{ij}}(-1)^s\Psi_{-1}(F_i^{(s)})\Psi_{-1}(F_j)\Psi_{-1}(F_i^{(1-a_{ij}-s)})=0,\\
  \sum_{s=0}^{1-a_{ij}}(-1)^s\Psi_{-1}(E_i^{(s)})\Psi_{-1}(E_j)\Psi_{-1}(E_i^{(1-a_{ij}-s)})=0.
\end{gather*}
 These complete the proof.
\end{proof}

Let $\Psi_{\varepsilon}^{\pm}$ denote the restriction of $\Psi_{\varepsilon}$ to $U_q^{\pm}(\widehat{\mathfrak{sl}}_n)$.

\begin{proposition}
  \textit{The map $\Psi_{\varepsilon}$ is injective.}
\end{proposition}

\begin{proof}
 By \cite[Theorem 2.1.1]{Li}, $\Psi_1^{\pm}$ are injective. Since $\Psi_{-1}=\Omega\circ\Phi\circ\Psi_{1}$ on $U_q^{\pm}(\widehat{\mathfrak{sl}}_n)$, $\Psi_{-1}$ also gives injections $U_q^+(\widehat{\mathfrak{sl}}_n)\hookrightarrow U_q^+(\widehat{\mathfrak{sl}}_{n+1})$ and $U_q^-(\widehat{\mathfrak{sl}}_n)\hookrightarrow U_q^-(\widehat{\mathfrak{sl}}_{n+1})$.  Additionally, the restriction of $\Psi_{\varepsilon}$ induces an injection $U_q^0(\widehat{\mathfrak{sl}}_n)\hookrightarrow U_q^0(\widehat{\mathfrak{sl}}_{n+1})$. Therefore, for each $\varepsilon=\pm1$, the map 
  \[
  U_q^+(\widehat{\mathfrak{sl}}_n)\otimes U_q^0(\widehat{\mathfrak{sl}}_n)\otimes U_q^-(\widehat{\mathfrak{sl}}_n)\to U_q^+(\widehat{\mathfrak{sl}}_{n+1})\otimes U_q^0(\widehat{\mathfrak{sl}}_{n+1})\otimes U_q^-(\widehat{\mathfrak{sl}}_{n+1})
  \]
 obtained by these injections is also injective, which proves the desired result.
\end{proof}

Note that the restriction of the map $\Psi_{\varepsilon}$ to $U_q^{\prime}(\widehat{\mathfrak{sl}}_n)$ induces the injective algebra homomorphism
\[
\Psi_{\varepsilon}^{\prime}:U_q^{\prime}(\widehat{\mathfrak{sl}}_n)\to U_q^{\prime}(\widehat{\mathfrak{sl}}_{n+1})
\]
and the restriction to $U_q^-(\widehat{\mathfrak{sl}}_n)$ induces the injective algebra homomorphism
\[
\Psi_{\varepsilon}^-:U_q^-(\widehat{\mathfrak{sl}}_n)\to U_q^-(\widehat{\mathfrak{sl}}_{n+1}).
\]

\subsection{Branching rule for $V(\breve{\varpi}_i)$}\label{m1sec}

For a ring homomorphism $f:A\to B$ and a $B$-module $M$, let $f^*M$ denote the $A$-module whose underlying set is $M$, with the action defined by $x\cdot u=f(x)\cdot u$ for $x\in A$ and $u\in M$. In this subsection, we study the structure of the $U_q(\widehat{\mathfrak{sl}}_n)$-module $\Psi_{\varepsilon}^*V(\breve{\varpi}_i)$.

\begin{proposition}
\textit{Let $M=\bigoplus_{\lambda\in \breve{P}}M_{\lambda}$ be an integrable $U_q(\widehat{\mathfrak{sl}}_{n+1})$-module. Then $\Psi_{\varepsilon}^*M$ is also integrable as a $U_q(\widehat{\mathfrak{sl}}_n)$-module.}
\end{proposition}

\begin{proof}
  For $\lambda\in \breve{P}$ and $u\in M_{\lambda}$, we have
\[
\Psi_{\varepsilon}(q^h)u=q^{j(h)}u=q^{\langle j(h), \lambda\rangle} u=q^{\langle h, j^*(\lambda)\rangle} u
\]
 for all $h\in P^*$.
 Hence, $\Psi_{\varepsilon}^*M$ has the weight space decomposition $\Psi_{\varepsilon}^*M=\bigoplus_{\lambda\in \breve{P}}M_{j^*(\lambda)}$.

Next, we take $u\in \Psi_{\varepsilon}^*M$ and prove that $E_i^{(m)}u=F_i^{(m)}u=0$ for sufficiently large $m$. For $i\in I_0$ this follows immediately from the integrability of $M$. For $i=0$, we have 
\[
E_0^{(m)}u=T_n^{\varepsilon}\left(\breve{E}_0^{(m)}\right)u=T_n^{\varepsilon}\left(\breve{E}_0^{(m)}T_n^{-\varepsilon}(u)\right)
\]
and
\[
F_0^{(m)}u=T_n^{\varepsilon}\left(\breve{F}_0^{(m)}\right)u=T_n^{\varepsilon}\left(\breve{F}_0^{(m)}T_n^{-\varepsilon}(u)\right).
\] Thus, by the integrability of $M$, we conclude that $E_0^{(m)}u = F_0^{(m)}u = 0$ for sufficiently large $m$.
\end{proof}

Let $M$ be an $U_q(\widehat{\mathfrak{sl}}_{n+1})$-module and let $u\in M$ be a weight vector. To avoid confusion, if $u$ is an extremal vector as an element of $M$, we say that $u$ is $\widehat{\mathfrak{sl}}_{n+1}$-extremal. Similarly, if $u$ is an extremal vector as an element of $\Psi_{\varepsilon}^*M$, we say $u$ is $\widehat{\mathfrak{sl}}_n$-extremal.

\begin{proposition}\label{ex}
  \textit{Let $M$ be an integrable $U_q(\widehat{\mathfrak{sl}}_{n+1})$-module and let $u\in M$ be a $\widehat{\mathfrak{sl}}_{n+1}$-extremal vector of weight $\lambda\in \breve{P}$. Then, $u\in \Psi_{\varepsilon}^*M$ is a $\widehat{\mathfrak{sl}}_n$-extremal vector of weight $j^*(\lambda)$.}
\end{proposition}

\begin{proof}
  It suffices to show the following two properties:
  \begin{enumerate}
  \item If $v\in M$ is a $\widehat{\mathfrak{sl}}_{n+1}$-extremal vector of weight $\mu\in \breve{P}$, then $v$ is $i$-extremal for all $i\in I$ as an element of $\Psi_{\varepsilon}^*M$;
    \item For all $\widehat{\mathfrak{sl}}_{n+1}$-extremal vectors $v\in M$ of weight $\mu\in \breve{P}$, 
    \[
    S_iv=\left\{
\begin{array}{ll}
F_i^{(\langle \alpha_i^{\vee},j^*(\mu)\rangle)}v & (\langle \alpha_i^{\vee},j^*(\mu)\rangle\ge 0) \\
E_i^{(-\langle \alpha_i^{\vee},j^*(\mu)\rangle)}v & (\langle \alpha_i^{\vee},j^*(\mu)\rangle\le 0)
\end{array}
\right.
\]
is also a $\widehat{\mathfrak{sl}}_{n+1}$-extremal vector.
  \end{enumerate}

Let us start by proving the first assertion. If $i\neq 0$, the assertion is clear since $\langle \alpha_i^{\vee}, j^*(\mu)\rangle=\langle \breve{\alpha}_i^{\vee}, \mu\rangle$, and $\Phi_{\varepsilon}(E_i)=\breve{E}_i$. Now, let us assume that $i=0$. Then we have $\langle \alpha_0^{\vee}, j^*(\mu)\rangle=\langle j(\alpha_0^{\vee}), \mu\rangle=\langle \breve{\alpha}_0^{\vee}, s_n(\mu)\rangle$. If $\langle \alpha_0^{\vee}, j^*(\mu)\rangle=\langle \breve{\alpha}_0^{\vee}, s_n(\mu)\rangle\ge 0$, then we have 
\[
E_0v=T_n^{\varepsilon}(E_0)(v)=T_n^{\varepsilon}(\breve{E}_0T_n^{-\varepsilon}(v))=0
\]
since $T_n^{-\varepsilon}(v)$ is a $\widehat{\mathfrak{sl}}_{n+1}$-extremal vector of weight $s_n(\mu)$. Similarly, we can deduce that $F_0v=0$ when $\langle \alpha_0^{\vee}, j^*(\mu)\rangle=\langle \breve{\alpha}_0^{\vee}, s_n(\mu)\rangle\le 0$.

Next, we prove the second assertion. If $i\neq 0$, then using $\langle \alpha_i^{\vee},j^*(\mu)\rangle=\langle \breve{\alpha}_i^{\vee}, \mu\rangle, \Psi_{\varepsilon}(E_i)=\breve{E}_i$, and $\Psi_{\varepsilon}(F_i)=\breve{F}_i$, we have
\begin{align*}
  S_iv&=\left\{
\begin{array}{ll}
F_i^{(\langle \alpha_i^{\vee},j^*(\mu)\rangle}v & (\langle \alpha_i^{\vee},j^*(\mu)\rangle\ge 0) \\
E_i^{(-\langle \alpha_i^{\vee},j^*(\mu)\rangle)}v & (\langle \alpha_i^{\vee},j^*(\mu)\rangle\le 0)
\end{array}
\right.\\
&=\left\{
\begin{array}{ll}
\breve{F}_i^{(\langle \breve{\alpha}_i^{\vee},\mu)\rangle}v & (\langle \alpha_i^{\vee},j^*(\mu)\rangle\ge 0) \\
\breve{E}_i^{(-\langle \breve{\alpha}_i^{\vee},\mu)\rangle)}v & (\langle \alpha_i^{\vee},j^*(\mu)\rangle\le 0)
\end{array}
\right.\\
&=\breve{S}_iv,
\end{align*}
which is a $\widehat{\mathfrak{sl}}_{n+1}$-extremal vector. Now, let us assume that $i=0$. If $\langle \alpha_0^{\vee}, j^*(\mu)\rangle=\langle \breve{\alpha}_0^{\vee}, s_n(\mu)\rangle\ge 0$, we have
\begin{align*}
  F_0^{(\langle \alpha_0^{\vee}, j^*(\mu) \rangle)}v
  &=T_n^{\varepsilon}\left(\breve{F}_0^{(\langle \alpha_0^{\vee}, j^*(\mu) \rangle)}\right)(v)=T_n^{\varepsilon}\left(\breve{F}_0^{(\langle \alpha_0^{\vee}, j^*(\mu)\rangle)}T_n^{-\varepsilon}(v)\right)\\
  &=T_n^{\varepsilon}\left(\breve{F}_0^{(\langle \breve{\alpha}_0^{\vee}, s_n(\mu)\rangle)}T_n^{-\varepsilon}(v)\right)=T_n^{\varepsilon}\left(\breve{S}_0T_n^{-\varepsilon}(v)\right).
\end{align*}
By (\ref{T1}) and (\ref{T2}), $F_0^{(\langle \alpha_0^{\vee}, j^*(\mu) \rangle)}v=(-q)^a\breve{S}_n\breve{S}_0\breve{S}_nv$ for some $a\in\mathbb{Z}$.  Hence, $F_0^{(\langle \alpha_0^{\vee}, j^*(\mu) \rangle)}v$ is a $\widehat{\mathfrak{sl}}_{n+1}$-extremal vector, as desired. In the case $\langle \alpha_0^{\vee}, j^*(\mu)\rangle\le 0$, we can show that $E_0^{(-\langle \alpha_0^{\vee},j^*(\mu)\rangle)}v$ is a $\widehat{\mathfrak{sl}}_{n+1}$-extremal vector in the same manner. These complete the proof.
\end{proof}

\begin{lemma}
  \textit{There is a group homomorphism $\omega:W\to \breve{W}$ such that
  \[
  \omega(s_i)=s_i\quad(i\in I_0),\quad\omega(s_0)=s_ns_0s_n=s_{\breve{\alpha}_0+\breve{\alpha}_n}.
  \]}
\end{lemma}

\begin{proof}
  The only non-trivial relations are $\omega(s_1)\omega(s_0)\omega(s_1)=\omega(s_0)\omega(s_1)\omega(s_0)$ and $\omega(s_{n-1})\omega(s_0)\omega(s_{n-1})=\omega(s_0)\omega(s_{n-1})\omega(s_0)$ when $n>2$. We have
\begin{gather*}
  \omega(s_1)\omega(s_0)\omega(s_1)=s_1s_{\breve{\alpha}_0+\breve{\alpha}_n}s_1=s_{\breve{\alpha}_0+\breve{\alpha}_1+\breve{\alpha}_n},\\
  \omega(s_0)\omega(s_1)\omega(s_0)=s_{\breve{\alpha}_0+\breve{\alpha}_n}s_1s_{\breve{\alpha}_0+\breve{\alpha}_n}=s_{\breve{\alpha}_0+\breve{\alpha}_1+\breve{\alpha}_n}.
\end{gather*}

  Thus, we obtain
  $\omega(s_1)\omega(s_0)\omega(s_1)=\omega(s_0)\omega(s_1)\omega(s_0)$ as required. The equation $\omega(s_{n-1})\omega(s_0)\omega(s_{n-1})=\omega(s_0)\omega(s_{n-1})\omega(s_0)$ can be proved in a similar manner.
\end{proof}

Let $M$ be an integrable $U_q(\widehat{\mathfrak{sl}}_{n+1})$-module and let $u\in M$ be a $\widehat{\mathfrak{sl}}_{n+1}$-extremal vector. By the proof of Lemma \ref{ex}, for all $i\in I$, we have $S_iu=(-q)^a\breve{S}_{\omega(s_i)}u$ for some $a\in \mathbb{Z}$. Hence, we have the following lemma.

\begin{lemma}\label{wact}
  \textit{Let $M$ be an integrable $U_q(\widehat{\mathfrak{sl}}_{n+1})$-module. For a $\widehat{\mathfrak{sl}}_{n+1}$-extremal vector $u$ of weight $\mu$ and $x\in W$, we have $S_xu=(-q)^a\breve{S}_{\omega(x)}u$ for some $a\in \mathbb{Z}$. \qed}
\end{lemma}

\begin{lemma}\label{tal}
  \textit{For $i\in I_0$, we have $\omega(t_{\alpha_i^{\vee}})=t_{\breve{\alpha}_i^{\vee}}$.}
\end{lemma}

\begin{proof}
  Let $\theta, \breve{\theta}$ be the highest root of $\mathfrak{sl}_n, \mathfrak{sl}_{n+1}$, respectively. Then we have $r_0r_{\theta}=t_{\theta^{\vee}}$ (in $W$) and $r_0r_{\breve{\theta}}=t_{\breve{\theta}^{\vee}}$ (in $\breve{W}$). Hence, we obtain
  \begin{align*}
    \omega(t_{\theta^{\vee}})&=\omega(r_0r_{\theta})=\omega(r_0)\omega(r_1r_2\cdots r_{n-2}r_{n-1}r_{n-2}\cdots r_1)\\
    &=r_nr_0r_nr_1r_2\cdots r_{n-2}r_{n-1}r_{n-2}\cdots r_1=r_nr_0r_nr_{\breve{\alpha}_1+\cdots+\breve{\alpha}_{n-1}}\\
    &=r_nr_0r_{\breve{\theta}}r_n=r_nt_{\breve{\theta}^{\vee}}r_n=t_{\breve{\alpha}_1^{\vee}+\cdots+\breve{\alpha}_{n-1}^{\vee}}.
  \end{align*}
  We take $w\in W_0$ such that $\alpha_i^{\vee}=w(\theta^{\vee})$. Then we have $\breve{\alpha}_i^{\vee}=\omega(w)(\breve{\alpha}_1^{\vee}+\cdots+\breve{\alpha}_{n-1}^{\vee})$. Hence,
  \begin{align*}
    \omega(t_{\alpha_i^{\vee}})&=\omega(wt_{\theta^{\vee}}w^{-1})=\omega(w)\omega(t_{\theta}^{\vee})\omega(w^{-1})\\
    &=\omega(w)t_{\breve{\alpha}_1^{\vee}+\cdots+\breve{\alpha}_{n-1}^{\vee}}\omega(w^{-1})=t_{\breve{\alpha}_i^{\vee}},
  \end{align*}
  as desired.
\end{proof}

We set
\begin{equation}
  w_i=s_ns_{n-1}\cdots s_i\in \breve{W}. \label{wi}
\end{equation}

We define $a_{i,\varepsilon}\in \mathbb{Z}\;(2\le i\le n)$ and $b_{i,\varepsilon}\in \mathbb{Z}\;(1\le i\le n-1)$ as follows.

For $2\le i\le n$, the element $\breve{S}_{w_i}u_{\breve{\varpi}_i}\in \Psi_{\varepsilon}^*V(\breve{\varpi}_i)$ is $\widehat{\mathfrak{sl}}_n$-extremal. Hence, by Lemmas \ref{wact} and \ref{tal}, we can define $a_{i,\varepsilon}\in \mathbb{Z}$ by
\[
S_{t_{\alpha_{i-1}^{\vee}}}\breve{S}_{w_i}u_{\breve{\varpi}_i}=(-q)^{a_{i,\varepsilon}}\breve{S}_{t_{\breve{\alpha}_{i-1}^{\vee}}}\breve{S}_{w_i}u_{\breve{\varpi}_i}\left(=(-q)^{a_{i,\varepsilon}}\breve{S}_{w_i}\breve{S}_{t_{\breve{\alpha}_i^{\vee}}}u_{\breve{\varpi}_i}\right).
\]

For $1\le i\le n-1$, the element $u_{\breve{\varpi}_i}\in \Psi_{\varepsilon}^*V(\breve{\varpi}_i)$ is $\widehat{\mathfrak{sl}}_n$-extremal. Hence, by Lemmas \ref{wact} and \ref{tal}, we can define $b_{i,\varepsilon}\in \mathbb{Z}$ by
\[
S_{t_{\alpha_i^{\vee}}}u_{\breve{\varpi}_i}=(-q)^{b_{i,\varepsilon}}\breve{S}_{t_{\breve{\alpha}_i^{\vee}}}u_{\breve{\varpi}_i}.
\]

\begin{lemma}\label{tk}
\textit{We have:
  \begin{enumerate}
    \item For $2\le i\le n$ and $k\in \mathbb{Z}$, we have 
    \[
    S_{t_{k\alpha_{i-1}^{\vee}}}\breve{S}_{w_i}u_{\breve{\varpi}_i}=(-q)^{ka_{i,\varepsilon}}\breve{S}_{w_i}\breve{S}_{kt_{\breve{\alpha}_i^{\vee}}}u_{\breve{\varpi}_i};
    \]
    \item For $1\le i\le n-1$ and $k\in \mathbb{Z}$, we have
    \[
    S_{t_{k\alpha_i^{\vee}}}u_{\breve{\varpi}_i}=(-q)^{kb_{i,\varepsilon}}\breve{S}_{t_{k\breve{\alpha}_i^{\vee}}}u_{\breve{\varpi}_i}.
    \]
  \end{enumerate}
  }
\end{lemma}

\begin{proof}

For $k\in \mathbb{Z}$, there is a $U_q(\widehat{\mathfrak{sl}}_{n+1})$-linear isomorphism
\[
\Upsilon:V(\breve{\varpi}_i)\stackrel{\cong}{\longrightarrow}V(\breve{\varpi}_i)\otimes V\left(-k\breve{\delta}\right)
\]
such that $u_{\breve{\varpi}_i}\mapsto \left(\breve{S}_{t_{-k\breve{\alpha}_i^{\vee}}}u_{\breve{\varpi}_i}\right)\otimes u_{-k\breve{\delta}}$ by Lemma \ref{txi}. We have
\[
 \Upsilon\left(S_{t_{\alpha_{i-1}^{\vee}}}\breve{S}_{w_i}\breve{S}_{t_{k\breve{\alpha}_i^{\vee}}}u_{\breve{\varpi}_i}\right)
  =\left(S_{t_{\alpha_{i-1}^{\vee}}}\breve{S}_{w_i}u_{\breve{\varpi}_i}\right)\otimes u_{-k\breve{\delta}}=\left((-q)^{a_{i,\varepsilon}}\breve{S}_{w_i}\breve{S}_{t_{\breve{\alpha}_i^{\vee}}}u_{\breve{\varpi}_i}\right)\otimes u_{-k\breve{\delta}}
\]
and
\[
\Upsilon\left(\breve{S}_{w_i}\breve{S}_{t_{(k+1)\breve{\alpha}_i^{\vee}}}u_{\breve{\varpi}_i}\right)=\left(\breve{S}_{w_i}\breve{S}_{t_{\breve{\alpha}_i^{\vee}}}u_{\breve{\varpi}_i}\right)\otimes u_{-k\breve{\delta}}.
\]
Thus, we obtain
\begin{equation}
  S_{t_{\alpha_{i-1}^{\vee}}}\breve{S}_{w_i}\breve{S}_{t_{k\breve{\alpha}_i^{\vee}}}u_{\breve{\varpi}_i}=(-q)^{a_{i,\varepsilon}}\breve{S}_{w_i}\breve{S}_{t_{(k+1)\breve{\alpha}_i^{\vee}}}u_{\breve{\varpi}_i}.\label{eq_k1}
\end{equation}

Similarly, we have
\[
\Upsilon\left(S_{t_{\alpha_i^{\vee}}}\breve{S}_{t_{k\breve{\alpha}_{i-1}^{\vee}}}u_{\breve{\varpi}_i}\right)=S_{t_{\alpha_i^{\vee}}}u_{\breve{\varpi}_i}
=(-q)^{b_{i,\varepsilon}}\breve{S}_{t_{\breve{\alpha}_i^{\vee}}}u_{\breve{\varpi}_i}
\]
and
\[
\Upsilon\left(\breve{S}_{t_{(k+1)\breve{\alpha}_i^{\vee}}}u_{\breve{\varpi}_i}\right)=\breve{S}_{t_{\breve{\alpha}_i^{\vee}}}u_{\breve{\varpi}_i}.
\]
Therefore, we obtain 
\begin{equation}
  S_{t_{\alpha_i^{\vee}}}\breve{S}_{t_{k\breve{\alpha}_{i-1}^{\vee}}}u_{\breve{\varpi}_i}=(-q)^{b_{i,\varepsilon}}\breve{S}_{t_{(k+1)\breve{\alpha}_i^{\vee}}}u_{\breve{\varpi}_i}.\label{eq_k2}
\end{equation}

Using (\ref{eq_k1}) and (\ref{eq_k2}) inductively, we obtain the desired result.
\end{proof}

We define the $U_q(\widehat{\mathfrak{sl}}_n)$-modules $N_{i,1}, N_{i,2}$ by
\[
N_{i,1}=\left\{
\begin{array}{ll}
\bigoplus_{k \in \mathbb{Z}} V(k \delta) & (i=1) \\
V(\varpi_{i-1}) & (2\le i \le n)
\end{array}
\right.,
\]

\[
N_{i,2}=\left\{
\begin{array}{ll}
V(\varpi_i) & (1\le i \le n-1) \\
\bigoplus_{k \in \mathbb{Z}} V(k \delta) & (i=n)
\end{array}
\right..
\]

We define the $U_q(\widehat{\mathfrak{sl}}_n)$-homomorphisms $\psi_{1,\varepsilon}^i:N_{i,1}\to \Psi_{\varepsilon}^*V(\breve{\varpi}_i)\;(1\le i\le n, \varepsilon\in\{\pm1\})$ as follows:
\begin{itemize}
  \item For $i=1$, the $U_q(\widehat{\mathfrak{sl}}_n)$-homomorphism $\psi_{1,\varepsilon}^1:N_{1,1}\to \Psi_{\varepsilon}^*V(\breve{\varpi}_1)$ is defined by $u_{k\delta}\mapsto u_{-\breve{\varpi}_n+k\breve{\delta}}=\breve{S}_{w_1t_{-k\breve{\alpha}_1^{\vee}}}u_{\breve{\varpi}_1}$;
  \item For $2\le i\le n$, the $U_q(\widehat{\mathfrak{sl}}_n)$-homomorphism $\psi_{1,\varepsilon}^i:N_{i,1}\to \Psi_{\varepsilon}^*V(\breve{\varpi}_i)$ is defined by $u_{\varpi_{i-1}}\mapsto u_{\breve{\varpi}_{i-1}-\breve{\varpi}_n}=\breve{S}_{w_i}u_{\breve{\varpi}_i}$.
\end{itemize}

We also define the $U_q(\widehat{\mathfrak{sl}}_n)$-homomorphisms $\psi_{2,\varepsilon}^i:N_{i,2}\to \Psi_{\varepsilon}^*V(\breve{\varpi}_i)\;(1\le i\le n, \varepsilon\in\{\pm1\})$ as follows:
\begin{itemize}
  \item For $1\le i\le n-1$, the $U_q(\widehat{\mathfrak{sl}}_n)$-homomorphism $\psi_{2,\varepsilon}^i:N_{i,2}\to \Psi_{\varepsilon}^*V(\breve{\varpi}_i)$ is defined by $u_{\varpi_i}\mapsto u_{\breve{\varpi}_i}$;
  \item For $i=n$, the $U_q(\widehat{\mathfrak{sl}}_n)$-homomorphism $\psi_{2,\varepsilon}^n:N_{n,2}\to \Psi_{\varepsilon}^*V(\breve{\varpi}_i)$ is defined by $u_{k\delta}\mapsto  u_{\breve{\varpi}_n+k\breve{\delta}}=\breve{S}_{t_{-k\breve{\alpha}_n^{\vee}}}u_{\breve{\varpi}_n}$.
\end{itemize}

For $2\le i\le n$ and $k\in \mathbb{Z}$, we have
\begin{equation}\label{eq:psi1}
  \begin{aligned}
    \psi_{1,\varepsilon}^i(u_{\varpi_{i-1}+k\delta})&=\psi_{1,\varepsilon}^i(S_{t_{-k\alpha_{i-1}^{\vee}}}u_{\varpi_{i-1}})=S_{t_{-k\alpha_{i-1}^{\vee}}}\breve{S}_{w_i}u_{\breve{\varpi}_i}\\
    &=(-q)^{-ka_{i,\varepsilon}}\breve{S}_{w_i}\breve{S}_{t_{-k\breve{\alpha}_i^{\vee}}}u_{\breve{\varpi}_i}=(-q)^{-ka_{i,\varepsilon}}\breve{S}_{w_i}u_{\breve{\varpi}_i+k\breve{\delta}}
  \end{aligned}
\end{equation}
by Lemma \ref{tk}.

Also, for $1\le i\le n-1$, we have
\begin{equation}\label{eq:psi2}
  \begin{aligned}
    \psi_{2,\varepsilon}^i(u_{\varpi_i+k\delta})&=\psi_{2,\varepsilon}^i(S_{t_{-k\alpha_i^{\vee}}}u_{\varpi_i})=S_{t_{-k\alpha_i^{\vee}}}u_{\breve{\varpi}_i}\\
    &=(-q)^{-kb_{i,\varepsilon}}\breve{S}_{t_{-k\breve{\alpha}_i^{\vee}}}u_{\breve{\varpi}_i}=(-q)^{-kb_{i,\varepsilon}}u_{\breve{\varpi}_i+k\breve{\delta}}
  \end{aligned}
\end{equation}
 by Lemma \ref{tk}.

\begin{lemma}\label{brfin}
  \textit{Let $\Lambda\in \mathcal{P}_n$ and set $\lambda=\pi_{n+1}(\Lambda)$. We denote the restriction of $V(\lambda)$ as a $U_q(\mathfrak{sl}_{n+1})$-module to $U_q(\mathfrak{sl}_n)$ by $V(\lambda)|_{U_q(\mathfrak{sl}_n)}$. Then $V(\lambda)|_{U_q(\mathfrak{sl}_n)}$ decomposes as
  \[
  V(\lambda)|_{U_q(\mathfrak{sl}_n)}\cong\bigoplus_{\substack{M\in \mathcal{P}_n \\ \Lambda/M:\text{horizontal strip}}}V(\pi_n(M)).
  \]}
\end{lemma}

\begin{proof}
  By \cite[Corollary 6.2.3]{L}, the module $V(\lambda)|_{U_q(\mathfrak{sl}_n)}$ decomposes into a direct sum of irreducible highest weight modules over $U_q(\mathfrak{sl}_n)$. Furthermore, by \cite[Theorem 3.4.6]{HK}, the character of $V(\lambda)$ is given by $\Pi_{n+1}(s_{\Lambda}(x_1, \ldots, x_{n+1}))$, which can be expressed as  
\[
\sum_{\substack{M \in \mathcal{P}_n \\ \Lambda / M : \text{horizontal strip}}} x_{n+1}^{|\Lambda - M|} s_M(x_1, \ldots, x_n)
\]  
by Corollary \ref{schbr}.
Since $j^*(\pi_{n+1}(\varepsilon_{n+1})) = 0$ and the restriction of $j^* \circ \pi_{n+1}$ to $X_n$ coincides with $\pi_n$, we have $\Theta\circ\Pi_{n+1}\left(x_{n+1}^{|\Lambda - M|} s_M(x_1, \ldots, x_n)\right)=\Pi_n(s_M(x_1, \ldots, x_n)))$. Hence, the character of $V(\lambda)|_{U_q(\mathfrak{sl}_n)}$ is  
\[
\sum_{\substack{M \in \mathcal{P}_n \\ \Lambda / M : \text{horizontal strip}}} \Pi_n(s_M(x_1, \ldots, x_n)),
\]  
which establishes the desired result.
\end{proof}

\begin{proposition}\label{b1}
\textit{For $1\le i\le n$, $\psi_{1,\varepsilon}^i\oplus\psi_{2,\varepsilon}^i:N_{i,1}\oplus N_{i,2}\to \Psi_{\varepsilon}^*V(\breve{\varpi}_i)$ is an isomorphism of $U_q(\widehat{\mathfrak{sl}}_n)$-modules.}
\end{proposition}

\begin{proof}
For $1\le i\le n-1$, we have 
\[
V(\varpi_i)=\bigoplus_{k\in \mathbb{Z}}U_q(\mathfrak{sl}_n)u_{\varpi_i+k\delta}
\]
as $U_q(\mathfrak{sl}_n)$-module, where each $U_q(\mathfrak{sl}_n)u_{\varpi_i+k\delta}$ is the irreducible highest weight representation of highest weight $\cl(\varpi_i)$ by Theorem \ref{f}. Similarly, for $1\le i\le n$, we have 
\[
V(\breve{\varpi}_i)=\bigoplus_{k\in \mathbb{Z}}U_q(\mathfrak{sl}_{n+1})u_{\breve{\varpi}_i+k\breve{\delta}}
\]
as $U_q(\mathfrak{sl}_{n+1})$-module, where each $U_q(\mathfrak{sl}_{n+1})u_{\breve{\varpi}_i+k\breve{\delta}}$ is the irreducible highest weight representation of highest weight $\cl(\breve{\varpi}_i)$ by Theorem \ref{f}.

First, we deduce that $\psi_{1,\varepsilon}^i\oplus\psi_{2,\varepsilon}^i:N_{i,1}\oplus N_{i,2}\to \Psi_{\varepsilon}^*V(\breve{\varpi}_i)$ is an isomorphism in the case $i=1$. By (\ref{eq:psi2}) and the branching rule for the finite dimensional $U_q(\mathfrak{sl}_{n+1})$-module $U_q(\mathfrak{sl}_{n+1})u_{\breve{\varpi}_1+k\breve{\delta}}$ from $U_q(\mathfrak{sl}_{n+1})$ to $U_q(\mathfrak{sl}_n)$ (see Lemma \ref{brfin}), the restriction of $\psi_{1,\varepsilon}^1\oplus\psi_{2,\varepsilon}^1:N_{1,1}\oplus N_{1,2}\to \Psi_{\varepsilon}^*V(\breve{\varpi}_1)$ to $V(k\delta)\oplus U_q(\mathfrak{sl}_n)u_{\varpi_1+k\delta}$ gives an isomorphism of $U_q(\mathfrak{sl}_n)$-modules
\[
V(k\delta)\oplus U_q(\mathfrak{sl}_n)u_{\varpi_1+k\delta}\stackrel{\cong}{\longrightarrow} U_q(\mathfrak{sl}_{n+1})u_{\breve{\varpi}_1+k\breve{\delta}}.
\]
Thus, $\psi_{1,\varepsilon}^1\oplus\psi_{2,\varepsilon}^1:N_{1,1}\oplus N_{1,2}\to \Psi_{\varepsilon}^*V(\breve{\varpi}_1)$ is an isomorphism of $U_q(\widehat{\mathfrak{sl}}_n)$-modules.

Next, we deduce that $\psi_{1,\varepsilon}^i\oplus\psi_{2,\varepsilon}^i:N_{i,1}\oplus N_{i,2}\to \Psi_{\varepsilon}^*V(\breve{\varpi}_i)$ is an isomorphism in the case $2\le i\le n-1$.  By (\ref{eq:psi1}), (\ref{eq:psi2}), and the branching rule for the finite dimensional $U_q(\mathfrak{sl}_{n+1})$-module $U_q(\mathfrak{sl}_{n+1})u_{\breve{\varpi}_i+k\breve{\delta}}$ from $U_q(\mathfrak{sl}_{n+1})$ to $U_q(\mathfrak{sl}_n)$ (see Lemma \ref{brfin}), the restriction of $\psi_{1,\varepsilon}^i\oplus\psi_{2,\varepsilon}^i:N_{1,1}\oplus N_{1,2}\to \Psi_{\varepsilon}^*V(\breve{\varpi}_i)$ to $U_q(\mathfrak{sl}_n)u_{\varpi_{i-1}+k\delta}\oplus U_q(\mathfrak{sl}_n)u_{\varpi_i+k\delta}$ yields an isomorphism of $U_q(\mathfrak{sl}_n)$-modules
\[
U_q(\mathfrak{sl}_n)u_{\varpi_{i-1}+k\delta}\oplus U_q(\mathfrak{sl}_n)u_{\varpi_i+k\delta}\stackrel{\cong}{\longrightarrow} U_q(\mathfrak{sl}_{n+1})u_{\breve{\varpi}_i+k\breve{\delta}}.
\]
Thus, $\psi_{1,\varepsilon}^i\oplus\psi_{2,\varepsilon}^i:N_{1,1}\oplus N_{1,2}\to \Psi_{\varepsilon}^*V(\breve{\varpi}_i)$ is an isomorphism of $U_q(\widehat{\mathfrak{sl}}_n)$-modules.

Finally, we deduce that $\psi_{1,\varepsilon}^i\oplus\psi_{2,\varepsilon}^i:N_{i,1}\oplus N_{i,2}\to \Psi_{\varepsilon}^*V(\breve{\varpi}_i)$ is an isomorphism in the case $i=n$.  By (\ref{eq:psi1}) and the branching rule for the finite dimensional $U_q(\mathfrak{sl}_{n+1})$-module $U_q(\mathfrak{sl}_{n+1})u_{\breve{\varpi}_n+k\breve{\delta}}$ from $U_q(\mathfrak{sl}_{n+1})$ to $U_q(\mathfrak{sl}_n)$ (see Lemma \ref{brfin}), the restriction of $\psi_{1,\varepsilon}^n\oplus\psi_{2,\varepsilon}^n:N_{1,1}\oplus N_{1,2}\to \Psi_{\varepsilon}^*V(\breve{\varpi}_n)$ to $U_q(\mathfrak{sl}_n)u_{\varpi_{n-1}+k\delta}\oplus V(k\delta)$ gives an isomorphism of $U_q(\mathfrak{sl}_n)$-modules
\[
U_q(\mathfrak{sl}_n)u_{\varpi_{n-1}+k\delta}\oplus V(k\delta)\stackrel{\cong}{\longrightarrow} U_q(\mathfrak{sl}_{n+1})u_{\breve{\varpi}_n+k\breve{\delta}}.
\]
Thus, $\psi_{1,\varepsilon}^n\oplus\psi_{2,\varepsilon}^n:N_{1,1}\oplus N_{1,2}\to \Psi_{\varepsilon}^*V(\breve{\varpi}_n)$ is an isomorphism of $U_q(\widehat{\mathfrak{sl}}_n)$-modules.
\end{proof}

\begin{corollary}
  \textit{For $1\le i\le n$, there exist an isomorphism of $U_q^{\prime}(\widehat{\mathfrak{sl}}_n)$-modules
  \[
(\Psi_{\varepsilon}^{\prime})^*W(\breve{\varpi}_i)\cong\left\{
\begin{array}{ll}
V(0)\oplus W(\varpi_1)_{(-q)^{-b_{1,\varepsilon}}} & (i=1) \\
W(\varpi_{i-1})_{(-q)^{-a_{i,\varepsilon}}}\oplus W(\varpi_i)_{(-q)^{-b_{i,\varepsilon}}} & (2\le i \le n-1) \\
W(\varpi_{n-1})_{(-q)^{-a_{n,\varepsilon}}}\oplus V(0) & (i=n)
\end{array}
\right..
\]
}
\end{corollary}

\begin{proof}
This result follows immediately from (\ref{eq:psi1}), (\ref{eq:psi2}), and Proposition \ref{b1}.
\end{proof}

Let $\mathbb{Q}(q)[t_1^{\pm1},\ldots,t_m^{\pm1}]$ be the ring of Laurent polynomials. We equip $\mathbb{Q}(q)[t_1^{\pm1},\ldots,t_m^{\pm1}]$ with the $U_q(\widehat{\mathfrak{sl}}_n)$-module structure as follows:
\begin{itemize}
  \item $E_i$ and $F_i$ act trivially;
  \item the monomial $t_1^{k_1}\cdots t_l^{k_m}\;(k_1\ldots,k_m \in \mathbb{Z})$ has weight $(k_1+\cdots+k_l)\delta$.
\end{itemize}

The $m$-th symmetric group $\mathfrak{S}_m$ acts on $\mathbb{Q}(q)[t_1^{\pm1},\ldots,t_m^{\pm1}]$ and on $\mathbb{Q}(q)[t_1^{-1},\ldots,t_m^{-1}](\subset\mathbb{Q}(q)[t_1^{\pm1},\ldots,t_m^{\pm1}])$ by permuting variables. The set of fixed points by these actions are denoted by $\mathbb{Q}(q)[t_1^{\pm1},\ldots,t_m^{\pm1}]^{\mathfrak{S}_m}$ and $\mathbb{Q}(q)[t_1^{-1},\ldots,t_m^{-1}]^{\mathfrak{S}_m}$, respectively.

We identify $\mathbb{Q}(q)[t^{\pm1}]$ with $\bigoplus_{k\in \mathbb{Z}}V(k\delta)$ via the map $t^k\mapsto u_{k\delta}$.

Then Proposition \ref{b1} can also be expressed as follows.

 \[
\Psi_{\varepsilon}^*V(\breve{\varpi}_i)\cong\left\{
\begin{array}{ll}
\mathbb{Q}(q)[t^{\pm1}]\oplus V(\varpi_1) & (i=1) \\
V(\varpi_{i-1})\oplus V(\varpi_i) & (2\le i \le n-1) \\
\mathbb{Q}(q)[t^{\pm1}]\oplus V(\varpi_{n-1}) & (i=n)
\end{array}
\right..
\]

\subsection{A direct sum decomposition of $\Psi_{\varepsilon}^*V(\lambda)$}

In this subsection, we fix $\lambda=\sum_{i\in \breve{I}_0}m_i\breve{\varpi}_i\in \breve{P}_{0,+}$.

We set $L_{1,\varepsilon}^i=\psi_{1,\varepsilon}^i(N_{i,1})$ and $L_{2,\varepsilon}^i=\psi_{2,\varepsilon}^i(N_{i,2})$. Let $m=-\langle\breve{\alpha}_0^{\vee},\lambda\rangle=\sum_{i\in \breve{I}_0}m_i$, and we define the map $\chi_{\lambda}:\{1,\ldots,m\}\to \{1,\ldots,n\}$ so that 
\[
m_1+\cdots+m_{\chi_{\lambda}(j)-1}+1\le j\le m_1+\cdots+m_{\chi_{\lambda}(j)}
\]
for each $j\in \{1,\ldots,m\}$.

Define $\mathcal{I}=\{1,2\}^m$. For each $\mathfrak{i}=(i_1,\ldots,i_m)\in \mathcal{I}$, let $L_{\mathfrak{i},\varepsilon}$ denote the $U_q(\mathfrak{sl}_n)$-submodule of $\Psi_{\varepsilon}^*\left(\bigotimes_{i\in \breve{I}_0}V(\breve{\varpi}_i)^{\otimes m_i}\right)$ defined by 
\[
L_{\mathfrak{i},\varepsilon}=\bigotimes_{j=1}^m L_{i_j,\varepsilon}^{\chi_{\lambda}(j)}.
\]

Note that $L_{\mathfrak{i},\varepsilon}$ is \textit{not} a $U_q(\widehat{\mathfrak{sl}}_n)$-module in general.

Then we have a direct sum decomposition
\begin{equation}
  \Psi_{\varepsilon}^*\left(\bigotimes_{i\in \breve{I}_0}V(\breve{\varpi}_i)^{\otimes m_i}\right)=\bigoplus_{\mathfrak{i}\in \mathcal{I}}L_{\mathfrak{i},\varepsilon} \label{Liep}
\end{equation}
as a $U_q(\mathfrak{sl}_n)$-module.

For $0\le p\le m$, we define $\mathcal{I}_p$ by
\[
\mathcal{I}_p=\{\mathfrak{i}=(i_1,\ldots,i_m)\in \mathcal{I} \mid p=\#\{j \mid i_j=1\}\}.
\]

Let $L_{p,\varepsilon}$ denote the $U_q(\mathfrak{sl}_n)$-submodule of $\Psi_{\varepsilon}^*\left(\bigotimes_{i\in \breve{I}_0}V(\breve{\varpi}_i)^{\otimes m_i}\right)$ defined by 
\[
L_{p,\varepsilon}=\bigoplus_{\mathfrak{i}\in \mathcal{I}_p} L_{\mathfrak{i},\varepsilon}.
\]

Then we have a direct sum decomposition
\begin{equation}
  \Psi_{\varepsilon}^*\left(\bigotimes_{i\in \breve{I}_0}V(\breve{\varpi}_i)^{\otimes m_i}\right)=\bigoplus_{p=0}^mL_{p,\varepsilon} \label{Lpep}
\end{equation}
as a $U_q(\mathfrak{sl}_n)$-module by (\ref{Liep}).

We define $\tilde{h}\in \breve{P}^*$ by 
\[
\tilde{h}=\sum_{i=1}^n i\breve{\alpha}_i^{\vee}.
\]

We have
\begin{equation}
 \langle \tilde{h}, \breve{\alpha}_i^{\vee}\rangle=\left\{
\begin{array}{ll}
-(n+1) & (i=0) \\
0 & (1\le i\le n-1) \\
n+1 & (i=n)
\end{array}
\right.. \label{tildeh}
\end{equation}

\begin{lemma}\label{qhcom}
 \textit{The element $q^{\tilde{h}}\in U_q(\widehat{\mathfrak{sl}}_{n+1})$ commutes with all elements in $\Psi_{\varepsilon}(U_q(\widehat{\mathfrak{sl}}_n))$.}
\end{lemma}

\begin{proof}
 Let $x \in U_q(\widehat{\mathfrak{sl}}_n)$ be an element with weight $\zeta = \sum_{i \in I_0} k_i \alpha_i \in Q$. By (\ref{tildeh}), it follows that 
 \[
 \langle \tilde{h}, \gamma(\zeta)\rangle =0.
 \]
 
 Since the weight of $\Psi_{\varepsilon}(x)$ is $\gamma(\zeta)$, we have
  \[
  q^{\tilde{h}}\Psi_{\varepsilon}(x)q^{-\tilde{h}}=\Psi_{\varepsilon}(x).
  \]
These complete the proof.
\end{proof}

\begin{lemma}\label{L1L2}
  \textit{Let $1\le i\le n$. As a $U_q(\widehat{\mathfrak{sl}}_n)$-submodule of $\Psi_{\varepsilon}^*V(\breve{\varpi}_i)$, we have
  \[
  L_{1,\varepsilon}^i=\{u\in \Psi_{\varepsilon}^*V(\breve{\varpi}_i) \mid q^{\tilde{h}}u=q^{i-(n+1)}u\},\quad
  L_{2,\varepsilon}^i=\{u\in \Psi_{\varepsilon}^*V(\breve{\varpi}_i) \mid q^{\tilde{h}}u=q^iu\}.
  \]
  }
\end{lemma}

\begin{proof}
 Note that by Lemma \ref{qhcom}, $\{u\in \Psi_{\varepsilon}^*V(\breve{\varpi}_i) \mid q^{\tilde{h}}u=q^{i-(n+1)}u\}$ and $\{u\in \Psi_{\varepsilon}^*V(\breve{\varpi}_i) \mid q^{\tilde{h}}u=q^iu\}$ are $U_q(\widehat{\mathfrak{sl}}_n)$-submodules of $\Psi_{\varepsilon}^*V(\breve{\varpi}_i)$.

First, consider the case $i=1$. For $k\in \mathbb{Z}$, we have
\[
q^{\tilde{h}}u_{-\breve{\varpi}_n+k\breve{\delta}}=q^{-n}u_{-\breve{\varpi}_n+k\delta},\quad q^{\tilde{h}}u_{\breve{\varpi}_1}=qu_{\breve{\varpi}_1}.
\]
As a result,
\[
L_{1,\varepsilon}^1\subset\{u\in \Psi_{\varepsilon}^*V(\breve{\varpi}_1) \mid q^{\tilde{h}}u=q^{-n}u\},\quad
L_{2,\varepsilon}^1\subset\{u\in \Psi_{\varepsilon}^*V(\breve{\varpi}_1) \mid q^{\tilde{h}}u=qu\}.
\]
By Proposition \ref{b1}, it follows that $\Psi_{\varepsilon}^*V(\breve{\varpi}_1)=L_{1,\varepsilon}^1\oplus L_{2,\varepsilon}^1$, confirming the desired result for $i=1$.

Next, consider the case $2 \leq i \leq n-1$. Here, we have
\[
q^{\tilde{h}}u_{\breve{\varpi}_{i-1}-\breve{\varpi}_n}=q^{i-(n+1)}u_{\breve{\varpi}_{i-1}-\breve{\varpi}_n},\quad q^{\tilde{h}}u_{\breve{\varpi}_i}=q^iu_{\breve{\varpi}_i}.
\]
Consequently,
\[
L_{1,\varepsilon}^i\subset\{u\in \Psi_{\varepsilon}^*V(\breve{\varpi}_i) \mid q^{\tilde{h}}u=q^{i-(n+1)}u\},\quad
L_{2,\varepsilon}^i\subset\{u\in \Psi_{\varepsilon}^*V(\breve{\varpi}_i) \mid q^{\tilde{h}}u=q^iu\}.
\]
Applying Proposition \ref{b1}, we deduce the decomposition $\Psi_{\varepsilon}^*V(\breve{\varpi}_i)=L_{1,\varepsilon}^i\oplus L_{2,\varepsilon}^i$, establishing the claim for $2 \leq i \leq n-1$.

Finally, consider the case $i=n$. For $k\in \mathbb{Z}$, we have
\[
q^{\tilde{h}}u_{\breve{\varpi}_{n-1}-\breve{\varpi}_n}=q^{-1}u_{\breve{\varpi}_{n-1}-\breve{\varpi}_n},\quad q^{\tilde{h}}u_{\breve{\varpi}_n}=q^nu_{\breve{\varpi}_n+k\breve{\delta}}.
\]
Thus,
\[
L_{1,\varepsilon}^n\subset\{u\in \Psi_{\varepsilon}^*V(\breve{\varpi}_n) \mid q^{\tilde{h}}u=q^{-1}u\},\quad
L_{2,\varepsilon}^n\subset\{u\in \Psi_{\varepsilon}^*V(\breve{\varpi}_n) \mid q^{\tilde{h}}u=q^nu\}.
\]
Proposition \ref{b1} then gives the decomposition $\Psi_{\varepsilon}^*V(\breve{\varpi}_n)=L_{1,\varepsilon}^n\oplus L_{2,\varepsilon}^n$, completing the proof for $i=n$.
\end{proof}

\begin{lemma}\label{0nact}
  \textit{Let $1\le i\le n$. The images of $L_{1,\varepsilon}^i, L_{2,\varepsilon}^i\subset V(\breve{\varpi}_i)$ by the actions of $\breve{F}_n$, $\breve{F}_0$, $\breve{E}_n$, and $\breve{E}_0$ satisfy the following relations:
  \[\left\{
\begin{array}{ll}
\breve{F}_n(L_{1,\varepsilon}^i)=0, &\breve{F}_n(L_{2,\varepsilon}^i)\subset L_{1,\varepsilon}^i, \\
\breve{F}_0(L_{1,\varepsilon}^i)\subset L_{2,\varepsilon}^i, & \breve{F}_0(L_{2,\varepsilon}^i)=0, \\
\breve{E}_n(L_{1,\varepsilon}^i)\subset L_{2,\varepsilon}^i, & \breve{E}_n(L_{2,\varepsilon}^i)=0, \\
\breve{E}_0(L_{1,\varepsilon}^i)=0, & \breve{E}_0(L_{2,\varepsilon}^i)\subset L_{1,\varepsilon}^i.
\end{array}
\right.
\]}
\end{lemma}

\begin{proof}
Combining Lemma \ref{L1L2} and (\ref{tildeh}), the assertion follows immediately.
\end{proof}

By Lemma \ref{L1L2}, we have
  \begin{equation}
    L_{p,\varepsilon}=\left\{u\in \Psi_{\varepsilon}^*\left(\bigotimes_{i\in \breve{I}_0}V(\breve{\varpi}_i)^{\otimes m_i}\right)\quad \middle| \quad q^{\tilde{h}}u=q^{\langle \tilde{h},\lambda\rangle-p(n+1)}u\right\}\label{Lpeps}
  \end{equation}
  for $0\le p\le m$, since $\sum_{i=1}^n im_i=\langle\tilde{h},\lambda\rangle$. By combining Lemma \ref{qhcom}, (\ref{Lpep}), and (\ref{Lpeps}), we obtain the following.

\begin{lemma}\label{dirsum1}
\textit{
  For $0\le p\le m$, $L_{p,\varepsilon}$ is a $U_q(\widehat{\mathfrak{sl}}_n)$-submodule of $\Psi_{\varepsilon}^*\left(\bigotimes_{i\in \breve{I}_0}V(\breve{\varpi}_i)^{\otimes m_i}\right)$. Moreover, we have a direct sum decomposition
  \[
  \Psi_{\varepsilon}^*\left(\bigotimes_{i\in \breve{I}_0}V(\breve{\varpi}_i)^{\otimes m_i}\right)= L_{0,\varepsilon}\oplus \cdots \oplus L_{m,\varepsilon}
  \]
  as a $U_q(\widehat{\mathfrak{sl}}_n)$-module.
  \qed}
\end{lemma}

For $p\in \mathbb{Z}$, we set
\[
M_{p,\varepsilon}=\{u\in \Psi_{\varepsilon}^*V(\lambda) \mid q^{\tilde{h}}u=q^{\langle \tilde{h},\lambda \rangle-(n+1)p}u\}.
\]
Note that each $M_{p,\varepsilon}$ is a $U_q(\widehat{\mathfrak{sl}}_n)$-submodule of $\Psi_{\varepsilon}V(\lambda)$ by Lemma \ref{qhcom}.

\begin{lemma}\label{dirsum1.5}
\textit{We have:
  \begin{enumerate}
    \item For a homogeneous element $x\in U_q(\widehat{\mathfrak{sl}}_{n+1})$ of weight $\sum_{i\in \breve{I}}k_i\breve{\alpha}_i$, we have $xu_{\lambda}\in M_{k_0-k_n,\varepsilon}$;
    \item There is a direct sum decomposition 
\[
\Psi_{\varepsilon}^*V(\lambda)=\bigoplus_{p\in \mathbb{Z}} M_{p,\varepsilon}
\]
as a $U_q(\widehat{\mathfrak{sl}}_n)$-module.
  \end{enumerate}
  }
\end{lemma}

\begin{proof}
  The first assertion is a direct consequence of (\ref{tildeh}), and the second follows immediately from the first.
\end{proof}

\begin{proposition}\label{dirsum2}
\textit{We have:
\begin{enumerate}
  \item If $0\le p\le m$, then $\Phi_{\lambda}(M_{p,\varepsilon})\subset L_{p,\varepsilon}$; otherwise, $\Phi_{\lambda}(M_{p,\varepsilon})=0$;
  \item There is a direct sum decomposition
  \[
  \Psi_{\varepsilon}^* V(\lambda)= M_{0,\varepsilon}\oplus\cdots\oplus M_{m,\varepsilon}
  \]
  as a $U_q(\widehat{\mathfrak{sl}}_n)$-module. Moreover, for $0\le p\le m$, a weight vector $u\in \Psi_{\varepsilon}^* V(\lambda)$ is contained in $M_{p,\varepsilon}$ if and only if the weight of $u$ is the form $\lambda+\sum_{i\in \breve{I}}k_i\breve{\alpha}_i$ with $k_0=k_n+p$. 
\end{enumerate}
}
\end{proposition}

\begin{proof}
 Let $x \in U_q(\widehat{\mathfrak{sl}}_{n+1})$ be a homogeneous element of weight $\sum_{i \in \breve{I}} k_i \breve{\alpha}_i$. From (\ref{tildeh}), it follows that
\[
q^{\tilde{h}}\Phi_{\lambda}(xu_{\lambda}) = q^{\langle\tilde{h}, \lambda\rangle - (k_0 - k_n)(n+1)}\Phi_{\lambda}(xu_{\lambda}).
\]
Therefore, by (\ref{Lpeps}), we conclude that $\Phi_{\lambda}(xu_{\lambda}) \in L_{k_0-k_n,\varepsilon}$ if $0\le k_0-k_n\le m$; otherwise $\Phi_{\lambda}(xu_{\lambda}) = 0$. This proves the first assertion. The second assertion follows directly from the first and Lemma \ref{dirsum1.5}.
\end{proof}

We define the lexicographic order $<$ on $\mathcal{I}_p$ as follows.

\begin{definition}
  For $\mathfrak{i} = (i_1, \ldots, i_m)$ and $\mathfrak{j} = (j_1, \ldots, j_m)$ in $\mathcal{I}_p$, we write $\mathfrak{i} < \mathfrak{j}$ if there exists an integer $1 \leq k \leq m$ such that $i_l = j_l$ for all $l < k$ and $i_k < j_k$. Note that the relation $<$ defines a total order on $\mathcal{I}_p$.
\end{definition}

For $\mathfrak{i}\in \mathcal{I}_p$, we define the $U_q(\mathfrak{sl}_n)$-submodules $L_{\le\mathfrak{i},\varepsilon}$ and $L_{\ge \mathfrak{i},\varepsilon}$ of $L_{p,\varepsilon}$ by
\[
L_{\le\mathfrak{i},\varepsilon}=\bigoplus_{\mathfrak{j}\in \mathcal{I}_p, \;{\mathfrak{j}\le\mathfrak{i}}}L_{\mathfrak{j},\varepsilon}\quad\text{and}\quad L_{\ge\mathfrak{i},\varepsilon}=\bigoplus_{\mathfrak{j}\in \mathcal{I}_p,\;{\mathfrak{j}\ge\mathfrak{i}}}L_{\mathfrak{j},\varepsilon}.
\]

\begin{proposition}\label{Lsub}
  \textit{Let $\mathfrak{i}\in \mathcal{I}_p$.
  \begin{enumerate}
    \item If $\varepsilon=1$, then $L_{\ge\mathfrak{i},\varepsilon}$ is a $U_q(\widehat{\mathfrak{sl}}_n)$-submodule of $L_{p,\varepsilon}$;
    \item If $\varepsilon=-1$, then $L_{\le\mathfrak{i},\varepsilon}$ is a $U_q(\widehat{\mathfrak{sl}}_n)$-submodule of $L_{p,\varepsilon}$.
  \end{enumerate}
  }
\end{proposition}

\begin{proof}
  First, we consider the case $\varepsilon=1$. We take $\mathfrak{j}=(j_1,\ldots,j_m)\in \mathcal{I}_p$ and $u=\otimes_{j=1}^m u_j\in L_{\mathfrak{j},1}$. We have
    \begin{align*}
      T_n(\breve{F}_0)(u)=&(\breve{F}_0\breve{F}_n-q\breve{F}_n\breve{F}_0)(\otimes_{j=1}^m u_j) \\
      =&\sum_{k=1}^m q^{\sum_{a=1}^{k-1}\langle \breve{\alpha}_0^{\vee}+\breve{\alpha}_n^{\vee}, \wt(u_a)\rangle}u_1\otimes \cdots \otimes u_{k-1}\otimes T_n(\breve{F}_0)u_k\otimes u_{k+1}\otimes \cdots u_m \\
      &+\sum_{k<l, j_k=1, j_l=2} q^{\sum_{a=1}^{k-1}\langle \breve{\alpha}_0^{\vee}, \wt(u_a)\rangle+\sum_{b=1}^{l-1}\langle \breve{\alpha}_n^{\vee}, \wt(u_b)\rangle}(1-q^{1-\langle\breve{\alpha}_n^{\vee},\breve{\alpha}_0\rangle})u^{\prime} \\
      &+\sum_{k<l, j_k=2, j_l=1} q^{\sum_{a=1}^{k-1}\langle \breve{\alpha}_n^{\vee}, \wt(u_a)\rangle+\sum_{b=1}^{l-1}\langle \breve{\alpha}_0^{\vee}, \wt(u_b)\rangle}(q^{-\langle \breve{\alpha}_0^{\vee},\breve{\alpha}_n\rangle}-q)u^{\prime\prime}\\
      =&\sum_{k=1}^m q^{\sum_{a=1}^{k-1}\langle \breve{\alpha}_0^{\vee}+\breve{\alpha}_n^{\vee}, \wt(u_a)\rangle}u_1\otimes \cdots \otimes u_{k-1}\otimes T_n(\breve{F}_0)u_k\otimes u_{k+1}\otimes \cdots u_m \\
      &+\sum_{k<l, j_k=1, j_l=2} q^{\sum_{a=1}^{k-1}\langle \breve{\alpha}_0^{\vee}, \wt(u_a)\rangle+\sum_{b=1}^{l-1}\langle \breve{\alpha}_n^{\vee}, \wt(u_b)\rangle}(1-q^{1-\langle\breve{\alpha}_n^{\vee},\breve{\alpha}_0\rangle})u^{\prime},
    \end{align*}
    where 
    \begin{align*}
      u^{\prime}=u_1\otimes\cdots u_{k-1}\otimes \breve{F}_0u_k\otimes u_{k+1}\otimes \cdots\otimes u_{l-1}\otimes\breve{F}_n u_l\otimes u_{l+1}\otimes \cdots\otimes u_m,\\
      u^{\prime\prime}=u_1\otimes\cdots u_{k-1}\otimes \breve{F}_nu_k\otimes u_{k+1}\otimes \cdots\otimes u_{l-1}\otimes\breve{F}_0 u_l\otimes u_{l+1}\otimes \cdots\otimes u_m.
    \end{align*}
   
Here, the first equality follows from $T_n(\breve{F}_0)=\breve{F}_0\breve{F}_n-q\breve{F}_n\breve{F}_0$ and the last equality follows from $\langle \breve{\alpha}_0^{\vee},\breve{\alpha}_n\rangle=-1$.
    
    Also, we have
    \begin{align*}
      T_n(\breve{E}_0)(u)=&(\breve{E}_n\breve{E}_0-q^{-1}\breve{E}_0\breve{E}_n)(\otimes_{j=1}^m u_j) \\
      =&\sum_{k=1}^m q^{-\sum_{a=k+1}^m\langle \breve{\alpha}_0^{\vee}+\breve{\alpha}_n^{\vee}, \wt(u_a)\rangle}u_1\otimes \cdots \otimes u_{k-1}\otimes T_n(\breve{E}_0)u_k\otimes u_{k+1}\otimes \cdots u_m \\
      &+\sum_{k<l, j_k=1, j_l=2} q^{\sum_{a=k+1}^m\langle \breve{\alpha}_n^{\vee}, \wt(u_a)\rangle+\sum_{b=l+1}^m\langle \breve{\alpha}_0^{\vee}, \wt(u_b)\rangle}(q^{-\langle \breve{\alpha}_n^{\vee},\breve{\alpha}_0\rangle}-q^{-1})v^{\prime} \\
      &+\sum_{k<l, j_k=2, j_l=1} q^{\sum_{a=k+1}^m\langle \breve{\alpha}_0^{\vee}, \wt(u_a)\rangle+\sum_{b=l+1}^m\langle \breve{\alpha}_n^{\vee}, \wt(u_b)\rangle}(1-q^{-1}q^{-\langle \breve{\alpha}_0^{\vee},\breve{\alpha}_n\rangle})v^{\prime\prime}\\
      =&\sum_{k=1}^m q^{-\sum_{a=k+1}^m\langle \breve{\alpha}_0^{\vee}+\breve{\alpha}_n^{\vee}, \wt(u_a)\rangle}u_1\otimes \cdots \otimes u_{k-1}\otimes T_n(\breve{E}_0)u_k\otimes u_{k+1}\otimes \cdots u_m \\
      &+\sum_{k<l, j_k=1, j_l=2} q^{\sum_{a=k+1}^m\langle \breve{\alpha}_n^{\vee}, \wt(u_a)\rangle+\sum_{b=l+1}^m\langle \breve{\alpha}_0^{\vee}, \wt(u_b)\rangle}(q^{-\langle \breve{\alpha}_n^{\vee},\breve{\alpha}_0\rangle}-q^{-1})v^{\prime},
    \end{align*}
    where 
    \begin{align*}
      v^{\prime}=u_1\otimes\cdots u_{k-1}\otimes \breve{E}_nu_k\otimes u_{k+1}\otimes \cdots\otimes u_{l-1}\otimes\breve{E}_0 u_l\otimes u_{l+1}\otimes \cdots\otimes u_m,\\
      v^{\prime\prime}=u_1\otimes\cdots u_{k-1}\otimes \breve{E}_0u_k\otimes u_{k+1}\otimes \cdots\otimes u_{l-1}\otimes\breve{E}_n u_l\otimes u_{l+1}\otimes \cdots\otimes u_m.
    \end{align*}

Here, the first equality follows from $T_n(\breve{E}_0)=\breve{E}_n\breve{E}_0-q^{-1}\breve{E}_0\breve{E}_n$ and the last equality follows from $\langle \breve{\alpha}_0^{\vee},\breve{\alpha}_n\rangle=-1$.

    For each $1\le k< l\le m$ with $j_k=1, j_l=2$, we have
    \[
    \mathfrak{j}<(j_1,\ldots,j_{k-1},2,j_{k+1},\ldots,j_{l-1},1,j_{l+1},\ldots,j_m).
    \]
    Hence, we have $T_n(\breve{F}_0)(u), T_n(\breve{E}_0)(u)\in L_{\ge\mathfrak{j},1}$ by Lemma \ref{0nact}. Therefore, $L_{\ge\mathfrak{i},1}$ is a $U_q(\widehat{\mathfrak{sl}}_n)$-submodule of $L_{p,1}$.

    Next, we consider the case $\varepsilon=-1$. We take an $\mathfrak{j}=(j_1,\ldots,j_m)\in \mathcal{I}_p$ and $u=\otimes_{j=1}^m u_j\in L_{\mathfrak{j},-1}$. We have
    \begin{align*}
      T_n^{-1}(\breve{F}_0)(u)=&(-q\breve{F}_0\breve{F}_n+\breve{F}_n\breve{F}_0)(\otimes_{j=1}^m u_j) \\
      =&\sum_{k=1}^m q^{\sum_{a=1}^{k-1}\langle \breve{\alpha}_0^{\vee}+\breve{\alpha}_n^{\vee}, \wt(u_a)\rangle}u_1\otimes \cdots \otimes u_{k-1}\otimes T_n^{-1}(\breve{F}_0)u_k\otimes u_{k+1}\otimes \cdots u_m \\
      &+\sum_{k<l, j_k=1, j_l=2} q^{\sum_{a=1}^{k-1}\langle \breve{\alpha}_0^{\vee}, \wt(u_a)\rangle+\sum_{b=1}^{l-1}\langle \breve{\alpha}_n^{\vee}, \wt(u_b)\rangle}(-q+q^{-\langle\breve{\alpha}_n^{\vee},\breve{\alpha}_0\rangle})u^{\prime} \\
      &+\sum_{k<l, j_k=2, j_l=1} q^{\sum_{a=1}^{k-1}\langle \breve{\alpha}_n^{\vee}, \wt(u_a)\rangle+\sum_{b=1}^{l-1}\langle \breve{\alpha}_0^{\vee}, \wt(u_b)\rangle}(-qq^{-\langle \breve{\alpha}_0^{\vee},\breve{\alpha}_n\rangle}+1)u^{\prime\prime}\\
      =&\sum_{k=1}^m q^{\sum_{a=1}^{k-1}\langle \breve{\alpha}_0^{\vee}+\breve{\alpha}_n^{\vee}, \wt(u_a)\rangle}u_1\otimes \cdots \otimes u_{k-1}\otimes T_n^{-1}(\breve{F}_0)u_k\otimes u_{k+1}\otimes \cdots u_m \\
      &+\sum_{k<l, j_k=2, j_l=1} q^{\sum_{a=1}^{k-1}\langle \breve{\alpha}_n^{\vee}, \wt(u_a)\rangle+\sum_{b=1}^{l-1}\langle \breve{\alpha}_0^{\vee}, \wt(u_b)\rangle}(-qq^{-\langle \breve{\alpha}_0^{\vee},\breve{\alpha}_n\rangle}+1)u^{\prime\prime},
    \end{align*}
    where 
    \begin{align*}
      u^{\prime}=u_1\otimes\cdots u_{k-1}\otimes \breve{F}_0u_k\otimes u_{k+1}\otimes \cdots\otimes u_{l-1}\otimes\breve{F}_n u_l\otimes u_{l+1}\otimes \cdots\otimes u_m,\\
      u^{\prime\prime}=u_1\otimes\cdots u_{k-1}\otimes \breve{F}_nu_k\otimes u_{k+1}\otimes \cdots\otimes u_{l-1}\otimes\breve{F}_0 u_l\otimes u_{l+1}\otimes \cdots\otimes u_m.
    \end{align*}

Here, the first equality follows from $T_n^{-1}(\breve{F}_0)=-q\breve{F}_0\breve{F}_n+\breve{F}_n\breve{F}_0$ and the last equality follows from $\langle \breve{\alpha}_0^{\vee},\breve{\alpha}_n\rangle=-1$.

    Also, we have
    \begin{align*}
      T_n^{-1}(\breve{E}_0)(u)=&(-q^{-1}\breve{E}_n\breve{E}_0+\breve{E}_0\breve{E}_n)(\otimes_{j=1}^m u_j) \\
      =&\sum_{k=1}^m q^{-\sum_{a=k+1}^m\langle \breve{\alpha}_0^{\vee}+\breve{\alpha}_n^{\vee}, \wt(u_a)\rangle}u_1\otimes \cdots \otimes u_{k-1}\otimes T_n^{-1}(\breve{E}_0)u_k\otimes u_{k+1}\otimes \cdots u_m \\
      &+\sum_{k<l, j_k=1, j_l=2} q^{\sum_{a=k+1}^m\langle \breve{\alpha}_n^{\vee}, \wt(u_a)\rangle+\sum_{b=l+1}^m\langle \breve{\alpha}_0^{\vee}, \wt(u_b)\rangle}(-q^{-1}q^{-\langle \breve{\alpha}_n^{\vee},\breve{\alpha}_0\rangle}+1)v^{\prime} \\
      &+\sum_{k<l, j_k=2, j_l=1} q^{\sum_{a=k+1}^m\langle \breve{\alpha}_0^{\vee}, \wt(u_a)\rangle+\sum_{b=l+1}^m\langle \breve{\alpha}_n^{\vee}, \wt(u_b)\rangle}(-q^{-1}+q^{-\langle \breve{\alpha}_0^{\vee},\breve{\alpha}_n\rangle})v^{\prime\prime}\\
      =&\sum_{k=1}^m q^{-\sum_{a=k+1}^m\langle \breve{\alpha}_0^{\vee}+\breve{\alpha}_n^{\vee}, \wt(u_a)\rangle}u_1\otimes \cdots \otimes u_{k-1}\otimes T_n^{-1}(\breve{E}_0)u_k\otimes u_{k+1}\otimes \cdots u_m \\
      &+\sum_{k<l, j_k=2, j_l=1} q^{\sum_{a=k+1}^m\langle \breve{\alpha}_0^{\vee}, \wt(u_a)\rangle+\sum_{b=l+1}^m\langle \breve{\alpha}_n^{\vee}, \wt(u_b)\rangle}(-q^{-1}+q^{-\langle \breve{\alpha}_0^{\vee},\breve{\alpha}_n\rangle})v^{\prime\prime},
    \end{align*}
    where 
    \begin{align*}
      v^{\prime}=u_1\otimes\cdots u_{k-1}\otimes \breve{E}_nu_k\otimes u_{k+1}\otimes \cdots\otimes u_{l-1}\otimes\breve{E}_0 u_l\otimes u_{l+1}\otimes \cdots\otimes u_m,\\
      v^{\prime\prime}=u_1\otimes\cdots u_{k-1}\otimes \breve{E}_0u_k\otimes u_{k+1}\otimes \cdots\otimes u_{l-1}\otimes\breve{E}_n u_l\otimes u_{l+1}\otimes \cdots\otimes u_m.
    \end{align*}

Here, the first equality follows from $T_n^{-1}(\breve{E}_0)=-q^{-1}\breve{E}_n\breve{E}_0+\breve{E}_0\breve{E}_n$ and the last equality follows from $\langle \breve{\alpha}_0^{\vee},\breve{\alpha}_n\rangle=-1$.

    For each $1\le k< l\le m$ with $j_k=2, j_l=1$, we have
    \[
    \mathfrak{j}>(j_1,\ldots,j_{k-1},1,j_{k+1},\ldots,j_{l-1},2,j_{l+1},\ldots,j_m).
    \]
    Hence, we have $T_n^{-1}(\breve{F}_0)(u), T_n^{-1}(\breve{E}_0)(u)\in L_{\le\mathfrak{j},-1}$ by Lemma \ref{0nact}. Therefore, $L_{\le\mathfrak{i},-1}$ is a $U_q(\widehat{\mathfrak{sl}}_n)$-submodule of $L_{p,-1}$.
\end{proof}

Let $\mathfrak{i}, \mathfrak{j}\in \mathcal{I}_p$. If $\mathfrak{i}<\mathfrak{j}$ and there is no $\mathfrak{k}$ such that $\mathfrak{i}<\mathfrak{k}<\mathfrak{j}$, then we write $\mathfrak{j}=\mathfrak{i}_+, \mathfrak{i}=\mathfrak{j}_-$.

For $\mathfrak{i}\in \mathcal{I}_p$, we define the $U_q(\widehat{\mathfrak{sl}}_n)$-module $\mathcal{L}_{\mathfrak{i},\varepsilon}$ as follows.
\begin{itemize}
  \item For $\varepsilon=1$, we define 
  \[
  \mathcal{L}_{\mathfrak{i},1}=\left\{
\begin{array}{ll}
L_{\ge\mathfrak{i},1}(=L_{\mathfrak{i},1}) & (\mathfrak{i}\;\text{is maximum}) \\
L_{\ge\mathfrak{i},1}/L_{\ge\mathfrak{i}_+,1} & (\text{otherwise})
\end{array}
\right.;
  \]
  \item For $\varepsilon=-1$, we define 
  \[
  \mathcal{L}_{\mathfrak{i},-1}=\left\{
\begin{array}{ll}
L_{\le\mathfrak{i},-1}(=L_{\mathfrak{i},-1}) & (\mathfrak{i}\;\text{is minimum}) \\
L_{\le\mathfrak{i},-1}/L_{\le\mathfrak{i}_-,-1} & (\text{otherwise})
\end{array}
\right..
  \]
\end{itemize}

The canonical projections $L_{\ge \mathfrak{i},1}\to \mathcal{L}_{\mathfrak{i},1}$ and $L_{\le \mathfrak{i},-1}\to \mathcal{L}_{\mathfrak{i},-1}$ are denoted by $p_{\mathfrak{i},1}$ and $p_{\mathfrak{i},-1}$, respectively. Note that 
\[
\bigotimes_{j=1}^m N_{\chi_{\lambda}(j),i_j}\xrightarrow{\otimes_{j=1}^m \psi_{i_j,\varepsilon}^{\chi_{\lambda}(j)}}\bigotimes_{j=1}^m L_{i_j,\varepsilon}^{\chi_{\lambda}(j)}=L_{\mathfrak{i},\varepsilon}
\]
is an isomorphism of $U_q(\mathfrak{sl}_n)$-modules. From this, it follows that if $\varepsilon=1$, then the composition 
\[
\Xi_{\mathfrak{i},1}:\bigotimes_{j=1}^mN_{\chi_{\lambda}(j),i_j}\xrightarrow{\otimes_{j=1}^m \psi_{i_j,1}^{\chi_{\lambda}(j)}}L_{\mathfrak{i},1}\hookrightarrow L_{\ge\mathfrak{i},1}\xrightarrow{p_{\mathfrak{i},1}}\mathcal{L}_{\mathfrak{i},1}
\]
is an isomorphism of $U_q(\mathfrak{sl}_n)$-modules. Similarly, if $\varepsilon=-1$, then the composition
\[
\Xi_{\mathfrak{i},-1}:\bigotimes_{j=1}^mN_{\chi_{\lambda}(j),i_j}\xrightarrow{\otimes_{j=1}^m \psi_{i_j,-1}^{\chi_{\lambda}(j)}}L_{\mathfrak{i},-1}\hookrightarrow L_{\le\mathfrak{i},-1}\xrightarrow{p_{\mathfrak{i},-1}}\mathcal{L}_{\mathfrak{i},-1}
\]
is an isomorphism of $U_q(\mathfrak{sl}_n)$-modules.

\begin{proposition}\label{quot}
\textit{Let $\mathfrak{i}=(i_1,\ldots,i_m)\in \mathcal{I}_p$.
\begin{enumerate}
  \item For $\varepsilon=1$, the $U_q(\mathfrak{sl}_n)$-isomorphism
  \[
  \Xi_{\mathfrak{i},1}:\bigotimes_{j=1}^m N_{\chi_{\lambda}(j),i_j}\stackrel{\cong}{\longrightarrow}\mathcal{L}_{\mathfrak{i},1}
  \]
  is $U_q(\widehat{\mathfrak{sl}}_n)$-linear;
  \item For $\varepsilon=-1$, the $U_q(\mathfrak{sl}_n)$-isomorphism
  \[
  \Xi_{\mathfrak{i},-1}:\bigotimes_{j=1}^m N_{\chi_{\lambda}(j),i_j}\stackrel{\cong}{\longrightarrow}\mathcal{L}_{\mathfrak{i},-1}
  \]
  is $U_q(\widehat{\mathfrak{sl}}_n)$-linear.
\end{enumerate}}
\end{proposition}

\begin{proof}
 We first assume that $\varepsilon=1$.

Let $u_j\in N_{\chi_{\lambda}(j),i_j}$ for $j\in \{1,\ldots,m\}$ and we set $v_j=\psi_{i_j,1}^{\chi_{\lambda}(j)}(u_j)$. Then we have
    \begin{multline*}
      T_n(\breve{F}_0)\left(\otimes_{j=1}^m \psi_{i_j,1}^{\chi_{\lambda}(j)}(\otimes_{j=1}^m u_j)\right)\\
      \in \sum_{k=1}^m q^{-\sum_{a=1}^{k-1}\langle\breve{\alpha}_0^{\vee}+\breve{\alpha}_n^{\vee},\wt(v_a)\rangle}v_1\otimes\cdots\otimes v_{k-1}\otimes T_n(\breve{F}_0)v_k\otimes v_{k+1}\otimes\cdots\otimes v_m+\bigoplus_{\mathfrak{j}\in \mathcal{I}_p,\mathfrak{j}>\mathfrak{i}}L_{\mathfrak{j},1}
       \end{multline*}
       by the proof of Proposition \ref{Lsub}. Since
       \begin{align*}
         &\sum_{k=1}^m q^{-\sum_{a=1}^{k-1}\langle\breve{\alpha}_0^{\vee}+\breve{\alpha}_n^{\vee},\wt(v_a)\rangle}v_1\otimes\cdots\otimes v_{k-1}\otimes T_n(\breve{F}_0)v_k\otimes v_{k+1}\otimes\cdots\otimes v_m\\
         &=\sum_{k=1}^m q^{-\sum_{a=1}^{k-1}\langle\alpha_0,\wt(u_a)\rangle}v_1\otimes\cdots\otimes v_{k-1}\otimes \psi_{i_k,1}^{\chi_{\lambda}(k)}(F_0(u_k))\otimes v_{k+1}\otimes\cdots\otimes v_m\\
      &=\otimes_{j=1}^m\psi_{i_j,1}^{\chi_{\lambda}(j)}(F_0(\otimes_{j=1}^m u_j)),
       \end{align*}
      we conclude that 
      \[
      T_n(\breve{F}_0)\left(\otimes_{j=1}^m \psi_{i_j,1}^{\chi_{\lambda}(j)}(\otimes_{j=1}^m u_j)\right)\in \otimes_{j=1}^m\psi_{i_j,1}^{\chi_{\lambda}(j)}(F_0(\otimes_{j=1}^m u_j))+\bigoplus_{\mathfrak{j}\in \mathcal{I}_p,\mathfrak{j}>\mathfrak{i}}L_{\mathfrak{j},1}.
      \]
   Similarly, we can show that
   \[
   T_n(\breve{E}_0)\left(\otimes_{j=1}^m \psi_{i_j,1}^{\chi_{\lambda}(j)}(\otimes_{j=1}^m u_j)\right)\in \otimes_{j=1}^m\psi_{i_j,1}^{\chi_{\lambda}(j)}(E_0(\otimes_{j=1}^m u_j))+\bigoplus_{\mathfrak{j}\in \mathcal{I}_p,\mathfrak{j}>\mathfrak{i}}L_{\mathfrak{j},1}.
   \]
   Therefore, the $U_q(\mathfrak{sl}_n)$-isomorphism
   \[
   \bigotimes_{j=1}^m N_{\chi_{\lambda}(j),i_j}\xrightarrow{\otimes_{j=1}^m \psi_{i_j,1}^{\chi(j)}}\bigotimes_{j=1}^m L_{i_j,1}^{\chi_{\lambda}(j)}\hookrightarrow L_{\ge\mathfrak{i},1}\xrightarrow{p_{\mathfrak{i},1}}\mathcal{L}_{\mathfrak{i},1}
   \]
   is $U_q(\widehat{\mathfrak{sl}}_n)$-linear.

  The second assertion can be proved similarly.
\end{proof}

By Proposition \ref{quot}, we have an isomorphism of $U_q(\widehat{\mathfrak{sl}}_n)$-modules
\[
\Xi_{\mathfrak{i},\varepsilon}: \bigotimes_{j=1}^m N_{\chi_{\lambda}(j),i_j}\stackrel{\cong}{\longrightarrow}\mathcal{L}_{\mathfrak{i},\varepsilon}.
\]

\subsection{The $U_q(\widehat{\mathfrak{sl}}_n)$-submodule of $\Psi_{\varepsilon}^*V(\lambda)$ generated by $\breve{S}_xu_{\lambda}$}

In this subsection, we fix $\lambda=\sum_{i\in \breve{I}_0}m_i\breve{\varpi}_i\in \breve{P}_{0,+}$.

For each $x\in \breve{W}$, there is a $U_q(\widehat{\mathfrak{sl}}_n)$-linear morphism
\[
\iota_{x,\varepsilon}:V(j^*(x\lambda))\to \Psi_{\varepsilon}^*V(\lambda)
\]
which maps $u_{j^*(x\lambda)}$ to $\breve{S}_xu_{\lambda}$ since $\breve{S}_xu_{\lambda}$ is $\widehat{\mathfrak{sl}}_n$-extremal.

\begin{proposition}
 \textit{For each $x\in \breve{W}$, the map $\iota_{x,\varepsilon}:V(j^*(x\lambda))\to \Psi_{\varepsilon}^*V(\lambda)$ is injective.}
\end{proposition}

\begin{proof}
  We set $\breve{W}_0=\langle r_i \mid i\in \breve{I}_0 \rangle \subset \breve{W}$. When $W_0$ is regarded as a subgroup of $\breve{W}_0$, the set of right cosets $W_0\backslash\breve{W}_0$ consists of $W_0w_k$ for $k\in \breve{I}_0$ along with $W_0e$. (For the definition of $w_k$, see (\ref{wi})).

  We first consider the case where $x\in \breve{W}_0$. By Lemma \ref{wact}, we may assume that $x=w_k$ for some $k\in \breve{I}_0$ or that $x=e$.

  Assume first that $x=e$. Then, the image of the composition map 
  \[
  V(j^*(\lambda))\xrightarrow{\iota_{e,\varepsilon}} \Psi_{\varepsilon}^*V(\lambda)\xrightarrow{\Phi_{\lambda}} \Psi_{\varepsilon}^*\left(\bigotimes_{i\in \breve{I}_0}V(\breve{\varpi}_i)^{\otimes m_i}\right)
  \]
  is contained in $L_{0,\varepsilon}=\mathcal{L}_{(2,2,\ldots,2),\varepsilon}$. By composing with the isomorphism 
  \[
  \Xi_{(2,2,\ldots,2),\varepsilon}^{-1}:\mathcal{L}_{(2,2,\ldots,2),\varepsilon}\stackrel{\cong}{\longrightarrow} \bigotimes_{j=1}^m N_{\chi_{\lambda}(j),2}=\left(\bigotimes_{i\in I_0}V(\varpi_i)^{\otimes m_i}\right)\otimes \mathbb{Q}(q)[t^{\pm1}]^{\otimes m_n},
  \]
  we obtain the $U_q(\widehat{\mathfrak{sl}}_n)$-linear morphism
  \[
  V(j^*(\lambda))\to \left(\bigotimes_{i\in I_0}V(\varpi_i)^{\otimes m_i}\right)\otimes \mathbb{Q}(q)[t^{\pm1}]^{\otimes m_n}
  \] 
  which maps $u_{j^*(\lambda)}$ to $(\otimes_{i\in I_0} u_{\varpi_i}^{\otimes m_i})\otimes 1$. Since this map is injective by Theorem \ref{inj}, we conclude that $\iota_{e,\varepsilon}:V(j^*(\lambda))\to \Psi_{\varepsilon}^*(\lambda)$ is injective.

  Next, we assume that $x=w_1$. Then the image of the composition map 
  \[
  V(j^*(w_1(\lambda)))\xrightarrow{\iota_{w_1,\varepsilon}} \Psi_{\varepsilon}^*V(\lambda)\xrightarrow{\Phi_{\lambda}} \Psi_{\varepsilon}^*\left(\bigotimes_{i\in \breve{I}_0}V(\breve{\varpi}_i)^{\otimes m_i}\right)
  \]
  is contained in $L_{m,\varepsilon}=\mathcal{L}_{(1,1,\ldots,1),\varepsilon}$. By composing with the isomorphism 
  \[
  \Xi_{(1,1,\ldots,1),\varepsilon}^{-1}:\mathcal{L}_{(1,1,\ldots,1),\varepsilon}\stackrel{\cong}{\longrightarrow} \bigotimes_{j=1}^m N_{\chi_{\lambda}(j),1}=\mathbb{Q}(q)[t^{\pm1}]^{\otimes m_1}\otimes\left(\bigotimes_{i\in I_0}V(\varpi_i)^{\otimes m_{i+1}}\right),
  \]
  we obtain the $U_q(\widehat{\mathfrak{sl}}_n)$-linear morphism
  \[
  V(j^*(w_1(\lambda)))\to \mathbb{Q}(q)[t^{\pm1}]^{\otimes m_1}\otimes\left(\bigotimes_{i\in I_0}V(\varpi_i)^{\otimes m_{i+1}}\right)
  \] 
  which maps $u_{j^*(w_1(\lambda))}$ to $1\otimes(\otimes_{i\in I_0}u_{\varpi_i}^{\otimes m_{i+1}})$. Since this map is injective by Theorem \ref{inj}, we conclude that $\iota_{w_1,\varepsilon}:V(j^*(w_1(\lambda)))\to \Psi_{\varepsilon}^*(\lambda)$ is injective.

  Now assume $x=w_k$ for some $2\le k\le n$. We set $\mathfrak{i}=(\overbrace{2,2,\ldots,2,2}^{m_1+\cdots+m_{k-1}},\overbrace{1,1,\ldots,1,1}^{m_k+\cdots+m_n})$. For $\varepsilon=1$, the image of the composition map 
  \[
  V(j^*(w_k(\lambda)))\xrightarrow{\iota_{w_k,\varepsilon}}\Psi_1^*(\lambda)\xrightarrow{\Phi_{\lambda}}\Psi_1^*\left(\bigotimes_{i\in \breve{I}_0}V(\breve{\varpi}_i)^{\otimes m_i}\right)
  \]
  is contained in $L_{\ge\mathfrak{i},1}$. Hence, we obtain a $U_q(\widehat{\mathfrak{sl}}_n)$-linear morphism 
  \[
  V(j^*(w_k(\lambda)))\to \mathcal{L}_{\mathfrak{i},1}.
  \]
  For $\varepsilon=-1$, the image of the composition map 
  \[
  V(j^*(w_k(\lambda)))\xrightarrow{\iota_{w_k,\varepsilon}}\Psi_{-1}^*(\lambda)\xrightarrow{\Phi_{\lambda}}\Psi_{-1}^*\left(\bigotimes_{i\in \breve{I}_0}V(\breve{\varpi}_i)^{\otimes m_i}\right)
  \]
  is contained in $L_{\le\mathfrak{i},-1}$, yielding a $U_q(\widehat{\mathfrak{sl}}_n)$-linear morphism 
  \[
  V(j^*(w_k(\lambda)))\to \mathcal{L}_{\mathfrak{i},-1}.
  \]
  Thus, we obtain the $U_q(\widehat{\mathfrak{sl}}_n)$-linear morphism
  \[
  V(j^*(w_k(\lambda)))\to \mathcal{L}_{\mathfrak{i},\varepsilon}
  \]
  for $\varepsilon\in \{\pm1\}$. By composing with the isomorphism
  \[
  \Xi_{\mathfrak{i},\varepsilon}^{-1}:\mathcal{L}_{\mathfrak{i},\varepsilon}\stackrel{\cong}{\longrightarrow} \left(\bigotimes_{i=1}^{k-1} V(\varpi_i)^{\otimes m_i}\right)\otimes\left(\bigotimes_{i=k}^nV(\varpi_{i-1})^{\otimes m_i}\right),
  \]
  we obtain the $U_q(\widehat{\mathfrak{sl}}_n)$-linear morphism
  \[
  V(j^*(w_k(\lambda)))\to \left(\bigotimes_{i=1}^{k-1} V(\varpi_i)^{\otimes m_i}\right)\otimes\left(\bigotimes_{i=k}^nV(\varpi_{i-1})^{\otimes m_i}\right)
  \]
  which sends $u_{j^*(w_k(\lambda))}$ to $(\otimes_{i=1}^{k-1} u_{\varpi_i}^{\otimes m_i})\otimes(\otimes_{i=k}^n u_{\varpi_{i-1}}^{\otimes m_i})$. Since this map is injective by Theorem \ref{inj}, we conclude that $\iota_{w_k,\varepsilon}:V(j^*(w_k(\lambda)))\to\Psi_{\varepsilon}^*(\lambda)$ is injective.

  Finally, consider the case where $x$ is the form $x=wt_{\xi}$ with $w\in \breve{W}_0$ and $\xi\in \breve{Q}_0^{\vee}$. We have the following commutative diagram.
  \[
\xymatrix@C=50pt{
V(j^*(x(\lambda))=V(j^*(w(\lambda))-\langle\xi,\lambda\rangle\delta)\ar[r]^-{\iota_{x,\varepsilon}}\ar[d]^{\cong} & \Psi_{\varepsilon}^*V(\lambda)\ar[d]^{\cong}\\
V(j^*(w(\lambda))\otimes V(-\langle\xi,\lambda\rangle\delta)\ar[r]^-{\iota_{w,\varepsilon}\otimes \id} & \Psi_{\varepsilon}^*V(\lambda)\otimes V(-\langle\xi,\lambda\rangle\delta)
}
\]
Here, $V(j^*(x(\lambda)))\stackrel{\cong}{\longrightarrow}V(j^*(w(\lambda)))\otimes V(-\langle\xi,\lambda\rangle\delta)$ is the $U_q(\widehat{\mathfrak{sl}}_n)$-isomorphism is defined by $u_{j^*(x(\lambda))}\mapsto u_{j^*(w\lambda)}\otimes u_{-\langle\xi,\lambda\rangle\delta}$ and $\Psi_{\varepsilon}^*V(\lambda)\stackrel{\cong}{\longrightarrow} \Psi_{\varepsilon}^*V(\lambda)\otimes V(-\langle\xi,\lambda\rangle\delta)$ is the $U_q(\widehat{\mathfrak{sl}}_n)$-isomorphism induced by the $U_q(\widehat{\mathfrak{sl}}_{n+1})$-isomorphism 
\[
V(\lambda)\stackrel{\cong}{\longrightarrow}V(\lambda)\otimes V\left(-\langle\xi,\lambda\rangle\breve{\delta}\right)
\]
defined by $u_{\lambda}\mapsto (\breve{S}_{-t_{\xi}}u_{\lambda})\otimes u_{-\langle\xi,\lambda\rangle\breve{\delta}}$, where we identify $\Psi_{\varepsilon}^*\left(V(\lambda)\otimes V\left(-\langle\xi,\lambda\rangle\breve{\delta}\right)\right)$ with $\Psi_{\varepsilon}^*V(\lambda)\otimes V(-\langle\xi,\lambda\rangle\delta)$ via $u\otimes u_{-\langle\xi,\lambda\rangle\breve{\delta}}\mapsto u\otimes u_{-\langle\xi,\lambda\rangle\delta}$.

Since $\iota_{w,\varepsilon}$ is injective by the argument above, we conclude that $\iota_{x,\varepsilon}$ is injective. 
\end{proof}

\subsection{Graded character of $(M_{p,\varepsilon}\cap V_e^-(\lambda))$}

In this subsection, we express $\gch(M_{p,\varepsilon}\cap V_e^-(\lambda))$ in terms of Macdonald polynomials for $\lambda\in \breve{P}_{0,+}$.

In the remainder of this subsection, we fix a dominant polynomial weight $\Lambda=\sum_{i=1}^{n+1}\lambda_i\varepsilon_i\in X_{n+1}$ with $\lambda_{n+1}=0$ and we set $\lambda=\pi_{n+1}(\Lambda)=\sum_{i\in \breve{I}_0}(\lambda_i-\lambda_{i+1})\breve{\varpi}_i$.

By specializing at $t=0$, we obtain the following result from Theorem \ref{brMac}.

\begin{proposition}\label{Mac}
\textit{
  \[
  P_{\Lambda}^{\mathop{GL}}(x_1,\ldots,x_n,x_{n+1};q,0)=\sum_{\substack{M=(\mu_1,\ldots,\mu_n)\in \mathcal{P}_n\\ \Lambda/M:\text{horizontal strip}}} x_{n+1}^{|\Lambda-M|}\prod_{i\ge1}\binom{\lambda_i-\lambda_{i+1}}{\lambda_i-\mu_i}_qP_M^{\mathop{GL}}(x_1,\ldots,x_n;q,0)
  \]
  where $\binom{m}{r}_q=\frac{(1-q^m)(1-q^{m-1})\cdots(1-q^{m-r+1})}{(1-q)(1-q^2)\cdots(1-q^r)}$ for $m,r\in \mathbb{Z}_{\ge0}$.
  \qed}
\end{proposition}

We set 
\[
\breve{Q}_0=\bigoplus_{i\in \breve{I}_0}\mathbb{Z}\breve{\alpha}_i,\quad \breve{Q}_{0,k}=\left\{\sum_{i\in \breve{I}_0}k_i\breve{\alpha}_i\in \breve{Q}_0 \mid k_n=k\right\}.
\]
Then, a weight vector $u\in V(\lambda)$ satisfies $u\in M_{p,\varepsilon}$ if and only if $\wt(u)$ is of the form $\lambda+\alpha+k\delta$ with $\alpha\in \breve{Q}_{0,-p}$ by Proposition \ref{dirsum2}.

We set $\tilde{\alpha}_i=\varepsilon_i-\varepsilon_{i+1}\in X_{n+1}$ for $i\in \breve{I}_0$. Then we have $\pi_{n+1}(\tilde{\alpha}_i)=\alpha_i$ for $i\in \breve{I}_0$.

\begin{lemma}\label{GLwt}
  \textit{For a polynomial weight $N=\sum_{i=1}^{n+1}\nu_i\varepsilon_i\in X_{n+1}$ with $|\Lambda|=\nu_1+\cdots+\nu_{n+1}$, $\pi_{n+1}(N)$ is of the form $\lambda+\alpha$ with $\alpha\in \breve{Q}_{0,-p}$ if and only if $\nu_{n+1}=p$.}
\end{lemma}

\begin{proof}
Since $|\Lambda|=\nu_1+\cdots+\nu_{n+1}$, we have 
\[
\sum_{i=1}^n\lambda_i\varepsilon_i+\sum_{i=1}^nk_i\tilde{\alpha}_i=\sum_{i=1}^{n+1}\nu_i\varepsilon_i,
\] 
where $k_i=\sum_{a=1}^i(\nu_a-\lambda_a)$. Applying $\pi$ to both sides yields the desired result.
\end{proof}

Using Theorem \ref{bq}, Proposition \ref{Mac}, and Lemma \ref{GLwt}, we arrive at the following proposition.

\begin{proposition}\label{Mgch}
  \textit{We have
  \begin{multline*}
   \gch (M_{p,\varepsilon}\cap V_e^-(\lambda))\\
    =\sum_{\substack{M=(\mu_1,\ldots,\mu_n)\in \mathcal{P}_n\\ \Lambda/M:\text{horizontal strip}\\ |\Lambda/M|=p}} \left(\prod_{i\ge1}\left(\prod_{r=1}^{\lambda_i-\mu_i}(1-q^{-r})\right)\left(\prod_{s=1}^{\mu_i-\lambda_{i+1}}(1-q^{-s})\right)\right)^{-1}P_{\pi_n(M)}(x,q^{-1},0).
  \end{multline*}
}
\end{proposition}

\begin{proof}
  Recall that a weight vector $u\in V(\lambda)$ satisfies $u\in M_{p,\varepsilon}$ if and only if $\wt(u)$ is of the form $\lambda+\alpha+k\delta$ with $\alpha\in \breve{Q}_{0,-p}$. Therefore, by Theorem \ref{bq}, Proposition \ref{Mac}, and Lemma \ref{GLwt}, we obtain
  \begin{multline*}
   \gch (M_{p,\varepsilon}\cap V_e^-(\lambda))\\
    =\sum_{\substack{M=(\mu_1,\ldots,\mu_n)\in \mathcal{P}_n\\ \Lambda/M:\text{horizontal strip}\\ |\Lambda/M|=p}} \left(\prod_{i\ge1}\left(\prod_{r=1}^{\lambda_i-\mu_i}(1-q^{-r})\right)\left(\prod_{s=1}^{\mu_i-\lambda_{i+1}}(1-q^{-s})\right)\right)^{-1}\\
    \times\Theta\circ\Pi_{n+1}\left(x_{n+1}^{|\Lambda-M|}P_M^{\mathop{GL}}(x_1,\ldots,x_n;q^{-1},0)\right).
  \end{multline*}
  Since $j^*(\pi_{n+1}(\varepsilon_{n+1}))=0$ and the restriction of $j^*\circ\pi_{n+1}$ to $X_n$ coincides with $\pi_n$, we have
  \[
  \Theta\circ\Pi_{n+1}\left(x_{n+1}^{|\Lambda-M|}P_M^{\mathop{GL}}(x_1,\ldots,x_n;q^{-1},0)\right)=P_{\pi_n(M)}(x;q^{-1},0).
  \]
  These complete the proof.
\end{proof}

\begin{corollary}\label{Mgcht}
  \textit{Keep the setting above. For $\xi\in \breve{Q}_0^{\vee}$, we have
  \begin{multline*}
    \gch (M_{p,\varepsilon}\cap V_{t_{\xi}}^-(\lambda))\\
    =q^{-\langle\xi,\lambda\rangle}\sum_{\substack{M=(\mu_1,\ldots,\mu_n)\in \mathcal{P}_n\\ \Lambda/M:\text{horizontal strip}\\ |\Lambda/M|=p}} \left(\prod_{i\ge1}\left(\prod_{r=1}^{\lambda_i-\mu_i}(1-q^{-r})\right)\left(\prod_{s=1}^{\mu_i-\lambda_{i+1}}(1-q^{-s})\right)\right)^{-1}P_{\pi_n(M)}(x,q^{-1},0).
  \end{multline*}
  }
\end{corollary}

\begin{proof}
  We have the isomorphism of $U_q(\widehat{\mathfrak{sl}}_{n+1})$-modules
  \[
  V(\lambda)\stackrel{\cong}{\longrightarrow} V(\lambda)\otimes V\left(-\langle\xi,\lambda\rangle\breve{\delta}\right)
  \]
  that maps $u_{\lambda}$ to $(\breve{S}_{t_{-\xi}}u_{\lambda})\otimes u_{-\langle\xi,\lambda\rangle\breve{\delta}}$ by Lemma \ref{txi}. Since this isomorphism maps $M_{p,\varepsilon}\cap V_{t_{\xi}}^-(\lambda)$ to $(M_{p,\varepsilon}\cap V_e^-(\lambda))\otimes V\left(-\langle\xi,\lambda\rangle\breve{\delta}\right)$, the desired assertion follows from Proposition \ref{Mgch}.
\end{proof}

\subsection{The structure of $M_{0,\varepsilon}$ and $M_{m,\varepsilon}$}

In this subsection, we fix $\lambda=\sum_{i\in \breve{I}_0}m_i\breve{\varpi}_i\in \breve{P}_{0,+}$.

Let $\breve{z}_{i,\nu}$ be the $U_q^{\prime}(\widehat{\mathfrak{sl}}_{n+1})$-linear automorphism of $\bigotimes_{i\in \breve{I}_0}V(\breve{\varpi}_i)^{\otimes m_i}$ obtained by the action $\breve{z}_i:V(\breve{\varpi}_i)\to V(\breve{\varpi}_i)$ given by $u_{\breve{\varpi}_i}\mapsto u_{\breve{\varpi}_i+\breve{\delta}}$ on the $\nu$-th factor of $V(\breve{\varpi}_i)^{\otimes m_i}$. Then, for $\mathbf{c}_0\in \overline{\Par}(\lambda)$ and $\xi=\sum_{i\in \breve{I}_0}k_i\breve{\alpha}_i^{\vee}$, we have
\[
\Phi_{\lambda}(\breve{S}_{\mathbf{c}_0}^-\breve{S}_{t_{\xi}}u_{\lambda})=s_{\mathbf{c}_0}(\breve{z}_{i,\nu}^{-1})\prod_{i\in \breve{I}_0}(\breve{z}_{i,1}\cdots \breve{z}_{i,m_i})^{-k_i}\tilde{u}_{\lambda}
\]
by Proposition \ref{BN}.

\begin{theorem}\label{M_0}
  \textit{There is an isomorphism of $U_q(\widehat{\mathfrak{sl}}_n)$-modules
  \[
  M_{0,\varepsilon}\cong V(j^*(\lambda))\otimes\left(\mathbb{Q}(q)[t_1^{\pm1},\ldots,t_{m_n}^{\pm1}]^{\mathfrak{S}_{m_n}}\right).
  \]
  }
\end{theorem}

\begin{proof}
  We define the $U_q(\widehat{\mathfrak{sl}}_n)$-homomorphism $\sigma_{0,\varepsilon}$  as the composition
  \[
  \sigma_{0,\varepsilon}:M_{0,\varepsilon}\xrightarrow{\Phi_{\lambda}|_{M_{0,\varepsilon}}} L_{0,\varepsilon}=\mathcal{L}_{(2,2,\ldots,2),\varepsilon}\xrightarrow{\Xi_{(2,\ldots,2),\varepsilon}^{-1}}\bigotimes_{j=1}^mN_{\chi_{\lambda}(j),2}=\left(\bigotimes_{i\in I_0}V(\varpi_i)^{\otimes m_i}\right)\otimes \mathbb{Q}(q)[t^{\pm1}]^{\otimes m_n}.
  \]
  Let $t_{\nu}$ denote the variable in the $\nu$-th factor of $\mathbb{Q}(q)[t^{\pm1}]^{\otimes m_n}$. Then we can identify $\mathbb{Q}(q)[t^{\pm1}]^{\otimes m_n}$ with $\mathbb{Q}(q)[t_1^{\pm1},\ldots,t_{m_n}^{\pm1}]$.

  We take $\xi=\sum_{i\in \breve{I}_0}k_i\breve{\alpha}_i^{\vee}$. We deduce that the restriction of $\sigma_{0,\varepsilon}$ to $M_{0,\varepsilon}\cap V_{t_{\xi}}^-(\lambda)$ gives an isomorphism of $U_q^-(\widehat{\mathfrak{sl}}_n)$-modules
  \[
   M_{0,\varepsilon}\cap V_{t_{\xi}}^-(\lambda)\cong \left(U_q^-(\widehat{\mathfrak{sl}}_n) \left(\bigotimes_{i\in I_0}u_{\varpi_i-k_i\delta}^{\otimes m_i}\right)\right)\otimes \left((t_1\cdots t_{m_n})^{-k_n}\mathbb{Q}(q)[t_1^{-1},\ldots,t_{m_n}^{-1}]^{\mathfrak{S}_{m_n}}\right).
  \]

  Let $\mathbf{c}_0\in \overline{\Par}(\lambda)$ be of the form $(\emptyset,\ldots,\emptyset,\Lambda)$ where $\emptyset$ is the empty partition. Then we have $\breve{S}_{\mathbf{c}_0}^-\breve{S}_{t_{\xi}}u_{\lambda}\in M_{0,\varepsilon}$ and 
  \begin{align*}
  \sigma_{0,\varepsilon}(\breve{S}_{\mathbf{c}_0}^-\breve{S}_{t_{\xi}}u_{\lambda})&=\Xi_{(2,2,\ldots,2),\varepsilon}^{-1}\left(s_{\Lambda}(\breve{z}_{n,1}^{-1},\ldots,\breve{z}_{n,m_n}^{-1})\prod_{i\in \breve{I}_0}(\breve{z}_{i,1}\cdots \breve{z}_{i,m_i})^{-k_i}\tilde{u}_{\lambda}\right)\\
    &=(-q)^{-\sum_{i\in I_0}k_im_ib_{i,\varepsilon}}\left(\bigotimes_{i\in I_0}u_{\varpi_i-k_i\delta}^{\otimes m_i}\right)\otimes \left((t_1\cdots t_{m_n})^{-k_n}s_{\Lambda}(t_1^{-1},\ldots,t_{m_n}^{-1})\right)
  \end{align*}
   by (\ref{eq:psi2}). Since $\{s_{\Lambda}(t_1^{-1},\ldots,t_{m_n}^{-1})\}_{\Lambda\in \mathcal{P}_{m_n}}$ is a $\mathbb{Q}(q)$-basis of $\mathbb{Q}(q)[t_1^{-1},\ldots,t_{m_n}^{-1}]^{\mathfrak{S}_{m_n}}$, the image of $\left(M_{0,\varepsilon}\cap V_{t_{\xi}}^-(\lambda)\right)$ contains the $U_q^-(\widehat{\mathfrak{sl}}_n)$-submodule 
  \[
  \left(U_q^-(\widehat{\mathfrak{sl}}_n) \left(\bigotimes_{i\in I_0}u_{\varpi_i-k_i\delta}^{\otimes m_i}\right)\right)\otimes \left((t_1\cdots t_{m_n})^{-k_n}\mathbb{Q}(q)[t_1^{-1},\ldots,t_{m_n}^{-1}]^{\mathfrak{S}_{m_n}}\right).
  \]

  By Corollary \ref{br} and Corollary \ref{Mgcht}, we have
  \begin{align*}
    \gch(M_{0,\varepsilon}\cap V_{t_{\xi}}^-(\lambda))&=\left(\prod_{i=1}^n\prod_{r=1}^{m_i}(1-q^{-r})\right)^{-1}P_{j^*(\lambda)}(x;q^{-1},0)\\
    &=q^{-k_nm_n}\left(\prod_{r=1}^{m_n}(1-q^{-r})\right)^{-1}\gch V_{t_{\zeta}}^-(j^*(\lambda)),
  \end{align*}
  where $\zeta=\sum_{i\in I_0}k_i\alpha_i^{\vee}\in Q_0^{\vee}$.

  By Theorem \ref{inj}, there is an injective $U_q(\widehat{\mathfrak{sl}}_n)$-linear morphism $\Phi_{j^*(\lambda)}:V(j^*(\lambda))\to \bigotimes_{i\in I_0}V(\varpi_i)^{\otimes m_i}$ given by $u_{j^*(\lambda)}\mapsto \bigotimes_{i\in I_0}u_{\varpi_i}^{\otimes m_i}$. Since $\Phi_{j^*(\lambda)}(S_{t_{\zeta}}u_{j^*(\lambda)})=\left(\otimes_{i\in I_0}u_{\varpi_i-k_i\delta}^{\otimes m_i}\right)$, the restriction of $\Phi_{j^*(\lambda)}$ gives an isomorphism of $U_q^-(\widehat{\mathfrak{sl}}_n)$-modules
  \[
  V_{t_{\zeta}}^-(j^*(\lambda))\cong U_q^-(\widehat{\mathfrak{sl}}_n) \left(\bigotimes_{i\in I_0}u_{\varpi_i-k_i\delta}^{\otimes m_i}\right).
  \]

 Hence, the graded character of $\left(U_q^-(\widehat{\mathfrak{sl}}_n) \left(\bigotimes_{i\in I_0}u_{\varpi_i-k_i\delta}^{\otimes m_i}\right)\right)\otimes \left((t_1\cdots t_{m_n})^{-k_n}\mathbb{Q}(q)[t_1^{-1},\ldots,t_{m_n}^{-1}]^{\mathfrak{S}_{m_n}}\right)$ is computed as
  \[
  q^{-k_nm_n}\sum_{\Lambda\in \mathcal{P}_{m_n}}q^{-|\Lambda|}\gch V_{\zeta}^-(j^*(\lambda))\\
    =q^{-k_nm_n}\left(\prod_{r=1}^{m_n}(1-q^{-r})\right)^{-1}\gch V_{t_{\zeta}}^-(j^*(\lambda)).
  \]

  Thus, by comparing the graded characters, we see that the restriction of $\sigma_{0,\varepsilon}$ gives an isomorphism of $U_q^-(\widehat{\mathfrak{sl}}_n)$-modules 
  \begin{align*}
    M_{0,\varepsilon}\cap V_{t_{\xi}}^-(\lambda)&\cong \left(U_q^-(\widehat{\mathfrak{sl}}_n) \left(\bigotimes_{i\in I_0}u_{\varpi_i-k_i\delta}^{\otimes m_i}\right)\right)\otimes \left((t_1\cdots t_{m_n})^{-k_n}\mathbb{Q}(q)[t_1^{-1},\ldots,t_{m_n}^{-1}]^{\mathfrak{S}_{m_n}}\right)\\
    &=\Phi_{j^*(\lambda)}(V_{t_{\zeta}}^-(j^*(\lambda)))\otimes \left((t_1\cdots t_{m_n})^{-k_n}\mathbb{Q}(q)[t_1^{-1},\ldots,t_{m_n}^{-1}]^{\mathfrak{S}_{m_n}}\right).
  \end{align*}

  By Proposition \ref{dem}, we have
  \[
   M_{0,\varepsilon}=\bigcup_{\xi\in \breve{Q}_0^{\vee}} \left(M_{0,\varepsilon}\cap V_{t_{\xi}}^-(\lambda)\right)\quad\text{and}\quad
    V(j^*(\lambda))=\bigcup_{\zeta\in Q_0^{\vee}} V_{t_{\zeta}}^-(j^*(\lambda)).
  \]

  Therefore, $\sigma_{0,\varepsilon}$ induces an isomorphism
 \[
  M_{0,\varepsilon}\cong \Phi_{j^*(\lambda)}(V(j^*(\lambda)))\otimes(\mathbb{Q}(q)[t_1^{\pm1},\ldots,t_{m_n}^{\pm1}]^{\mathfrak{S}_{m_n}})
    \cong V(j^*(\lambda))\otimes (\mathbb{Q}(q)[t_1^{\pm1},\ldots,t_{m_n}^{\pm1}]^{\mathfrak{S}_{m_n}}),
  \]
 as required.
\end{proof}

\begin{theorem}\label{M_m}
  \textit{There is an isomorphism of $U_q(\widehat{\mathfrak{sl}}_n)$-modules 
  \[
  M_{m,\varepsilon}\cong \left(\mathbb{Q}(q)[t_1^{\pm1},\ldots,t_{m_1}^{\pm1}]^{\mathfrak{S}_{m_1}}\right)\otimes V(j^*(w_1(\lambda))).
  \]
  }
\end{theorem}

\begin{proof}
 We define a $U_q(\widehat{\mathfrak{sl}}_n)$-homomorphism $\sigma_{m,\varepsilon}$ as the composition
  \[
  \sigma_{m,\varepsilon}:M_{m,\varepsilon}\xrightarrow{\Phi_{\lambda}|_{M_m}}L_{m,\varepsilon}=\mathcal{L}_{(1,1,\ldots,1),\varepsilon}\xrightarrow{\Xi_{(1,\ldots,1),\varepsilon}^{-1}}\bigotimes_{j=1}^mN_{\chi_{\lambda}(j),1}=\mathbb{Q}(q)[t^{\pm1}]^{\otimes m_1}\otimes\left(\bigotimes_{i\in I_0}V(\varpi_i)^{\otimes m_{i+1}}\right).
  \]
  Let $t_{\nu}$ denote the variable in the $\nu$-th factor of $\mathbb{Q}(q)[t^{\pm1}]^{\otimes m_1}$. Then we can identify $\mathbb{Q}(q)[t^{\pm1}]^{\otimes m_1}$ with $\mathbb{Q}(q)[t_1^{\pm1},\ldots,t_{m_1}^{\pm1}]$.

  We take $\xi=\sum_{i\in \breve{I}_0}k_i\breve{\alpha}_i^{\vee}$. We deduce that the restriction of $\sigma_{m,\varepsilon}$ to $M_{m,\varepsilon}\cap V_{t_{\xi}}^-(\lambda)$ gives an isomorphism of $U_q^-(\widehat{\mathfrak{sl}}_n)$-modules
  \[
   M_{m,\varepsilon}\cap V_{t_{\xi}}^-(\lambda)\cong \left((t_1\cdots t_{m_1})^{-k_1}\mathbb{Q}(q)[t_1^{-1},\ldots,t_{m_1}^{-1}]^{\mathfrak{S}_{m_1}}\right)\otimes\left(U_q^-(\widehat{\mathfrak{sl}}_n) \left(\bigotimes_{i\in I_0}u_{\varpi_i-k_{i+1}\delta}^{\otimes m_{i+1}}\right)\right).
  \]

  Let $\mathbf{c}_0\in \overline{\Par}(\lambda)$ be of the form $(\Lambda,\emptyset,\ldots,\emptyset)$. Then we have $\breve{S}_{w_1}\breve{S}_{\mathbf{c}_0}^-\breve{S}_{t_{\xi}}u_{\lambda}\in M_{m,\varepsilon}$ and 
  \begin{align*}
  \sigma_{m,\varepsilon}(\breve{S}_{w_1}\breve{S}_{\mathbf{c}_0}^-\breve{S}_{t_{\xi}}u_{\lambda})&=\Xi_{(1,1,\ldots,1),\varepsilon}^{-1}\left(\breve{S}_{w_1}s_{\Lambda}(\breve{z}_{1,1}^{-1},\ldots,\breve{z}_{1,m_1}^{-1})\prod_{i\in \breve{I}_0}(\breve{z}_{i,1}\cdots \breve{z}_{i,m_i})^{-k_i}\tilde{u}_{\lambda}\right)\\
    &=(-q)^{-\sum_{i\in I_0}k_{i+1}m_{i+1}a_{i+1,\varepsilon}}\left((t_1\cdots t_{m_1})^{-k_1}s_{\Lambda}(t_1^{-1},\ldots,t_{m_1}^{-1})\right)\otimes\left(\bigotimes_{i\in I_0}u_{\varpi_i-k_{i+1}\delta}^{\otimes m_{i+1}}\right)
  \end{align*}
   by (\ref{eq:psi2}). (For the definition of $w_1$, see (\ref{wi})). Since $\{s_{\Lambda}(t_1^{-1},\ldots,t_{m_1}^{-1})\}_{\Lambda\in \mathcal{P}_{m_1}}$ is a $\mathbb{Q}(q)$-basis of $\mathbb{Q}(q)[t_1^{-1},\ldots,t_{m_1}^{-1}]^{\mathfrak{S}_{m_1}}$, the image of $\left(M_{m,\varepsilon}\cap V_{t_{\xi}}^-(\lambda)\right)$ contains the $U_q^-(\widehat{\mathfrak{sl}}_n)$-submodule 
  \[
  \left((t_1\cdots t_{m_1})^{-k_1}\mathbb{Q}(q)[t_1^{-1},\ldots,t_{m_1}^{-1}]^{\mathfrak{S}_{m_1}}\right)\otimes\left(U_q^-(\widehat{\mathfrak{sl}}_n) \left(\bigotimes_{i\in I_0}u_{\varpi_i-k_{i+1}\delta}^{\otimes m_{i+1}}\right)\right).
  \]

  By Corollary \ref{br} and Corollary \ref{Mgcht}, we have
  \begin{align*}
    \gch(M_{m,\varepsilon}\cap V_{t_{\xi}}^-(\lambda))&=\left(\prod_{i=1}^n\prod_{r=1}^{m_i}(1-q^{-r})\right)^{-1}P_{j^*(w_1(\lambda))}(x;q^{-1},0)\\
    &=q^{-k_1m_1}\left(\prod_{r=1}^{m_n}(1-q^{-r})\right)^{-1}\gch V_{t_{\zeta}}^-(j^*(w_1(\lambda))),
  \end{align*}
  where $\zeta=\sum_{i\in I_0}k_{i+1}\alpha_i^{\vee}\in Q_0^{\vee}$.

  By Theorem \ref{inj}, there is an injective $U_q(\widehat{\mathfrak{sl}}_n)$-linear morphism $\Phi_{j^*(w_1(\lambda))}:V(j^*(w_1(\lambda)))\to \bigotimes_{i\in I_0}V(\varpi_i)^{\otimes m_{i+1}}$ given by $u_{j^*(w_1(\lambda))}\mapsto \bigotimes_{i\in I_0}u_{\varpi_i}^{\otimes m_{i+1}}$. Since $\Phi_{j^*(w_1(\lambda))}(S_{t_{\zeta}}u_{j^*(w_1(\lambda))})=\left(\bigotimes_{i\in I_0}u_{\varpi_i-k_{i+1}\delta}^{\otimes m_{i+1}}\right)$, the restriction of $\Phi_{j^*(\lambda)}$ gives an isomorphism of $U_q^-(\widehat{\mathfrak{sl}}_n)$-modules
  \[
  V_{t_{\zeta}}^-(j^*(w_1(\lambda)))\cong U_q^-(\widehat{\mathfrak{sl}}_n) \left(\bigotimes_{i\in I_0}u_{\varpi_i-k_{i+1}\delta}^{\otimes m_{i+1}}\right).
  \]

 Hence, the graded character of $\left((t_1\cdots t_{m_1})^{-k_1}\mathbb{Q}(q)[t_1^{-1},\ldots,t_{m_1}^{-1}]^{\mathfrak{S}_{m_1}}\right)\otimes\left(U_q^-(\widehat{\mathfrak{sl}}_n) \left(\bigotimes_{i\in I_0}u_{\varpi_i-k_{i+1}\delta}^{\otimes m_{i+1}}\right)\right)$ is computed as
  \[
  q^{-k_1m_1}\sum_{\Lambda\in \mathcal{P}_{m_1}}q^{-|\Lambda|}\gch V_{\zeta}^-(j^*(w_1(\lambda)))\\
    =q^{-k_1m_1}\left(\prod_{r=1}^{m_1}(1-q^{-r})\right)^{-1}\gch V_{t_{\zeta}}^-(j^*(w_1(\lambda))).
  \]

  Thus, by comparing the graded characters, we see that the restriction of $\sigma_{m,\varepsilon}$ gives an isomorphism of $U_q^-(\widehat{\mathfrak{sl}}_n)$-modules 
  \begin{align*}
    M_{m,\varepsilon}\cap V_{t_{\xi}}^-(\lambda)&\cong \left((t_1\cdots t_{m_1})^{-k_1}\mathbb{Q}(q)[t_1^{-1},\ldots,t_{m_1}^{-1}]^{\mathfrak{S}_{m_1}}\right)\otimes\left(U_q^-(\widehat{\mathfrak{sl}}_n) \left(\bigotimes_{i\in I_0}u_{\varpi_i-k_{i+1}\delta}^{\otimes m_{i+1}}\right)\right)\\
    &=\left((t_1\cdots t_{m_1})^{-k_1}\mathbb{Q}(q)[t_1^{-1},\ldots,t_{m_1}^{-1}]^{\mathfrak{S}_{m_1}}\right)\otimes\Phi_{j^*(w_1(\lambda))}(V_{t_{\zeta}}^-(j^*(w_1(\lambda))).
  \end{align*}

  By Proposition \ref{dem}, we have
  \[
  M_{m,\varepsilon}=\bigcup_{\xi\in \breve{Q}_0^{\vee}} \left(M_{m,\varepsilon}\cap V_{t_{\xi}}^-(\lambda)\right)\quad\text{and}\quad
    V(j^*(w_1(\lambda)))=\bigcup_{\zeta\in Q_0^{\vee}} V_{t_{\zeta}}^-(j^*(w_1(\lambda))).
  \]

  Therefore, $\sigma_{m,\varepsilon}$ induces an isomorphism
  \[
  M_{m,\varepsilon}\cong (\mathbb{Q}(q)[t_1^{\pm1},\ldots,t_{m_1}^{\pm1}]^{\mathfrak{S}_{m_1}})\otimes\Phi_{j^*(w_1(\lambda))}(V(j^*(w_1(\lambda))))\cong  (\mathbb{Q}(q)[t_1^{\pm1},\ldots,t_{m_1}^{\pm1}]^{\mathfrak{S}_{m_1}})\otimes V(j^*(w_1(\lambda))),
  \]
  as required.
\end{proof}

\section{Branching rule}

 Keep the setting of the previous section. In this section, we fix $m\in \mathbb{Z}_{>0}$.

\subsection{Some lemmas}

We establish some lemmas that will be used later.

\begin{lemma}\label{tensor}
  \textit{Let $M_1,\ldots,M_k$ be integrable $U_q(\widehat{\mathfrak{sl}}_{n+1})$-modules and $v_1,\ldots,v_k$ be weight vectors with weights $\lambda_1,\ldots,\lambda_k$. Fix $i\in \breve{I}$. Assume that $\langle \alpha_i^{\vee},\lambda_j\rangle=1$ and $\breve{F}_i^{(2)}v_j=0$ for all $j=1,\ldots,k$. Then we have
  \begin{multline*}
    \breve{F}_i^{(p)}(v_1\otimes\cdots\otimes v_k)=\sum_{1\le i_1<\cdots <i_p\le k}q^{i_1+\cdots+i_p-\frac{1}{2}p(p+1)}v_1\otimes\cdots v_{i_1-1}\otimes \breve{F}_iv_{i_1}\otimes v_{i_1+1}\otimes \cdots\\
    \cdots \otimes v_{i_p-1}\otimes \breve{F}_iv_{i_p}\otimes v_{i_p+1}\otimes\cdots\otimes v_k 
  \end{multline*}
  for all $1\le p\le k$.}
\end{lemma}

\begin{proof}
  We prove this assertion by induction on $p$.

  The assertion on $p=0$ is clear. Next, we consider the case of $p>0$. Using  $\Delta(F_i)=F_i\otimes1+t_i\otimes F_i$, we obtain 
  \begin{multline*}
    \breve{F}_i\breve{F}_i^{(p-1)}(v_1\otimes\cdots\otimes v_k)\\
   =\sum_{1\le i_1<\cdots <i_p\le k}\sum_{a=1}^pq^{i_1+\cdots+i_{a-1}+i_{a+1}+i_p-\frac{1}{2}p(p-1)}q^{i_a-2a+1} v_1\otimes\cdots v_{i_1-1}\otimes \breve{F}_iv_{i_1}\otimes v_{i_1+1}\otimes \cdots \\
   \cdots \otimes v_{i_p-1}\otimes \breve{F}_iv_{i_p}\otimes v_{i_p+1}\otimes\cdots\otimes v_k 
  \end{multline*}
  by using the induction hypothesis. Since 
  \[
  \sum_{a=1}^pq^{i_1+\cdots+i_{a-1}+i_{a+1}+\cdots +i_p-\frac{1}{2}p(p-1)}q^{i_a-2a+1}=q^{i_1+\ldots+i_p-\frac{1}{2}p(p+1)}[p]_q,
  \]
  we have
  \begin{multline*}
    \breve{F}_i^{(p)}(v_1\otimes\cdots\otimes v_k)=\sum_{1\le i_1<\cdots <i_p\le k}q^{i_1+\cdots+i_p-\frac{1}{2}p(p+1)}v_1\otimes\cdots v_{i_1-1}\otimes \breve{F}_iv_{i_1}\otimes v_{i_1+1}\otimes \cdots\\
    \cdots \otimes v_{i_p-1}\otimes \breve{F}_iv_{i_p}\otimes v_{i_p+1}\otimes\cdots\otimes v_k .
  \end{multline*}
This gives the desired result.
\end{proof}

\begin{lemma}\label{t}
  \textit{For $0\le p\le m$, we have the following equality in $V(\breve{\varpi}_i)^{\otimes m}$:
  \begin{multline*}
     \breve{F}_n^{(p)}\breve{F}_{n-1}^{(p)}\cdots \breve{F}_i^{(p)}u_{\breve{\varpi}_i+k_1\breve{\delta}}\otimes\cdots\otimes  u_{\breve{\varpi}_i+k_m\breve{\delta}}\\
    =\sum_{1\le i_1< \cdots < i_p \le m} q^{i_1+\cdots+ i_p-\frac{1}{2}p(p+1)} u_{\breve{\varpi}_i+k_1\breve{\delta}}\otimes \cdots \otimes u_{\breve{\varpi}_i+k_{i_1-1}\breve{\delta}}\otimes \breve{S}_{w_i}u_{\breve{\varpi}_i+k_{i_1}\breve{\delta}}\otimes u_{\breve{\varpi}_i+k_{i_1+1}\breve{\delta}}\otimes \cdots \\
     \cdots \otimes u_{\breve{\varpi}_i+k_{i_p-1}\breve{\delta}}\otimes \breve{S}_{w_i}u_{\breve{\varpi}_i+k_{i_p}\breve{\delta}}\otimes u_{\breve{\varpi}_i+k_{i_p+1}\breve{\delta}}\otimes \cdots \otimes u_{\breve{\varpi}_i+k_m\breve{\delta}}.
  \end{multline*}
  Here, $w_i\in \breve{W}$ is defined in (\ref{wi}).}
\end{lemma}

\begin{proof}
 We prove that 
  \begin{multline*}
    \breve{F}_r^{(p)}\breve{F}_{n-1}^{(p)}\cdots \breve{F}_i^{(p)}u_{\breve{\varpi}_i+k_1\breve{\delta}}\otimes\cdots\otimes  u_{\breve{\varpi}_i+k_m\breve{\delta}}\\
    =\sum_{1\le i_1< \cdots < i_p \le m} q^{i_1+\cdots+ i_p-\frac{1}{2}p(p+1)} u_{\breve{\varpi}_i+k_1\breve{\delta}}\otimes \cdots \otimes u_{\breve{\varpi}_i+k_{i_1-1}\breve{\delta}}\otimes \breve{S}_{w_{i,r}}u_{\breve{\varpi}_i+k_{i_1}\breve{\delta}}\otimes u_{\breve{\varpi}_i+k_{i_1+1}\breve{\delta}}\otimes \cdots \\
     \cdots \otimes u_{\breve{\varpi}_i+k_{i_p-1}\breve{\delta}}\otimes \breve{S}_{w_{i,r}}u_{\breve{\varpi}_i+k_{i_p}\breve{\delta}}\otimes u_{\breve{\varpi}_i+k_{i_p+1}\breve{\delta}}\otimes \cdots \otimes u_{\breve{\varpi}_i+k_m\breve{\delta}}
  \end{multline*}
  for $i\le r\le n$, where $w_{i,r}=s_rs_{r-1}\cdots s_i$. We prove this assertion by induction on $r$.

  The case for $r=i$ follows from Lemma \ref{tensor}. Next, we consider the case where $r>i$. By the induction hypothesis, we have
  \begin{multline*}
    \breve{F}_{r-1}^{(p)}\breve{F}_{n-1}^{(p)}\cdots \breve{F}_i^{(p)}u_{\breve{\varpi}_i+k_1\breve{\delta}}\otimes\cdots\otimes  u_{\breve{\varpi}_i+k_m\breve{\delta}}\\
    =\sum_{1\le i_1< \cdots < i_p \le m} q^{i_1+\cdots +i_p-\frac{1}{2}p(p+1)} u_{\breve{\varpi}_i+k_1\breve{\delta}}\otimes \cdots \otimes u_{\breve{\varpi}_i+k_{i_1-1}\breve{\delta}}\otimes \breve{S}_{w_{i,{r-1}}}u_{\breve{\varpi}_i+k_{i_1}\breve{\delta}}\otimes u_{\breve{\varpi}_i+k_{i_1+1}\breve{\delta}}\otimes \cdots \\
     \cdots \otimes u_{\breve{\varpi}_i+k_{i_p-1}\breve{\delta}}\otimes \breve{S}_{w_{i,{r-1}}}u_{\breve{\varpi}_i+k_{i_p}\breve{\delta}}\otimes u_{\breve{\varpi}_i+k_{i_p+1}\breve{\delta}}\otimes \cdots \otimes u_{\breve{\varpi}_i+k_m\breve{\delta}}.
  \end{multline*}
We fix $(i_1,\ldots,i_p)$ such that $1\le i_1<\cdots<i_p\le m$ and set
\begin{gather*}
  v_1=u_{\breve{\varpi}_i+k_1\breve{\delta}}\otimes \cdots \otimes u_{\breve{\varpi}_i+k_{i_1-1}\breve{\delta}}\otimes \breve{S}_{w_{i,{r-1}}}u_{\breve{\varpi}_i+k_{i_1}\breve{\delta}},\\
  v_2=u_{\breve{\varpi}_i+k_{i_1+1}\breve{\delta}}\otimes \cdots \otimes u_{\breve{\varpi}_i+k_{i_2-1}\breve{\delta}}\otimes \breve{S}_{w_{i,{r-1}}}u_{\breve{\varpi}_i+k_{i_2}\breve{\delta}},\\
  \vdots \\
  v_{p-1}=u_{\breve{\varpi}_i+k_{i_{p-2}+1}\breve{\delta}}\otimes \cdots \otimes u_{\breve{\varpi}_i+k_{i_{p-1}-1}\breve{\delta}}\otimes \breve{S}_{w_{i,{r-1}}}u_{\breve{\varpi}_i+k_{i_{p-1}}\breve{\delta}},\\
  v_p=u_{\breve{\varpi}_i+k_{i_{p-1}+1}\breve{\delta}}\otimes \cdots \otimes u_{\breve{\varpi}_i+k_{i_p-1}\breve{\delta}}\otimes \breve{S}_{w_{i,{r-1}}}u_{\breve{\varpi}_i+k_{i_p}\breve{\delta}}\otimes u_{\breve{\varpi}_i+k_{i_p+1}\breve{\delta}}\otimes \cdots\otimes u_{\breve{\varpi}_i+k_m\breve{\delta}}.
\end{gather*}

Since the weight $\lambda_j$ of $v_j$ satisfies $\langle \alpha_r^{\vee},\lambda_j\rangle =1$ for all $j=1,\ldots,p$, we have
\[
   \breve{F}_r^{(p)}(v_1\otimes\cdots\otimes v_p)=\breve{F}_r(v_1)\otimes \breve{F}_r(v_2)\otimes \cdots \otimes \breve{F}_r(v_p)
\]
by Lemma \ref{tensor}. Since $\Delta(\breve{F}_i)=\breve{F}_i\otimes1+t_i\otimes \breve{F}_i$, we have
\begin{gather*}
  \breve{F}_r(v_1)=u_{\breve{\varpi}_i+k_1\breve{\delta}}\otimes \cdots \otimes u_{\breve{\varpi}_i+k_{i_1-1}\breve{\delta}}\otimes \breve{S}_{w_{i,r}}u_{\breve{\varpi}_i+k_{i_1}\breve{\delta}},\\
  \breve{F}_r(v_2)=u_{\breve{\varpi}_i+k_{i_1+1}\breve{\delta}}\otimes \cdots \otimes u_{\breve{\varpi}_i+k_{i_2-1}\breve{\delta}}\otimes \breve{S}_{w_{i,r}}u_{\breve{\varpi}_i+k_{i_2}\breve{\delta}},\\
  \vdots \\
  \breve{F}_r(v_{p-1})=u_{\breve{\varpi}_i+k_{i_{p-2}+1}\breve{\delta}}\otimes \cdots \otimes u_{\breve{\varpi}_i+k_{i_{p-1}-1}\breve{\delta}}\otimes \breve{S}_{w_{i,r}}u_{\breve{\varpi}_i+k_{i_{p-1}}\breve{\delta}},\\
  \breve{F}_r(v_p)=u_{\breve{\varpi}_i+k_{i_{p-1}+1}\breve{\delta}}\otimes \cdots \otimes u_{\breve{\varpi}_i+k_{i_p-1}\breve{\delta}}\otimes \breve{S}_{w_{i,r}}u_{\breve{\varpi}_i+k_{i_p}\breve{\delta}}\otimes u_{\breve{\varpi}_i+k_{i_p+1}\breve{\delta}}\otimes \cdots\otimes u_{\breve{\varpi}_i+k_m\breve{\delta}}.
\end{gather*}
  Thus, we have 
  \begin{multline*}
    \breve{F}_r^{(p)}\breve{F}_{n-1}^{(p)}\cdots \breve{F}_i^{(p)}u_{\breve{\varpi}_i+k_1\breve{\delta}}\otimes\cdots\otimes  u_{\breve{\varpi}_i+k_m\breve{\delta}}\\
    =\sum_{1\le i_1< \cdots < i_p \le m} q^{i_1+\cdots i_p-\frac{1}{2}p(p+1)} u_{\breve{\varpi}_i+k_1\breve{\delta}}\otimes \cdots \otimes u_{\breve{\varpi}_i+k_{i_1-1}\breve{\delta}}\otimes \breve{S}_{w_{i,r}}u_{\breve{\varpi}_i+k_{i_1}\breve{\delta}}\otimes u_{\breve{\varpi}_i+k_{i_1+1}\breve{\delta}}\otimes \cdots \\
     \cdots \otimes u_{\breve{\varpi}_i+k_{i_p-1}\breve{\delta}}\otimes \breve{S}_{w_{i,r}}u_{\breve{\varpi}_i+k_{i_p}\breve{\delta}}\otimes u_{\breve{\varpi}_i+k_{i_p+1}\breve{\delta}}\otimes \cdots \otimes u_{\breve{\varpi}_i+k_m\breve{\delta}},
  \end{multline*}
  which is the required result.
\end{proof}

\subsection{Branching rule for $V(m\breve{\varpi}_i)$: the case $2\le i\le n-1$}

In this subsection, we assume that $\lambda\in \breve{P}_{0,+}$ is of the form $\lambda=m\breve{\varpi}_i$ with $2\le i\le n-1$.

\begin{theorem}\label{2ni.}
  \textit{For $0\le p\le m$, there is an isomorphism of $U_q(\widehat{\mathfrak{sl}}_n)$-modules
  \[
  M_{p,\varepsilon}\cong V(p\varpi_{i-1}+(m-p)\varpi_i).
  \]
  }
\end{theorem}

\begin{proof}

  We first assume that $\varepsilon=1$. We set $\mathfrak{i}=(\overbrace{1,\ldots,1}^p,\overbrace{2,\ldots,2}^{m-p})\in \mathcal{I}_p$. We define an $U_q(\widehat{\mathfrak{sl}}_n)$-linear morphism $\tau_{p,1}^i$ as the composition
  \[
  \tau_{p,1}^i:M_{p,1}\xrightarrow{\Phi_{m\breve{\varpi}_i}\mid_{M_{p,1}}} L_{p,1}=L_{\ge\mathfrak{i},1}\xrightarrow{p_{\mathfrak{i},1}}\mathcal{L}_{\mathfrak{i},1}\xrightarrow{\Xi_{\mathfrak{i},1}^{-1}}V(\varpi_{i-1})^{\otimes p}\otimes V(\varpi_i)^{\otimes m-p}. 
  \]

 Let $k\in \mathbb{Z}$ and we set $\xi=k\breve{\alpha}_i^{\vee}\in \breve{Q}_0^{\vee}$. By Lemma \ref{t}, we have 
  \[
  \Phi_{m\breve{\varpi}_i}(\breve{F}_n^{(p)}\breve{F}_{n-1}^{(p)}\cdots \breve{F}_i^{(p)}\breve{S}_{t_{\xi}}u_{m\breve{\varpi}_i})\in \left(\breve{S}_{w_i}u_{\breve{\varpi}_i-k\breve{\delta}}\right)^{\otimes p}\otimes (u_{\breve{\varpi}_i-k\breve{\delta}})^{\otimes m-p}+ \left(\bigoplus_{\mathfrak{j}>\mathfrak{i}}L_{\mathfrak{j},1}\right).
  \]
  Hence, we have
  \[
  \tau_{p,1}^i(\breve{F}_n^{(p)}\breve{F}_{n-1}^{(p)}\cdots \breve{F}_i^{(p)}\breve{S}_{t_{\xi}}u_{m\breve{\varpi}_i})=(-q)^{-kpa_{i,1}-k(m-p)b_{i,1}}(u_{{\varpi}_{i-1}-k\delta})^{\otimes p}\otimes(u_{\varpi_i-k\delta})^{\otimes m-p}.
  \]
  Thus, $\tau_{p,1}^i(M_{p,1}\cap V_{t_{\xi}}^-(m\breve{\varpi}_i))$ contains the $U_q^-(\widehat{\mathfrak{sl}}_n)$-submodule
  \[
  U_q^-(\widehat{\mathfrak{sl}}_n)\left((u_{\varpi_{i-1}-k\delta})^{\otimes p}\otimes(u_{\varpi_i-k\delta})^{\otimes m-p}\right)
  \]
  of $V(\varpi_{i-1})^{\otimes p}\otimes V(\varpi_i)^{\otimes m-p}$.

  By Corollary \ref{br} and Corollary \ref{Mgcht}, we have
  \begin{align*}
    \gch \left(M_{p,1}\cap V_{t_{\xi}}^-(m\breve{\varpi}_i)\right)&=q^{-km}\left(\left(\prod_{r=1}^p(1-q^{-r})\right)\left(\prod_{s=1}^{m-p}(1-q^{-s})\right)\right)^{-1}P_{p\varpi_{i-1}+(m-p)\varpi_i}(x;q^{-1},0)\\
    &=\gch V_{t_{\zeta}}^-(p\varpi_{i-1}+(m-p)\varpi_i),
  \end{align*}

where $\zeta=k\alpha_{i-1}^{\vee}+k\alpha_i^{\vee}\in Q_0^{\vee}$. By Theorem \ref{inj}, there is a $U_q(\widehat{\mathfrak{sl}}_n)$-linear injective morphism
\[
\Phi_{p\varpi_{i-1}+(m-p)\varpi_i}:V(p\varpi_{i-1}+(m-p)\varpi_i)\to V(\varpi_{i-1})^{\otimes p}\otimes V(\varpi_i)^{\otimes m-p}
\]
given by $u_{p\varpi_{i-1}+(m-p)\varpi_i}\mapsto \left(u_{\varpi_{i-1}}^{\otimes p}\right)\otimes\left(u_{\varpi_i}^{\otimes m-p}\right)$. Since 
\[
\Phi_{p\varpi_{i-1}+(m-p)\varpi_i}(S_{t_{\zeta}}u_{p\varpi_{i-1}+(m-p)\varpi_i})=(u_{\varpi_{i-1}-k\delta})^{\otimes p}\otimes(u_{\breve{\varpi}_i-k\delta})^{\otimes m-p},
\]
the restriction of $\Phi_{p\varpi_{i-1}+(m-p)\varpi_i}$ to $V_{t_{\zeta}}^-(p\varpi_{i-1}+(m-p)\varpi_i)$ gives an isomorphism of $U_q^-(\widehat{\mathfrak{sl}}_n)$-modules
\[
V_{t_{\zeta}}^-(p\varpi_{i-1}+(m-p)\varpi_i)\stackrel{\cong}{\longrightarrow} U_q^-(\widehat{\mathfrak{sl}}_n)\left((u_{\varpi_{i-1}-k\delta})^{\otimes p}\otimes(u_{\breve{\varpi}_i-k\delta})^{\otimes m-p}\right).
\]

Hence, we have
\[
  \gch U_q^-(\widehat{\mathfrak{sl}}_n)\left((u_{\varpi_{i-1}-k\delta})^{\otimes p}\otimes(u_{\breve{\varpi}_i-k\delta})^{\otimes m-p}\right)=\gch V_{t_{\zeta}}^-(p\varpi_{i-1}+(m-p)\varpi_i)
\]

Thus, by comparing the graded characters, we see that the restriction of $\tau_{p,1}^i$ to $\left(M_{p,1}\cap V_{t_{\xi}}^-(m\breve{\varpi}_i)\right)$ gives an isomorphism of $U_q(\widehat{\mathfrak{sl}}_n)$-modules
\begin{align*}
  M_{p,1}\cap V_{t_{\xi}}^-(m\breve{\varpi}_i)&\cong U_q^-(\widehat{\mathfrak{sl}}_n)\left((u_{\varpi_{i-1}-k\delta})^{\otimes p}\otimes(u_{\breve{\varpi}_i-k\delta})^{\otimes m-p}\right)\\
  &=\Phi_{p\varpi_{i-1}+(m-p)\varpi_i}(V_{t_{\zeta}}^-(p\varpi_{i-1}+(m-p)\varpi_i)).
\end{align*}

By Proposition \ref{dem}, we have
\[
   M_{p,1}=\bigcup_{k\in \mathbb{Z}}\left(M_{p,1}\cap V_{t_{k\breve{\alpha}^{\vee}}}(m\breve{\varpi}_i)\right)\quad\text{and}\quad
  V(p\varpi_{i-1}+(m-p)\varpi_i)=\bigcup_{k\in \mathbb{Z}} V_{t_{k(\alpha_{i-1}^{\vee}+\alpha_i^{\vee})}}^-(p\varpi_{i-1}+(m-p)\varpi_i).
  \]

Therefore, $\tau_{p,1}^i$ induces an isomorphism
  \[
   M_{p,1}\cong\Phi_{p\varpi_{i-1}+(m-p)\varpi_i}(V(p\varpi_{i-1}+(m-p)\varpi_i))\cong V(p\varpi_{i-1}+(m-p)\varpi_i)
  \]
in this case.

  Next, we assume that $\varepsilon=-1$. We set $\mathfrak{i}^{\prime}=(\overbrace{2,\ldots,2}^{m-p},\overbrace{1,\ldots,1}^p)\in \mathcal{I}_p$. We define an $U_q(\widehat{\mathfrak{sl}}_n)$-linear morphism $\tau_{p,-1}^i$ as the composition
  \[
  \tau_{p,-1}^i:M_{p,-1}\xrightarrow{\Phi_{m\breve{\varpi}_i}\mid_{M_{p,-1}}} L_{p,-1}=L_{\le\mathfrak{i}^{\prime},-1}\xrightarrow{p_{\mathfrak{i}^{\prime},-1}}\mathcal{L}_{\mathfrak{i}^{\prime},-1}\xrightarrow{\Xi_{\mathfrak{i}^{\prime},-1}^{-1}}V(\varpi_i)^{\otimes m-p}\otimes V(\varpi_{i-1})^{\otimes p}. 
  \]

 Let $k\in \mathbb{Z}$ and we set $\xi=k\breve{\alpha}_i^{\vee}\in \breve{Q}_0^{\vee}$. By Lemma \ref{t}, we have 
  \[
  \Phi_{m\breve{\varpi}_i}(\breve{F}_n^{(p)}\breve{F}_{n-1}^{(p)}\cdots \breve{F}_i^{(p)}\breve{S}_{t_{\xi}}u_{m\breve{\varpi}_i})\in q^{p(m-p)}(u_{\breve{\varpi}_i-k\breve{\delta}})^{\otimes m-p}\otimes \left(\breve{S}_{w_i}u_{\breve{\varpi}_i-k\breve{\delta}}\right)^{\otimes p}+ \left(\bigoplus_{\mathfrak{j}>\mathfrak{i}^{\prime}}L_{\mathfrak{j},-1}\right).
  \]
  Hence, we have
  \[
  \tau_{p,-1}^i(\breve{F}_n^{(p)}\breve{F}_{n-1}^{(p)}\cdots \breve{F}_i^{(p)}\breve{S}_{t_{\xi}}u_{m\breve{\varpi}_i})=q^{p(m-p)}(-q)^{-kpa_{i,-1}-k(m-p)b_{i,-1}}(u_{\varpi_i-k\delta})^{\otimes m-p}\otimes(u_{\varpi_{i-1}-k\delta})^{\otimes p}.
  \]
  Thus, $\tau_{p,-1}^i(M_{p,-1}\cap V_{t_{\xi}}^-(m\breve{\varpi}_i))$ contains the $U_q^-(\widehat{\mathfrak{sl}}_n)$-submodule
  \[
  U_q^-(\widehat{\mathfrak{sl}}_n)\left((u_{\varpi_i-k\delta})^{\otimes m-p}\otimes (u_{\varpi_{i-1}-k\delta})^{\otimes p}\right)
  \]
  of $V(\varpi_i)^{\otimes m-p}\otimes V(\varpi_{i-1})^{\otimes p}$.

By Theorem \ref{inj}, there is a $U_q(\widehat{\mathfrak{sl}}_n)$-linear injective morphism
\[
\Phi_{p\varpi_{i-1}+(m-p)\varpi_i}^{\prime}:V(p\varpi_{i-1}+(m-p)\varpi_i)\to V(\varpi_i)^{\otimes m-p}\otimes V(\varpi_{i-1})^{\otimes p}
\]
given by $u_{p\varpi_{i-1}+(m-p)\varpi_i}\mapsto \left(u_{\varpi_i}^{\otimes m-p}\right)\otimes \left(u_{\varpi_{i-1}}^{\otimes p}\right)$. Since 
\[
\Phi_{p\varpi_{i-1}+(m-p)\varpi_i}^{\prime}(S_{t_{\zeta}}u_{p\varpi_{i-1}+(m-p)\varpi_i})=(u_{\breve{\varpi}_i-k\delta})^{\otimes m-p}\otimes (u_{\varpi_{i-1}-k\delta})^{\otimes p},
\]
the restriction of $\Phi_{p\varpi_{i-1}+(m-p)\varpi_i}^{\prime}$ gives an isomorphism of $U_q^-(\widehat{\mathfrak{sl}}_n)$-modules
\[
V_{t_{\zeta}}^-(p\varpi_{i-1}+(m-p)\varpi_i)\stackrel{\cong}{\longrightarrow} U_q^-(\widehat{\mathfrak{sl}}_n)\left((u_{\varpi_i-k\delta})^{\otimes m-p}\otimes (u_{\varpi_{i-1}-k\delta})^{\otimes p}\right).
\]

Hence, we have
\[
\gch U_q^-(\widehat{\mathfrak{sl}}_n)\left((u_{\breve{\varpi}_i-k\delta})^{\otimes m-p}\otimes(u_{\varpi_{i-1}-k\delta})^{\otimes p}\right)=\gch V_{t_{\zeta}}^-(p\varpi_{i-1}+(m-p)\varpi_i)  
\]

On the other hand, the graded character of $\left(M_{p,-1}\cap V_{t_{\xi}}^-(m\breve{\varpi}_i)\right)$ coincides with $\gch V_{t_{\zeta}}^-(p\varpi_{i-1}+(m-p)\varpi_i)$ by the argument above. Thus, by comparing the graded characters, the restriction of $\tau_{p,-1}^i$ gives the isomorphism of $U_q(\widehat{\mathfrak{sl}}_n)$-modules
\begin{align*}
  M_{p,-1}\cap V_{t_{\xi}}^-(m\breve{\varpi}_i)&\cong U_q^-(\widehat{\mathfrak{sl}}_n)\left((u_{\varpi_i-k\delta})^{\otimes m-p}\otimes (u_{\varpi_{i-1}-k\delta})^{\otimes p}\right)\\
  &=\Phi_{p\varpi_{i-1}+(m-p)\varpi_i}^{\prime}(V_{t_{\zeta}}^-(p\varpi_{i-1}+(m-p)\varpi_i)).
\end{align*}

Since
\[
  M_{p,-1}=\bigcup_{k\in \mathbb{Z}}\left(M_{p,-1}\cap V_{t_{k\breve{\alpha}^{\vee}}}(m\breve{\varpi}_i)\right)\quad\text{and}\quad
   V(p\varpi_{i-1}+(m-p)\varpi_i)=\bigcup_{k\in \mathbb{Z}} V_{t_{k(\alpha_{i-1}^{\vee}+\alpha_i^{\vee})}}^-(p\varpi_{i-1}+(m-p)\varpi_i)
  \]
as stated above, $\tau_{p,-1}^i$ induces an isomorphism
  \[
   M_{p,-1}\cong\Phi_{p\varpi_{i-1}+(m-p)\varpi_i}^{\prime}(V(p\varpi_{i-1}+(m-p)\varpi_i))\cong V(p\varpi_{i-1}+(m-p)\varpi_i),
  \]
  as required.
\end{proof}

\begin{corollary}\label{2in}
  \textit{There is an isomorphism of $U_q(\widehat{\mathfrak{sl}}_n)$-modules
  \[
  \Psi_{\varepsilon}^*V(m\breve{\varpi}_i)\cong \bigoplus_{p=0}^mV(p\varpi_{i-1}+(m-p)\varpi_i).
  \]
  }
\end{corollary}

\begin{proof}
This result follows immediately from Proposition \ref{dirsum2} and Theorem \ref{2ni.}.
\end{proof}

By the proof of Theorem \ref{2ni.}, we obtain the following result.

\begin{corollary}
\textit{
  There is an isomorphism of $U_q^-(\widehat{\mathfrak{sl}}_n)$-modules
  \[
  \pushQED{\qed}
     (\Psi_{\varepsilon}^-)^* V_e^-(m\breve{\varpi}_i)\cong \bigoplus_{p=0}^m V_e^-(p\varpi_{i-1}+(m-p)\varpi_i).\qedhere
\popQED
  \]
  }
\end{corollary}

\subsection{Branching rule for $V(m\breve{\varpi}_i)$: the case $i=1$}

In this subsection, we assume that $\lambda\in \breve{P}_{0,+}$ is of the form $\lambda=m\breve{\varpi}_1$.

Let $c^{\Lambda}_{M,N}\;(\Lambda,M,N\in \mathcal{P})$ denote the \textit{Littlewood-Richardson coefficient} (see \cite[Section \rm{I}-5]{M1} for the definition). Littlewood-Richardson coefficients satisfy the following properties.

\begin{lemma}[{\cite[Section \rm{I}-5]{M1}\label{L-R1}}]
\textit{We have:
  \begin{enumerate}
    \item $c^{\Lambda}_{M,N}=c^{\Lambda}_{N,M}$ for $\Lambda,M,N\in \mathcal{P}$;
    \item $c^{\Lambda}_{\Lambda,\emptyset}=1$ for $\Lambda\in \mathcal{P}$;
    \item $c^{\Lambda}_{M,N}=0$ unless $|\Lambda|=|M|+|N|$;
    \item $c^{\Lambda}_{M,N}=0$ unless $\Lambda\supset M$.\qed
  \end{enumerate}
}
\end{lemma}

\begin{lemma}[{\cite[Section \rm{I}-5]{M1}\label{L-R2}}]
   \textit{Let $(x_1,\ldots,x_k)$ and $(y_1,\ldots,y_l)$ be two sets of variables. We have
  \[
  s_{\Lambda}(x_1,\ldots,x_k,y_1,\ldots,y_l)=\sum_{M\in \mathcal{P}_k, N\in\mathcal{P}_l}c^{\Lambda}_{M,N}s_M(x_1,\ldots,x_k)s_N(y_1,\ldots,y_l)
  \]
  for $\Lambda\in \mathcal{P}_{k+l}$. \qed}
\end{lemma}

\begin{theorem}\label{i1.}
  \textit{For $0\le p\le m$, there is an isomorphism of $U_q(\widehat{\mathfrak{sl}}_n)$-modules
  \[
  M_{p,\varepsilon}\cong \mathbb{Q}(q)[t_1^{\pm1},\ldots,t_p^{\pm1}]^{\mathfrak{S}_p}\otimes V((m-p)\varpi_1) \left(\cong V((m-p)\varpi_1)\otimes \mathbb{Q}(q)[t_1^{\pm1},\ldots,t_p^{\pm1}]^{\mathfrak{S}_p}\right).
  \]
  }
\end{theorem}

\begin{proof}
We first assume that $\varepsilon=1$. We set $\mathfrak{i}=(\overbrace{1,\ldots,1}^p,\overbrace{2,\ldots,2}^{m-p})\in \mathcal{I}_p$. We define an $U_q(\widehat{\mathfrak{sl}}_n)$-linear morphism $\tau_{p,1}^1$ as the composition
  \[
  \tau_{p,1}^1:M_{p,1}\xrightarrow{\Phi_{m\breve{\varpi}_1}\mid_{M_{p,1}}} L_{p,1}=L_{\ge\mathfrak{i},1}\xrightarrow{p_{\mathfrak{i},1}}\mathcal{L}_{\mathfrak{i},1}\xrightarrow{\Xi_{\mathfrak{i},1}^{-1}}\mathbb{Q}(q)[t^{\pm1}]^{\otimes p}\otimes V(\varpi_1)^{\otimes m-p}. 
  \]

  Let $t_{\nu}$ denote the variable in the $\nu$-th factor of $\mathbb{Q}(q)[t^{\pm}]^{\otimes p}$. Then we can identify $\mathbb{Q}(q)[t^{\pm}]^{\otimes p}$ with $\mathbb{Q}(q)[t_1^{\pm1},\ldots,t_p^{\pm1}]$.

Let $k\in \mathbb{Z}$ and we set $\xi=k\breve{\alpha}_1^{\vee}$. We deduce by induction on $|\Lambda|$ that $\tau_{p,1}^1(M_{p,1}\cap V_{t_{\xi}}^-(m\breve{\varpi}_1))$ contains the element
\[
s_{\Lambda}(t_1^{-1},\ldots,t_p^{-1})(t_1\cdots t_p)^{-k}\otimes \left(u_{\varpi_1-k\delta}^{\otimes m-p}\right)
\]
for all $\Lambda\in \mathcal{P}_p$.

Let $\mathbf{c}_0=(\Lambda,\emptyset,\ldots,\emptyset)\in \overline{\Par}(m\breve{\varpi}_1)$, where $\ell(\Lambda)\le p$.

  By Corollary \ref{BN}, Lemma \ref{t} and Lemma \ref{L-R2}, we have 
  \begin{multline*}
    \Phi_{m\breve{\varpi}_1}(\breve{F}_n^{(p)}\breve{F}_{n-1}^{(p)}\cdots \breve{F}_1^{(p)}\breve{S}_{\mathbf{c}_0}^-\breve{S}_{t_{\xi}}u_{m\breve{\varpi}_1})\\
    \in\left(\sum_{\substack{M\in \mathcal{P}_p\\ N\in \mathcal{P}_{m-p}}}c^{\Lambda}_{M,N}\left(\breve{S}_{w_1}\left(s_M(\breve{z}_{1,1}^{-1},\ldots,\breve{z}_{1,p}^{-1})u_{\breve{\varpi}_1-k\breve{\delta}}^{\otimes p}\right)\right)\otimes \left(s_N(\breve{z}_{1,p+1}^{-1},\ldots,\breve{z}_{1,m}^{-1})u_{\breve{\varpi}_1-k\breve{\delta}}^{\otimes m-p}\right)\right)\\
    + \left(\bigoplus_{\mathfrak{j}>\mathfrak{i}}L_{\mathfrak{j},1}\right).
  \end{multline*}

  Hence, we have
  \begin{align*}
  &\tau_{p,1}^1(\breve{F}_n^{(p)}\breve{F}_{n-1}^{(p)}\cdots \breve{F}_1^{(p)}\breve{S}_{\mathbf{c}_0}^-\breve{S}_{t_{\xi}}u_{m\breve{\varpi}_1})\\
  &=\sum_{\substack{M\in \mathcal{P}_p\\ N\in \mathcal{P}_{m-p}}}c^{\Lambda}_{M,N}(-q)^{-(|N|+k(m-p))b_{1,1}}s_M(t_1^{-1},\ldots,t_p^{-1})(t_1\cdots t_p)^{-k}\otimes s_N(z_1^{-1},\ldots,z_{m-p}^{-1})u_{\varpi_1-k\delta}^{\otimes m-p}\\
  &=\sum_{\substack{M\in \mathcal{P}_p\\ N\in \mathcal{P}_{m-p}}}c^{\Lambda}_{M,N}(-q)^{-(|N|+k(m-p))b_{1,1}}s_M(t_1^{-1},\ldots,t_p^{-1})(t_1\cdots t_p)^{-k}\otimes \left(S_{\mathbf{d}_0}^-u_{\varpi_1-k\delta}^{\otimes m-p}\right),
  \end{align*}
where $\mathbf{d}_0=(N,\emptyset,\ldots,\emptyset)\in \overline{\Par}((m-p)\varpi_1)$. By Lemma \ref{L-R1}, we know $c^{\Lambda}_{\Lambda,\emptyset}=1$, and if $c^{\Lambda}_{M,N}\neq 0$ with $M\neq \Lambda$, then $|M|<|\Lambda|$. Thus, by induction hypothesis, we have 
\[
s_{\Lambda}(t_1^{-1},\ldots,t_p^{-1})(t_1\cdots t_p)^{-k}\otimes \left(u_{\varpi_1-k\delta}^{\otimes m-p}\right)\in \tau_{p,1}^1(M_p\cap V_{t_{\xi}}^-(m\breve{\varpi}_1)),
\]
as desired.

Since $\{s_{\Lambda}(t_1^{-1},\ldots,t_p^{-1})\}_{\Lambda\in \mathcal{P}_p}$ is a $\mathbb{Q}(q)$-basis of $\mathbb{Q}(q)[t_1^{-1},\ldots,t_p^{-1}]^{\mathfrak{S}_p}$, $\tau_{p,1}^1(M_{p,1}\cap V_{t_{\xi}}^-(m\breve{\varpi}_1))$ contains the $U_q^-(\widehat{\mathfrak{sl}}_n)$-submodule
\[
\left((t_1\cdots t_p)^{-k}\mathbb{Q}(q)[t_1^{-1},\ldots,t_p^{-1}]^{\mathfrak{S}_p}\right)\otimes \left(U_q^-(\widehat{\mathfrak{sl}}_n)u_{\varpi_1-k\delta}^{\otimes m-p}\right).
\]

 By Corollary \ref{br} and Corollary \ref{Mgcht}, we have
 \begin{align*}
   \gch \left(M_{p,1}\cap V_{t_{\xi}}^-(m\breve{\varpi}_1)\right)&=q^{-km}\left(\left(\prod_{r=1}^p(1-q^{-r})\right)\left(\prod_{s=1}^{m-p}(1-q^{-s})\right)\right)^{-1}P_{(m-p)\varpi_1}(x;q^{-1},0)\\
   &=q^{-kp}\left(\prod_{r=1}^p(1-q^{-r})\right)^{-1}\gch V_{t_{\zeta}}^-((m-p)\varpi_1)
 \end{align*}
where $\zeta=k\alpha_1^{\vee}$. By Theorem \ref{inj}, there is an injective $U_q(\widehat{\mathfrak{sl}}_n)$-linear morphism $\Phi_{(m-p)\varpi_1}:V((m-p)\varpi_1)\to V(\varpi_1)^{\otimes m-p}$ given by $u_{(m-p)\varpi_1}\mapsto u_{\varpi_1}^{\otimes m-p}$. Since $\Phi_{(m-p)\varpi_1}(S_{t_{\zeta}}u_{(m-p)\varpi_1})=u_{\varpi_1-k\delta}^{\otimes m-p}$, the restriction of $\Phi_{(m-p)\varpi_1}$ gives an isomorphism of $U_q^-(\widehat{\mathfrak{sl}}_n)$-modules
\[
V_{t_{\zeta}}^-((m-p)\varpi_1)\stackrel{\cong}{\longrightarrow} U_q^-(\widehat{\mathfrak{sl}}_n)u_{\varpi_1-k\delta}^{\otimes m-p}.
\]

 Hence, the graded character of of $\left((t_1\cdots t_p)^{-k}\mathbb{Q}(q)[t_1^{-1},\ldots,t_p^{-1}]^{\mathfrak{S}_p}\right)\otimes \left(U_q^-(\widehat{\mathfrak{sl}}_n)u_{\varpi_1-k\delta}^{\otimes m-p}\right)$ is computed as
 \[
 q^{-kp}\sum_{\Lambda\in \mathcal{P}_p}q^{-|\Lambda|}\gch V_{t_{\zeta}}^-((m-p)\varpi_1)=q^{-kp}\left(\prod_{r=1}^p(1-q^{-r})\right)^{-1}\gch V_{t_{\zeta}}^-((m-p)\varpi_1).
 \]

 Thus, by comparing the characters, we see that the restriction of $\tau_{p,1}^1$ gives an isomorphism of $U_q^-(\widehat{\mathfrak{sl}}_n)$-modules
 \begin{align*}
   M_{p,1}\cap V_{t_{\xi}}^-(m\breve{\varpi}_1)&\cong \left((t_1\cdots t_p)^{-k}\mathbb{Q}(q)[t_1^{-1},\ldots,t_p^{-1}]^{\mathfrak{S}_p}\right)\otimes \left(U_q^-(\widehat{\mathfrak{sl}}_n)u_{\varpi_1-k\delta}^{\otimes m-p}\right)\\
   &=\left((t_1\cdots t_p)^{-k}\mathbb{Q}(q)[t_1^{-1},\ldots,t_p^{-1}]^{\mathfrak{S}_p}\right)\otimes \Phi_{(m-p)\varpi_1}\left(V_{t_{\zeta}}^-((m-p)\varpi_1)\right).
 \end{align*}

  By Lemma \ref{dem}, we have
  \[
  M_{p,1}=\bigcup_{k\in \mathbb{Z}}\left(M_{p,1}\cap V_{t_{k\breve{\alpha}_1^{\vee}}}^-(m\breve{\varpi}_1)\right)\quad\text{and}\quad
    V((m-p)\varpi_1)=\bigcup_{k\in \mathbb{Z}}V_{t_{k\alpha_1^{\vee}}}^-((m-p)\varpi_1).
  \]

Therefore, $\tau_{p,1}^1$ induces an isomorphism
\begin{align*}
  M_{p,1}&\cong \mathbb{Q}(q)[t_1^{\pm1},\ldots,t_p^{\pm1}]^{\mathfrak{S}_p}\otimes\Phi_{(m-p)\varpi_1}\left(V((m-p)\varpi_1)\right)\\
  &\cong \mathbb{Q}(q)[t_1^{\pm1},\ldots,t_p^{\pm1}]^{\mathfrak{S}_p}\otimes V((m-p)\varpi_1)
\end{align*}
in this case.

  Next, we assume that $\varepsilon=-1$. We set $\mathfrak{i}^{\prime}=(\overbrace{2,\ldots,2}^{m-p},\overbrace{1,\ldots,1}^p)\in \mathcal{I}_p$. We define an $U_q(\widehat{\mathfrak{sl}}_n)$-homomorphism $\tau_{p,-1}^1$ as the composition
  \[
  \tau_{p,-1}^1:M_{p,-1}\xrightarrow{\Phi_{m\breve{\varpi}_1}\mid_{M_{p,-1}}} L_{p,-1}=L_{\le\mathfrak{i}^{\prime},-1}\xrightarrow{p_{\mathfrak{i}^{\prime},-1}}\mathcal{L}_{\mathfrak{i}^{\prime},-1}\xrightarrow{\Xi_{\mathfrak{i}^{\prime},-1}^{-1}}V(\varpi_1)^{\otimes m-p}\otimes \mathbb{Q}(q)[t^{\pm1}]. 
  \]

  Let $t_{\nu}$ denote the variable in the $\nu$-th factor of $\mathbb{Q}(q)[t^{\pm}]^{\otimes p}$. Then we can identify $\mathbb{Q}(q)[t^{\pm}]^{\otimes p}$ with $\mathbb{Q}(q)[t_1^{\pm1},\ldots,t_p^{\pm1}]$.

Let $k\in \mathbb{Z}$ and we set $\xi=k\breve{\alpha}_1^{\vee}$. We deduce by induction on $|\Lambda|$ that $\tau_{p,-1}^1(M_{p,-1}\cap V_{t_{\xi}}^-(m\breve{\varpi}_1))$ contains the element
\[
\left(u_{\varpi_1-k\delta}^{\otimes m-p}\right)\otimes s_{\Lambda}(t_1^{-1},\ldots,t_p^{-1})(t_1\cdots t_p)^{-k}
\]
for all $\Lambda\in \mathcal{P}_p$.

Let $\mathbf{c}_0=(\Lambda,\emptyset,\ldots,\emptyset)\in \overline{\Par}(m\breve{\varpi}_1)$, where $\ell(\Lambda)\le p$.

  By Corollary \ref{BN}, Lemma \ref{t} and Lemma \ref{L-R2}, we have 
  \begin{multline*}
    \Phi_{m\breve{\varpi}_1}(\breve{F}_n^{(p)}\breve{F}_{n-1}^{(p)}\cdots \breve{F}_1^{(p)}\breve{S}_{\mathbf{c}_0}^-\breve{S}_{t_{\xi}}u_{m\breve{\varpi}_1})\\
    \in \sum_{\substack{M\in \mathcal{P}_{m-p}\\ N\in \mathcal{P}_p}}c^{\Lambda}_{M,N} q^{p(m-p)}\left(s_M(\breve{z}_{1,1}^{-1},\ldots,\breve{z}_{1,m-p}^{-1})u_{\breve{\varpi}_1-k\breve{\delta}}^{\otimes m-p}\right)\otimes \left(\breve{S}_{w_1}\left(s_N(\breve{z}_{1,m-p+1}^{-1},\ldots,\breve{z}_{1,m}^{-1})u_{\breve{\varpi}_1-k\breve{\delta}}^{\otimes p}\right)\right)\\
    + \left(\bigoplus_{\mathfrak{j}<\mathfrak{i}^{\prime}}L_{\mathfrak{j},1}\right).
  \end{multline*}

  Hence, we have
  \begin{align*}
  &\tau_{p,-1}^1(\breve{F}_n^{(p)}\breve{F}_{n-1}^{(p)}\cdots \breve{F}_1^{(p)}\breve{S}_{\mathbf{c}_0}^-\breve{S}_{t_{\xi}}u_{m\breve{\varpi}_1})\\
  &=\sum_{\substack{M\in \mathcal{P}_{m-p}\\ N\in \mathcal{P}_p}}c^{\Lambda}_{M,N}q^{p(m-p)}(-q)^{-(|M|+k(m-p))b_{1,-1}}s_M(z_1^{-1},\ldots,z_{m-p}^{-1})u_{\varpi_1-k\delta}^{\otimes m-p}\otimes s_N(t_1^{-1},\ldots,t_p^{-1})(t_1\cdots t_p)^{-k}\\
  &=\sum_{\substack{M\in \mathcal{P}_{m-p}\\ N\in \mathcal{P}_p}}c^{\Lambda}_{M,N}q^{p(m-p)}(-q)^{-(|M|+k(m-p))b_{1,-1}}\left(S_{\mathbf{d}_0}^-u_{\varpi_1-k\delta}^{\otimes m-p}\right)\otimes s_N(t_1^{-1},\ldots,t_p^{-1})(t_1\cdots t_p)^{-k},
  \end{align*}
where $\mathbf{d}_0=(M,\emptyset,\ldots,\emptyset)\in \overline{\Par}((m-p)\varpi_1)$. By Lemma \ref{L-R1}, we know $c^{\Lambda}_{\emptyset,\Lambda}=1$, and if $c^{\Lambda}_{M,N}\neq 0$ with $N\neq \Lambda$, then $|N|<|\Lambda|$. Thus, by induction hypothesis, we have 
\[
\left(u_{\varpi_1-k\delta}^{\otimes m-p}\right)\otimes s_{\Lambda}(t_1^{-1},\ldots,t_p^{-1})(t_1\cdots t_p)^{-k}\in \tau_{p,-1}^1(M_{p,-1}\cap V_{t_{\xi}}^-(m\breve{\varpi}_1)),
\]
as desired.

Since $\{s_{\Lambda}(t_1^{-1},\ldots,t_p^{-1})\}_{\Lambda\in \mathcal{P}_p}$ is a $\mathbb{Q}(q)$-basis of $\mathbb{Q}(q)[t_1^{-1},\ldots,t_p^{-1}]^{\mathfrak{S}_p}$, $\tau_{p,-1}^1(M_{p,-1}\cap V_{t_{\xi}}^-(m\breve{\varpi}_1))$ contains the $U_q^-(\widehat{\mathfrak{sl}}_n)$-submodule
\[
\left(U_q^-(\widehat{\mathfrak{sl}}_n)u_{\varpi_1-k\delta}^{\otimes m-p}\right)\otimes \left((t_1\cdots t_p)^{-k}\mathbb{Q}(q)[t_1^{-1},\ldots,t_p^{-1}]^{\mathfrak{S}_p}\right).
\]

 By the same argument as in the case $\varepsilon=1$, it can be shown that $\tau_{p,-1}^1$ induces an isomorphism
\begin{align*}
  M_{p,-1}&\cong \Phi_{(m-p)\varpi_1}\left(V((m-p)\varpi_1)\right)\otimes \mathbb{Q}(q)[t_1^{\pm1},\ldots,t_p^{\pm1}]^{\mathfrak{S}_p}\\
  &\cong  V((m-p)\varpi_1)\otimes \mathbb{Q}(q)[t_1^{\pm1},\ldots,t_p^{\pm1}]^{\mathfrak{S}_p},
\end{align*}
as required.
\end{proof}

\begin{corollary}\label{i1}
  \textit{There is an isomorphism of $U_q(\widehat{\mathfrak{sl}}_n)$-modules
  \[
  \Psi_{\varepsilon}^*V(m\breve{\varpi}_1)\cong \bigoplus_{p=0}^m \mathbb{Q}(q)[t_1^{\pm1},\ldots,t_p^{\pm1}]^{\mathfrak{S}_p}\otimes V((m-p)\varpi_1).
  \]
  }
\end{corollary}

\begin{proof}
  This result follows immediately from Proposition \ref{dirsum2} and Theorem \ref{i1.}.
\end{proof}

By the proof of Theorem \ref{i1.}, we obtain the following result.

\begin{corollary}
  \textit{There is an isomorphism of $U_q^-(\widehat{\mathfrak{sl}}_n)$-modules
  \[
  \pushQED{\qed}
  (\Psi_{\varepsilon}^-)^*V_e^-(m\breve{\varpi}_1)\cong \bigoplus_{p=0}^m \left( \mathbb{Q}(q)[t_1^{-1},\ldots,t_p^{-1}]^{\mathfrak{S}_p}\otimes V_e^-((m-p)\varpi_1)\right).\qedhere
\popQED
  \]
 }
\end{corollary}

\subsection{Branching rule for $V(m\breve{\varpi}_i)$: the case $i=n$}

In this subsection, we assume that $\lambda\in \breve{P}_{0,+}$ is of the form $\lambda=m\breve{\varpi}_n$.

\begin{theorem}\label{in.}
  \textit{For $0\le p\le m$, there is an isomorphism of $U_q(\widehat{\mathfrak{sl}}_n)$-modules
  \[
  M_{p,\varepsilon}\cong \mathbb{Q}(q)[t_1^{\pm1},\ldots,t_{m-p}^{\pm1}]^{\mathfrak{S}_{m-p}}\otimes V(p\varpi_{n-1}) \left(\cong V(p\varpi_{n-1})\otimes \mathbb{Q}(q)[t_1^{\pm1},\ldots,t_{m-p}^{\pm1}]^{\mathfrak{S}_{m-p}} \right).
  \]
  }
\end{theorem}

\begin{proof}
  We first assume that $\varepsilon=1$. We set $\mathfrak{i}=(\overbrace{1,\ldots,1}^p,\overbrace{2,\ldots,2}^{m-p})\in \mathcal{I}_p$. We define an $U_q(\widehat{\mathfrak{sl}}_n)$-linear morphism $\tau_{p,1}^n$ as the composition
  \[
  \tau_{p,1}^n:M_{p,1}\xrightarrow{\Phi_{m\breve{\varpi}_n}\mid_{M_{p,1}}} L_{p,1}=L_{\ge\mathfrak{i},1}\xrightarrow{p_{\mathfrak{i},1}}\mathcal{L}_{\mathfrak{i},1}\xrightarrow{\Xi_{\mathfrak{i},1}^{-1}} V(\varpi_{n-1})^{\otimes p}\otimes \mathbb{Q}(q)[t^{\pm1}]^{\otimes m-p}. 
  \]

  Let $t_{\nu}$ denote the variable in the $\nu$-th factor of $\mathbb{Q}(q)[t^{\pm}]^{\otimes p}$. Then we can identify $\mathbb{Q}(q)[t^{\pm}]^{\otimes p}$ with $\mathbb{Q}(q)[t_1^{\pm1},\ldots,t_{m-p}^{\pm1}]$.

Let $k\in \mathbb{Z}$ and we set $\xi=k\breve{\alpha}_n^{\vee}$. We deduce by induction on $|\Lambda|$ that $\tau_{p,1}^n(M_{p,1}\cap V_{t_{\xi}}^-(m\breve{\varpi}_n))$ contains the element
\[
 \left(u_{\varpi_{n-1}-k\delta}^{\otimes p}\right)\otimes s_{\Lambda}(t_1^{-1},\ldots,t_{m-p}^{-1})(t_1\cdots t_{m-p})^{-k}
\]
for all $\Lambda\in \mathcal{P}_{m-p}$.

Let $\mathbf{c}_0=(\emptyset,\ldots,\emptyset,\Lambda)\in \overline{\Par}(m\breve{\varpi}_n)$, where $\ell(\Lambda)\le {m-p}$.

  By Corollary \ref{BN}, Lemma \ref{t} and Lemma \ref{L-R2}, we have 
  \begin{multline*}
    \Phi_{m\breve{\varpi}_n}(\breve{F}_n^{(p)}\breve{S}_{\mathbf{c}_0}^-\breve{S}_{t_{\xi}}u_{m\breve{\varpi}_1})\\
    \in\sum_{\substack{M\in \mathcal{P}_p\\ N\in \mathcal{P}_{m-p}}}c^{\Lambda}_{M,N}\left(\breve{S}_{w_n}\left(s_M(\breve{z}_{1,1}^{-1},\ldots,\breve{z}_{1,p}^{-1})u_{\breve{\varpi}_n-k\breve{\delta}}^{\otimes p}\right)\right)\otimes \left(s_N(\breve{z}_{1,p+1}^{-1},\ldots,\breve{z}_{1,m}^{-1})u_{\breve{\varpi}_n-k\breve{\delta}}^{\otimes m-p}\right)\\
    + \left(\bigoplus_{\mathfrak{j}>\mathfrak{i}}L_{\mathfrak{j},1}\right).
  \end{multline*}

  Hence, we have
  \begin{align*}
  &\tau_{p,1}^n(\breve{F}_n^{(p)}\breve{S}_{\mathbf{c}_0}^-\breve{S}_{t_{\xi}}u_{m\breve{\varpi}_n})\\
  &=\sum_{\substack{M\in \mathcal{P}_p\\ N\in \mathcal{P}_{m-p}}}c^{\Lambda}_{M,N}(-q)^{-(|M|+kp)a_{n,1}} s_M(z_1^{-1},\ldots,z_p^{-1})u_{\varpi_{n-1}-k\delta}^{\otimes p}\otimes s_N(t_1^{-1},\ldots,t_{m-p}^{-1})(t_1\cdots t_{m-p})^{-k}\\
  &=\sum_{\substack{M\in \mathcal{P}_p\\ N\in \mathcal{P}_{m-p}}}c^{\Lambda}_{M,N}(-q)^{-(|M|+kp)a_{n,1}}\left(S_{\mathbf{d}_0}^-u_{\varpi_{n-1}-k\delta}^{\otimes p}\right)\otimes s_N(t_1^{-1},\ldots,t_{m-p}^{-1})(t_1\cdots t_{m-p})^{-k},
  \end{align*}
where $\mathbf{d}_0=(\emptyset,\ldots,\emptyset,M)\in \overline{\Par}(p\varpi_{n-1})$. By Lemma \ref{L-R1}, we know $c^{\Lambda}_{\emptyset,\Lambda}=1$, and if $c^{\Lambda}_{M,N}\neq 0$ with $N\neq \Lambda$, then $|N|<|\Lambda|$. Thus, by induction hypothesis, we have 
\[
\left(u_{\varpi_{n-1}-k\delta}^{\otimes p}\right)\otimes s_{\Lambda}(t_1^{-1},\ldots,t_{m-p}^{-1})(t_1\cdots t_{m-p})^{-k}\in \tau_{p,1}^n(M_{p,1}\cap V_{t_{\xi}}^-(m\breve{\varpi}_n)),
\]
as desired.

Since $\{s_{\Lambda}(t_1^{-1},\ldots,t_{m-p}^{-1})\}_{\Lambda\in \mathcal{P}_{m-p}}$ is a $\mathbb{Q}(q)$-basis of $\mathbb{Q}(q)[t_1^{-1},\ldots,t_{m-p}^{-1}]^{\mathfrak{S}_{m-p}}$, $\tau_{p,1}^n(M_{p,1}\cap V_{t_{\xi}}^-(m\breve{\varpi}_1))$ contains the $U_q^-(\widehat{\mathfrak{sl}}_n)$-submodule
\[
 \left(U_q^-(\widehat{\mathfrak{sl}}_n)u_{\varpi_{n-1}-k\delta}^{\otimes p}\right)\otimes \left((t_1\cdots t_{m-p})^{-k}\mathbb{Q}(q)[t_1^{-1},\ldots,t_{m-p}^{-1}]^{\mathfrak{S}_{m-p}}\right).
\]

 By Corollary \ref{br} and Corollary \ref{Mgcht}, we have
 \begin{align*}
   \gch \left(M_{p,1}\cap V_{t_{\xi}}^-(m\breve{\varpi}_n)\right)&=q^{-km}\left(\left(\prod_{r=1}^p(1-q^{-r})\right)\left(\prod_{s=1}^{m-p}(1-q^{-s})\right)\right)^{-1}P_{p\varpi_{n-1}}(x;q^{-1},0)\\
   &=q^{-k(m-p)}\left(\prod_{s=1}^{m-p}(1-q^{-s})\right)^{-1}\gch V_{t_{\zeta}}^-(p\varpi_{n-1})
 \end{align*}
where $\zeta=k\alpha_{n-1}^{\vee}$. By Theorem \ref{inj}, there is an injective $U_q(\widehat{\mathfrak{sl}}_n)$-linear morphism $\Phi_{p\varpi_{n-1}}:V(p\varpi_{n-1})\to V(\varpi_{n-1})^{\otimes p}$ given by $u_{p\varpi_{n-1}}\mapsto u_{\varpi_{n-1}}^{\otimes p}$. Since $\Phi_{p\varpi_{n-1}}(S_{t_{\zeta}}u_{p\varpi_{n-1}})=u_{\varpi_{n-1}-k\delta}^{\otimes p}$, the restriction of $\Phi_{p\varpi_{n-1}}$ gives an isomorphism of $U_q^-(\widehat{\mathfrak{sl}}_n)$-modules
\[
V_{t_{\zeta}}^-(p\varpi_{n-1})\stackrel{\cong}{\longrightarrow} U_q^-(\widehat{\mathfrak{sl}}_n)u_{\varpi_{n-1}-k\delta}^{\otimes p}.
\]

 Hence, the graded character of $\left(U_q^-(\widehat{\mathfrak{sl}}_n)u_{\varpi_{n-1}-k\delta}^{\otimes p}\right)\otimes \left((t_1\cdots t_{m-p})^{-k}\mathbb{Q}(q)[t_1^{-1},\ldots,t_{m-p}^{-1}]^{\mathfrak{S}_{m-p}}\right)$ is computed as
 \[
 q^{-k(m-p)}\sum_{\Lambda\in \mathcal{P}_{m-p}}q^{-|\Lambda|}\gch V_{t_{\zeta}}^-(p\varpi_{n-1})=q^{-k(m-p)}\left(\prod_{s=1}^{m-p}(1-q^{-s})\right)^{-1}\gch V_{t_{\zeta}}^-(p\varpi_{n-1}).
 \]

 Thus, by comparing the graded characters, we see that the restriction of $\tau_{p,1}^n$ gives an isomorphism of $U_q^-(\widehat{\mathfrak{sl}}_n)$-modules
 \begin{align*}
   M_{p,1}\cap V_{t_{\xi}}^-(m\breve{\varpi}_n)&\cong \left(U_q^-(\widehat{\mathfrak{sl}}_n)u_{\varpi_{n-1}-k\delta}^{\otimes p}\right)\otimes \left((t_1\cdots t_{m-p})^{-k}\mathbb{Q}(q)[t_1^{-1},\ldots,t_{m-p}^{-1}]^{\mathfrak{S}_{m-p}}\right)\\
   &=\Phi_{p\varpi_{n-1}}\left((V_{t_{\zeta}}^-(p\varpi_{n-1})\right)\otimes \left((t_1\cdots t_{m-p})^{-k}\mathbb{Q}(q)[t_1^{-1},\ldots,t_{m-p}^{-1}]^{\mathfrak{S}_{m-p}}\right).
 \end{align*}

  By Lemma \ref{dem}, we have
  \[
  M_{p,1}=\bigcup_{k\in \mathbb{Z}}\left(M_{p,1}\cap V_{t_{k\breve{\alpha}_1^{\vee}}}^-(m\breve{\varpi}_n)\right)\quad\text{and}\quad
    V((m-p)\varpi_{n-1})=\bigcup_{k\in \mathbb{Z}}V_{t_{k\alpha_1^{\vee}}}^-(p\varpi_{n-1}).
  \]

Therefore, $\tau_{p,1}^n$ induces an isomorphism
\begin{align*}
  M_{p,1}&\cong \Phi_{p\varpi_{n-1}}\left(V(p\varpi_{n-1})\right)\otimes \mathbb{Q}(q)[t_1^{\pm1},\ldots,t_{m-p}^{\pm1}]^{\mathfrak{S}_{m-p}}\\
  &\cong V(p\varpi_{n-1})\otimes \mathbb{Q}(q)[t_1^{\pm1},\ldots,t_{m-p}^{\pm1}]^{\mathfrak{S}_{m-p}}
\end{align*}
in this case.

  Next, we assume that $\varepsilon=-1$. We set $\mathfrak{i}^{\prime}=(\overbrace{2,\ldots,2}^{m-p},\overbrace{1,\ldots,1}^p)\in \mathcal{I}_p$. We define an $U_q(\widehat{\mathfrak{sl}}_n)$-homomorphism $\tau_{p,-1}^n$ as the composition
  \[
  \tau_{p,-1}^n:M_{p,-1}\xrightarrow{\Phi_{m\breve{\varpi}_1}\mid_{M_{p,-1}}} L_{p,-1}=L_{\le\mathfrak{i}^{\prime},-1}\xrightarrow{p_{\mathfrak{i}^{\prime},-1}}\mathcal{L}_{\mathfrak{i}^{\prime},-1}\xrightarrow{\Xi_{\mathfrak{i}^{\prime},-1}^{-1}} \mathbb{Q}(q)[t^{\pm1}]^{\otimes m-p}\otimes V(\varpi_{n-1})^{\otimes p}. 
  \]

  Let $t_{\nu}$ denote the variable in the $\nu$-th factor of $\mathbb{Q}(q)[t^{\pm}]^{\otimes p}$. Then we can identify $\mathbb{Q}(q)[t^{\pm}]^{\otimes p}$ with $\mathbb{Q}(q)[t_1^{\pm1},\ldots,t_p^{\pm1}]$.

Let $k\in \mathbb{Z}$ and we set $\xi=k\breve{\alpha}_1^{\vee}$. We deduce by induction on $|\Lambda|$ that $\tau_{p,-1}^n(M_{p,-1}\cap V_{t_{\xi}}^-(m\breve{\varpi}_n))$ contains the element
\[
s_{\Lambda}(t_1^{-1},\ldots,t_p^{-1})(t_1\cdots t_{m-p})^{-k}\otimes \left(u_{\varpi_{n-1}-k\delta}^{\otimes p}\right)
\]
for all $\Lambda\in \mathcal{P}_{m-p}$.

Let $\mathbf{c}_0=(\emptyset,\ldots,\emptyset,\Lambda)\in \overline{\Par}(m\breve{\varpi}_n)$, where $\ell(\Lambda)\le {m-p}$.

  By Corollary \ref{BN}, Lemma \ref{t} and Lemma \ref{L-R2}, we have 
  \begin{multline*}
    \Phi_{m\breve{\varpi}_n}(\breve{F}_n^{(p)}\breve{S}_{\mathbf{c}_0}^-\breve{S}_{t_{\xi}}u_{m\breve{\varpi}_n})\\
    \in \sum_{\substack{M\in \mathcal{P}_{m-p}\\ N\in \mathcal{P}_p}}c^{\Lambda}_{M,N} q^{p(m-p)}\left(s_M(\breve{z}_{1,1}^{-1},\ldots,\breve{z}_{1,m-p}^{-1})u_{\breve{\varpi}_n-k\breve{\delta}}^{\otimes m-p}\right)\otimes \left(\breve{S}_{w_n}\left(s_N(\breve{z}_{1,m-p+1}^{-1},\ldots,\breve{z}_{1,m}^{-1})u_{\breve{\varpi}_n-k\breve{\delta}}^{\otimes p}\right)\right)\\
    + \left(\bigoplus_{\mathfrak{j}<\mathfrak{i}^{\prime}}L_{\mathfrak{j},1}\right).
  \end{multline*}

  Hence, we have
  \begin{align*}
  &\tau_{p,-1}^n(\breve{F}_n^{(p)}\breve{S}_{\mathbf{c}_0}^-\breve{S}_{t_{\xi}}u_{m\breve{\varpi}_n})\\
  &=\sum_{\substack{M\in \mathcal{P}_{m-p}\\ N\in \mathcal{P}_p}}c^{\Lambda}_{M,N}q^{p(m-p)}(-q)^{-(|M|+k(m-p))a_{n,-1}}s_M(t_1^{-1},\ldots,t_{m-p}^{-1})(t_1\cdots t_{m-p})^{-k}\otimes s_N(z_1^{-1},\ldots,z_p^{-1})u_{\varpi_n-k\delta}^{\otimes p}\\
  &=\sum_{\substack{M\in \mathcal{P}_{m-p}\\ N\in \mathcal{P}_p}}c^{\Lambda}_{M,N}q^{p(m-p)}(-q)^{-(|M|+k(m-p))a_{n,-1}}s_M(t_1^{-1},\ldots,t_{m-p}^{-1})(t_1\cdots t_{m-p})^{-k}\otimes \left(S_{\mathbf{d}_0}^-u_{\varpi_1-k\delta}^{\otimes p}\right),
  \end{align*}
where $\mathbf{d}_0=(\emptyset,\ldots,\emptyset,N)\in \overline{\Par}(p\varpi_{n-1})$. By Lemma \ref{L-R1}, we know $c^{\Lambda}_{\Lambda,\emptyset}=1$, and if $c^{\Lambda}_{M,N}\neq 0$ with $M\neq \Lambda$, then $|M|<|\Lambda|$. Thus, by induction hypothesis, we have 
\[
s_{\Lambda}(t_1^{-1},\ldots,t_p^{-1})(t_1\cdots t_{m-p})^{-k}\otimes \left(u_{\varpi_{n-1}-k\delta}^{\otimes p}\right)\in \tau_{p,-1}^n(M_{p,-1}\cap V_{t_{\xi}}^-(m\breve{\varpi}_1)),
\]
as desired.

Since $\{s_{\Lambda}(t_1^{-1},\ldots,t_{m-p}^{-1})\}_{\Lambda\in \mathcal{P}_{m-p}}$ is a $\mathbb{Q}(q)$-basis of $\mathbb{Q}(q)[t_1^{-1},\ldots,t_{m-p}^{-1}]^{\mathfrak{S}_{m-p}}$, $\tau_{p,-1}^n(M_{p,-1}\cap V_{t_{\xi}}^-(m\breve{\varpi}_n))$ contains the $U_q^-(\widehat{\mathfrak{sl}}_n)$-submodule
\[
\left((t_1\cdots t_{m-p})^{-k}\mathbb{Q}(q)[t_1^{-1},\ldots,t_{m-p}^{-1}]^{\mathfrak{S}_{m-p}}\right)\otimes \left(U_q^-(\widehat{\mathfrak{sl}}_n)u_{\varpi_{n-1}-k\delta}^{\otimes p}\right).
\]

 By the same argument as in the case $\varepsilon=1$, it can be shown that $\tau_{p,-1}^1$ induces an isomorphism
\begin{align*}
  M_{p,-1}&\cong \mathbb{Q}(q)[t_1^{\pm1},\ldots,t_{m-p}^{\pm1}]^{\mathfrak{S}_{m-p}}\otimes \Phi_{p\varpi_{n-1}}\left(V(p\varpi_{n-1})\right)\\
  &\cong  \mathbb{Q}(q)[t_1^{\pm1},\ldots,t_{m-p}^{\pm1}]^{\mathfrak{S}_{m-p}}\otimes V(p\varpi_{n-1}),
\end{align*}
as required.
\end{proof}

\begin{corollary}\label{in}
  \textit{There is an isomorphism of $U_q(\widehat{\mathfrak{sl}}_n)$-modules
  \[
  \Psi_{\varepsilon}^*V(m\breve{\varpi}_n)\cong \bigoplus_{p=0}^m \mathbb{Q}(q)[t_1^{\pm1},\ldots,t_{m-p}^{\pm1}]^{\mathfrak{S}_{m-p}}\otimes V(p\varpi_{n-1}).
  \]
  }
\end{corollary}

\begin{proof}
   This result follows immediately from Proposition \ref{dirsum2} and Theorem \ref{in.}.
\end{proof}

By the proof of Theorem \ref{in.}, we obtain the following result.

\begin{corollary}
  \textit{There is an isomorphism of $U_q^-(\widehat{\mathfrak{sl}}_n)$-modules
  \[
  \pushQED{\qed}
  (\Psi_{\varepsilon}^-)^*V_e^-(m\breve{\varpi}_n)\cong \bigoplus_{p=0}^m \left(\mathbb{Q}(q)[t_1^{-1},\ldots,t_{m-p}^{-1}]^{\mathfrak{S}_{m-p}}\otimes V_e^-(p\varpi_{n-1})\right).\qedhere
\popQED
 \]
 }
\end{corollary}

\section{Acknowledgment}

The author would like to thank his supervisor Syu Kato for his helpful advice and continuous encouragement.

\end{document}